\documentclass[11pt,a4paper]{article}

\usepackage{amscd,amsmath}
\usepackage{amsfonts}
\usepackage{amssymb,amsthm}
\usepackage{longtable}
\usepackage{graphicx}
\usepackage{hyperref}
\usepackage{fancyvrb}

\usepackage{url}
\usepackage{fancyheadings}

\theoremstyle{definition}
\newtheorem{Thm}{Theorem}[section]
\newtheorem{Def}{Definition}[section]
\newtheorem{Lem}[Thm]{Lemma}
\newtheorem{Cor}[Thm]{Corollary}
\newtheorem{prop}[Thm]{Proposition}

\newtheorem{Rem}{Remark}[section]

\newtheorem*{Thm 4.3.}{Theorem 4.3}
\newtheorem*{prop 4.4.}{Proposition 4.4}
\newtheorem*{prop 4.5.}{Proposition 4.5}
\newtheorem*{Lem 4.6.}{Lemma 4.6}
\newtheorem*{Lem 4.7.}{Lemma 4.7}
\newtheorem*{Lem 4.8.}{Lemma 4.8}
\newtheorem*{Lem 4.9.}{Lemma 4.9}
\newtheorem*{prop 4.10.}{Proposition 4.10}
\newtheorem*{prop 4.11.}{Proposition 4.11}

\begin{document}

%%%%%%%%%%%%%%%%%%%%%%%%%%%%%%%%%%%%%%%%%%%%%%%%%%%%%%%%%%%%%%%%%
%%%%%%%%%%%%%%%%%%%%%%% Title
%%%%%%%%%%%%%%%%%%%%%%%%%%%%%%%%%%%%%%%%%%%%%%%%%%%%%%%%%%%%%%%%%

\title{Fiberwise convexity of Hill's lunar problem}
\author{\bf{Junyoung Lee}
\\ sink21@snu.ac.kr \\ Department of Mathematics and Research Institute of Mathematics \\ Seoul National University}
\date{}
\maketitle

%%%%%%%%%%%%%%%%%%%%%%%%%%%%%%%%%%%%%%%%%%%%%%%%%%%%%%%%%%%%%%%%%
%%%%%%%%%%%%%%%%%%%%%%% Preliminaries
%%%%%%%%%%%%%%%%%%%%%%%%%%%%%%%%%%%%%%%%%%%%%%%%%%%%%%%%%%%%%%%%%

\begin{abstract}
In this paper, we prove the fiberwise convexity of the regularized Hill's lunar problem below the critical energy level. This allows us to see Hill's lunar problem of any energy level below the critical value as the Legendre transformation of a geodesic problem on $S^2$ with a family of Finsler metrics. Therefore the compactified energy hypersurfaces below the critical energy level have the unique tight contact structure on $\mathbb{R} P^3$. Also one can apply the systolic inequality of Finsler geometry to the regularized Hill's lunar problem.
\end{abstract}

\section{Introduction}
\label{intro}

Studying the motion of the moon has been a challenging problem for a long time. If we consider only the sun, the earth and the moon, this problem is the three body problem. One can agree that the three body problem is one of the hardest problems in classical mechanics. For this reason, many researchers have studied this with some restrictions which depend on the situation of the problem. One problem with reasonable and practical restrictions is the (circular planar) restricted three body problem. The restricted three body problem is obtained by assuming that two primary particles $P_1, P_2$ follow Keplerian circular motion and one massless particle $S$ does not influence these primaries. Namely, the masses of the particles have the relation $M_1, M_2>>m$ where $M_1, M_2$ are the masses of the two primaries $P_1, P_2$ respectively and $m$ is the mass of $S$. One can study the motion of the moon in this set-up. However the lunar theory is a limit case of the restricted three body problem since the sun is much heavier than the others and the distance between the sun and the earth is much longer than the distance between the earth and the moon.\footnote{Precisely, ${M_{Sun} \over M_{Earth}} \sim 333000$, ${M_{Earth} \over M_{Moon}} \sim 81.3$ and ${|P_{Sun}-P_{Earth}| \over |P_{Earth}-P_{Moon}|} \sim 388$}
One suggestive formulation for this situation was given by Hill in \cite{Hill}. Hill introduced this lunar problem in 1878 to study the stability of the orbit of the moon. This can be obtained by taking the limit for $\mu:={M_2 \over {M_1 + M_2}}$ in the restricted three body problem. If we take only $\mu \rightarrow 0$ on the restricted three body problem, then we get the so-called rotating Kepler problem. This has played an important role  as an ingredient to understand the restricted three body problem. On the other hand, Hill derived an approximation of the three body problem by considering the previously stated feature of the Sun-Earth-Moon. Before Hill the Sun-Earth-Moon system was regarded as a system of two uncoupled Kepler problems, namely Sun-Earth and Earth-Moon. This turns out to be a poor approximation. Hill introduced a simple model that treats the entire system by zooming in on the earth. In modern language, Hill's idea can be understood by taking a blow up of the coordinates near the earth to the power ${1 \over 3}$ of $\mu$ when one takes $\mu \rightarrow 0$. We will explain this procedure in section \ref{R3BP}. After Hill gave a new formulation for the lunar theory, many researchers have used Hill's lunar problem to get accurate motion of the moon. 

An important feature on studying Hamiltonian systems is the existence of integrals of the system. Integrals of given Hamiltonian system make this problem easier. The extreme case is a completely integrable system. A Hamiltonian system is called completely integrable if it possesses the maximal number of Poisson commuting independent integrals. If the dimension of the phase space is $2n$, then the maximal possible number of Poisson commuting independent integrals is $n$. In Hill's lunar problem case the integrability is equivalent to the existence of a second integral, which is independent of the first integral. The non-integrability of Hill's lunar problem has been determined by some authors.  Meletlidou, Ichtiaroglou and Winterberg in \cite{MIW} proved the analytic non-integrability of Hill's lunar problem. Morales-Ruiz, Sim\'o and Simon gave an algebraic proof of meromorphic non-integrability in \cite{MSS}. Recently, Llibre and Roberto in \cite{LR} discussed the $C^1$ integrability. With these results about non-integrability, Sim$\acute{o}$ and Stuchi in their numerical research \cite{SS} pointed out the importance of 'numerical methods guided by the geometry of dynamical systems' by observing the chaotic behavior in the Levi-Civita regularization of Hill's lunar problem. We show in this paper that the Hamiltonian flow of Hill's lunar problem can be interpreted as a Finsler flow. This will provide a geometric feature of Hill's lunar problem.

One difficulty in the study of this problem comes from collision. Namely, this problem has a singularity at the origin. However, two body collision can be regularized. One way to regularize this problem is Moser regularization. Moser introduced this regularization for the Kepler problem in \cite{Moser}. In this paper, he tells us that the Hamiltonian flow of the Kepler problem can be interpreted as a geodesic flow on the 2-sphere endowed with its standard metric by interchanging the roles of position and momentum. We will discuss this relation in section \ref{Moser reg}. If one replaces the standard metric by a Finsler metric, then this idea can be applied to other problems which admit two body collisions. To get a Finsler metric one needs fiberwise convexity. One recent result using this is given in \cite{CFvK}. They prove that the rotating Kepler problem is fiberwise convex and so can be regarded as the Legendre transformation of the 2-sphere endowed with a Finsler metric. As in the rotating Kepler problem, one can ask whether Hill's lunar problem has also this property or not. The main theorem of this paper is the following.

\begin{Thm}
\label{main theorem}
The bounded components of the regularized Hill's lunar problem are fiberwise convex for the energy level below the critical value.
\end{Thm}

To understand the meaning of fiberwise convexity below the critical value, we need to see the Hamiltonian of Hill's lunar problem.
$$H_{HLP} : \mathbb{R}^2 \times (\mathbb{R}^2- \{(0, 0) \}) \rightarrow \mathbb{R},$$
$$H_{HLP} (q, p)={1\over 2}|p|^2-{1 \over |q|}-q_1^2+{1 \over 2}q_2^2+p_1q_2-p_2q_1$$
Here $q$ is the position variable and $p$ is the momentum variable. This Hamiltonian has one critical value. We can introduce the effective potential
$$U(q_1, q_2):=-{1 \over \sqrt{q_1^2+q_2^2}}-{3\over2}q_1^2$$
in order to see the critical points easily. In fact, we can write the Hamiltonian
$$H_{HLP}(q, p)={1 \over 2}((p_1+q_2)^2+(p_2-q_1)^2)+U(q_1, q_2)$$
using the effective potential. Since the other term is of degree 2, the critical points of $H_{HLP}$ correspond to the critical points of $U$. It means the correspondence
$$\pi (Crit (H_{HLP}))=Crit(U)$$
where $\pi$ is the projection to the $q$-coordinate. We can compute the critical points and the critical value of the effective potential $U$
$$Crit(U)=(\pm 3^{-1 \over 3}, 0), \quad U(\pm 3^{-1 \over 3}, 0)=-{{3^{4 \over 3}} \over 2}=:-c_0.$$
Then one can get the critical point of $H_{HLP}$
$$Crit(H_{HLP})=\{ (3^{-1 \over 3}, 0, 0, 3^{-1 \over 3}), (-3^{-1 \over 3}, 0, 0, -3^{-1 \over 3}) \}  \in \mathbb{R}^2 (q) \times \mathbb{R}^2 (p)$$
and the critical value
$$H_{HLP}((\pm 3^{-1 \over 3}, 0, 0, \pm 3^{-1 \over 3}))=-c_0.$$
We are interested in the energy level below this critical value $-c_0$ in order to prove the Theorem \ref{main theorem}. Thus we will assume for the energy level that
$$-c<-c_0 \iff c>c_0={{3^{4 \over 3}} \over 2}$$
throughout this paper. With these $c$, we define the Hamiltonians $K_c$ 
$$K_c(q, p)=|q|(H_{HLP}(q, p)+c)$$
for the regularization of this problem. It will be proven in section \ref{interpretation} that $\pi(K^{-1}_c(0))$, the projection of the zero level set of $K_c$ to the $q$-coordinate, has one bounded component and two unbounded components. Let us denote by $\Sigma_c$ the component of $K^{-1}_c(0)$ which projects to the bounded component. Namely $\Sigma_c$ is a connected component of $K^{-1}_c(0)$ and $\pi(\Sigma_c)$ is bounded. By the symplectomorphism $(q, p) \rightarrow (p, -q)$, we can think of $p$ as a position variable and of $q$ as a momentum variable. In this situation $p$ can be regarded as a value in $\mathbb{C}$ and so $\Sigma_c \subset T^* \mathbb{C}$. We can regard $T^* \mathbb{C}$ as a subset of $T^* S^2$ by the one point compactification of $\mathbb{C}$. Then we can think of $\Sigma_c$ as a subset of $T^*S^2$ using the stereographic projection. In this situation Theorem \ref{main theorem} can be rephrased as follows.

\begin{itemize}
\item \hypertarget{F1}{$\bold{(F1)}$} \label{F1} The closure $\overline { { \Sigma  }_{ c } }$ of $\Sigma_c$ in $T^* S^2$ is a submanifold of $T^* S^2$ for all $c>c_0$.

\item \hypertarget{F2}{$\bold{(F2)}$} For any fixed $p \in S^2$ and $c>c_0$, $\overline { { \Sigma  }_{ c }} \cap T^*_p S^2$ bounds a convex region which contains the origin in the cotangent plane $T^*_p S^2$.
\end{itemize}
In short, the connected component $\Sigma_{c}$ of the energy hypersurface $H_{HLP}^{-1}(-c)$ with energy $-c<-c_0$ can be symplectically embedded into $T^{*} S^2$ as a fiberwise convex hypersurface after compactification. By proving the above statements \hyperlink{F1}{$\bold{(F1)}$}, \hyperlink{F2}{$\bold{(F2)}$}, we can show that the regularized Hill's lunar problem can be regarded as Legendre dual to a geodesic problem in $S^2$ with Finsler metric. With this definition of fiberwise convexity, we have one obvious Corollary of Theorem \ref{main theorem}. 

\begin{Cor}
The bounded component of the regularized Hill's lunar problem has a contact structure for the energy level below the critical value. Moreover, this contact structure is the unique tight contact structure on $\mathbb{R}P^3$ up to contact isotopy.
\end{Cor}

It is clear that fiberwise convexity implies fiberwise starshapedness with respect to the origin for all $T^*_p S^2$. This implies that the restriction of the Liouville 1-form on $T^* S^2$ to $\overline{\Sigma_c}$ gives a contact form. Since the unit sphere bundle $S^* S^2$ of $T^* S^2$ with respect to any Finsler metric is diffeomorphic to $\mathbb{R}P^3$ and the contact structure given by the Liouville 1-form is strongly fillable, we have the strongly fillable contact structure on $\overline{\Sigma_c}$ which has diffeomorphism type $\mathbb{R}P^3$. From the criterion due to Eliashberg and Gromov in \cite{Eli1} and \cite{Gro}, any symplectically fillable contact 3-manifold is tight. Moreover $\mathbb{R}P^3$ admits a unique tight contact structure up to isotopy by the result of Eliashberg \cite{Eli2}.

We will prove Theorem \ref{main theorem} in section \ref{preparation and strategy} and \ref{proof} . As one can see in section \ref{proof}, by the complexity of computation, it seems hard to take further computations about the corresponding geodesic problem in spite of our knowledge of the existence of a corresponding Finsler metric. However, this correspondence itself gives us information about closed characteristics.

\begin{Cor}
The Conley-Zehnder indices of the closed characteristics of the regularized Hill's lunar problem below the critical energy level is nonnegative.
\end{Cor}

The Conley-Zehnder indices of the closed characteristics of the Hamiltonian flow including collision orbits coincide with the Morse indices of the corresponding geodesics. Therefore, we know that all closed characteristics of the regularized Hill's lunar problem have nonnegative Conley-Zehnder indices. Of course, it is well-known that the Conley-Zehnder indices of closed characteristics of the unregularized Hill's lunar problem are nonnegative. Indeed, the Hamiltonian of the unregularized Hill's lunar problem is a magnetic Hamiltonian and the Conley-Zehnder indices are nonnegative for any magnetic Hamiltonian. However, this result is new for collision orbits. Moreover, thanks to the result in \cite{ABT} using the systolic inequality, we can ensure the existence of a closed characteristic whose action is less than a volume related constant. We refer the following Theorem.

\begin{Thm}
(\'Alvarez Paiva-Balacheff-Tsanev) There exists a constant $k(S^2)>0$ such that every fiberwise convex hypersurface $\Sigma \subset T^* S^2$ bounding a volume $V$ carries a closed characteristic whose action is less than $k(S^2) \sqrt{V}$. Here $k(S^2)$ does not depend on $\Sigma$.
\end{Thm}

The volume $V$ in here is the Holmes-Thompson volume that is the symplectic volume with the canonical symplectic form in the cotangent bundle. This coincides with the contact volume of $\overline{\Sigma_c}$ with the canonical contact form $\alpha:=\lambda | _{\overline{\Sigma_c}}$ where $\lambda$ is the Liouville one form of $T^* S^2$ by Stokes' Theorem. Moreover, it is known that the constant $k(S^2)$ is less than $\sqrt{3\pi}10^8$ and this constant is independent of $\Sigma$. In \cite{ABT}, they explained the beautiful relationship between contact and systolic geometry which allows to extend the result  of Gromov and Croke in systolic inequality on Riemannian manifolds. As an application of this Theorem, we can formulate the following Corollary.

\begin{Cor}
The regularized Hill's lunar problem has at least one periodic orbit, possibly a collision orbit, whose action is less than $k(S^2) \sqrt{Vol(\overline{\Sigma_c})}$.
\end{Cor}

One interesting question is what we can get from systolic geometry for our practical Hamiltonian problems which have contact structures. In particular, one can ask how the systolic capacity changes under a perturbation of the Hamiltonian, because Hill's lunar problem is a limit case of the restricted three body problem. Hopefully, if one can answer this question, then one might get insight into the restricted three body problem using this information and method in the proof.

\begin{Def}
A contact form on a compact 3-manifold is called dynamically convex, if its contractible periodic Reeb orbits of positive period have Conley-Zehnder index greater than 2 with respect to any filling disk.
\end{Def}

As a goal on the Hill's lunar problem, I want to mention dynamical convexity. A motivation of this paper is showing the dynamical convexity of the double cover of Hill's lunar problem. It is known that the rotating Kepler problem is dynamically convex and the restricted three body problem is also dynamically convex for some mass ratio and energy. It is still unknown if the Hill's lunar problem is dynamically convex. Because we know an energy hypersurface of the regularized Hill's lunar problem is tight $\mathbb{R}P^3$, its double cover has the contact structure of the unique tight structure on $S^3$. If the double cover of Hill's lunar problem is dynamically convex, then this double cover allows a disk-like global surface of section for the Hamiltonian vector field. This will simplify the problem dramatically. We will see the related result in section \ref{R3BP}.
\\ \\
$\bold{Acknowledgements : }$ I thank Urs Frauenfelder and Otto van Koert for encouragements and discussions. I am also grateful to the colleagues in Augsburg university for their helps that make me adapt well to the new surroundings in Augsburg. I wish to express my thanks to the referee of this article for many helpful comments. This research is supported by DFG-CI 45/6-1: Algebraic Structures on Symplectic Homology and Their Applications.

\section{Prerequisite}
\label{pre}

It is based on Moser regularization in \cite{Moser} to understand why the fiberwise convexity is helpful to study Hill's lunar problem in Hamiltonian dynamics. Moser regularization tells us the planar Kepler problem can be compactified to the geodesic problem on the standard 2-sphere. This argument can be improved for the case of a fiberwise convex hypersurface which corresponds to the geodesic problem on a 2-sphere with Finsler metric. On the other hand, we need to know how Hill's lunar problem can be derived from the restricted three body problem. Since Hill's lunar problem is a limit of the restricted three body problem, they have relationships with each other. For example, Meyer and Schmidt in \cite{MS} show that any non-degenerate periodic solution of Hill's lunar problem whose period is not a multiple of $2\pi$ can be lifted to the three body problem. This could be proven by observing the derivation of Hill's lunar problem from the three body problem. Thus, understanding the relation between Hill's lunar problem and the restricted three body problem will be helpful to get some ideas for the restricted three body problem from the result of Hill's lunar problem. For example, fiberwise convexity of the restricted three body problem is still open. We will review Moser regularization of the Kepler problem in section \ref{Moser reg} and the restricted three body problem with its relation to Hill's lunar problem in section \ref{R3BP}.

\subsection{Kepler problem and Moser regularization}
\label{Moser reg}

The differential equation of the Kepler problem is given by
$${d^2q \over dt^2}=-{q\over |q|^3}$$
using some normalization. Therefore the potential function $V : (\mathbb{R}^2)^* \rightarrow \mathbb{R}$ is $V(q)=-{1 \over |q|}$ and this induces the Hamiltonian of the Kepler problem as the total energy.
$$H : \mathbb{R}^2 \times (\mathbb{R}^2)^* \rightarrow \mathbb{R}, \quad H(p,q)={1\over 2}|p|^2-{1 \over |q|}$$
However, this is not so practical to analyze by geometric methods because this Hamiltonian has a singularity at $q=0$. One preferred way to remove this singularity is Moser regularization. For constant $c \in \mathbb{R}$, we define the Hamiltonian
$$K_c (p,q):=|q|(H(p, q)+c)$$
Then we can easily see that $K_c$ has no singularity and $H^{-1}(-c)=K_c^{-1}(0)$. However, these two Hamiltonian dynamics on this level set arising from $H$ and $K_c$ are not exactly same, but have the same Hamiltonian flow up to time reparametrization. We introduce a new time parameter $s=\int {dt \over |q|}$ for $K_c$ to make these equivalent problems.

We briefly explain Moser's paper \cite{Moser} which shows that this regularized Kepler problem is equivalent to the geodesic problem on standard 2-sphere. We consider the energy level $-c=-{1\over2}$. 
$$K_{1\over2}(p,q)={1\over2}|q|(|p|^2+1)-1 \implies K_{1\over2}^{-1}(0)=\{(p,q) \in \mathbb{R}^2 \times (\mathbb{R}^2)^* | {1 \over 2}(|p|^2+1)|q|=1 \}$$
Other energy levels can be proven analogously by simple rescaling. Note that $(p, q) \mapsto (q,-p)$ is symplectic and in our case this seems like interchanging the role of $p$ and $q$. We can see that ${1 \over 2}(|p|^2+1)|q|=1$ comes from energy hypersurface $F(x, y)=1$ of $T^* S^2$ where $F(x, y)={1 \over 2}|y|_{round}^2$ the Hamiltonian for free particle via the stereographic projection. The flow of the Hamiltonian for a free particle is the geodesic flow in general. Therefore the Hamiltonian flow of the Kepler problem corresponds to the geodesic problem on $S^2$ with the round metric. Precisely, Moser proved the following Theorem for the Kepler problem in $n$-dimensional space.

\begin{Thm}[Moser]
For a negative energy $-c<0$, the energy hypersurface $H^{-1}(-c)$ can be mapped bijectively into the unit tangent bundle of $S^n - \{ \textrm{north pole} \}$. Furthermore, the flow defined by the Kepler problem is mapped into the geodesic flow on the punctured sphere $S^n-\{ \textrm{north pole} \}$.
\end{Thm}

The above argument can be extended to the fiberwise convex case. In the case of Kepler problem case, amazingly, the trajectory of $q$ for fixed position $p \in S^2$ is exactly unit circle in the cotangent space $T^*_p S^2$ with the round metric. Thus, if an energy hypersurface of a Hamiltonian system is the unit cotangent bundle $S^*_g S^2$ of a metric $g$, then the Hamiltonian system on this hypersurface corresponds to the problem of geodesic on $S^2$ with the metric $g$. Moreover, if a problem has a level set trajectory of $q$ which encircles the convex region containing the origin for any $p \in S^2$, then this will be the geodesic problem on $S^2$ with Finsler metric by defining the position of $q$ in $T^*_p S^2$ to be the unit length. Therefore, we set up \hyperlink{F1}{$\bold{(F1)}$} to determine whether the energy hypersurface of Hill's lunar problem can be seen as a submanifold in $T^* S^2$ after regularization and changing the role of $q$ and $p$. Moreover, we set up \hyperlink{F2}{$\bold{(F2)}$} to determine if the hypersurface can define a Finsler metric on $T^* S^2$.

\subsection{The restricted three body problem, the rotating Kepler problem and Hill's lunar problem}
\label{R3BP}

We can derive the time-independent Hamiltonian of the restricted three body problem by introducing rotating coordinates with unit angular velocity. It is important to understand how one can derive Hill's lunar problem from the restricted three body problem not only to decide which problem can be effective with Hill's setup, but also to get intuitions to know closed characteristics of the restricted three body problem from Hill's lunar problem.

First, we explain the derivation of the Hamiltonian for the restricted three body problem briefly. We denote the masses $M_1, M_2$ of two primaries $P_1, P_2$. We define $\mu={M_2\over M_1+M_2}$ and assume that two primaries have the following motion.
$$P_1(t)=(-\mu \cos t, -\mu \sin t), \quad  P_2(t)=((1-\mu) \cos t, (1-\mu) \sin t)$$
We are interested in the motion of a massless particle $S(t) \in \mathbb{R}^2-\{ P_1(t), P_2(t)\}$ and we can easily derive the Hamiltonian
$$H^i (t,q^i, p^i)={1 \over 2}|p^i|^2-{\mu \over |q^i -P_2(t)|}-{1-\mu \over |q^i -P_1(t)|}$$
for the restricted three body problem in inertial system. We put index $i$ to emphasize that this Hamiltonian is taken in the inertial system. Note that $H^i$ is time-dependent. Now we consider the rotating system to make this Hamiltonian become time-independent. We express the positions of $P_1$ and $P_2$
$$A_1:=(-\mu, 0), A_2:=(1-\mu, 0) \implies P_1(t)=R_t A_1, P_2(t)=R_t A_2$$
by the rotation $R_t=\begin{pmatrix} \cos t & -\sin t \\ \sin t & \cos t\end{pmatrix}$ of two fixed points $A_1$ and $A_2$, respectively.
We define the rotation 
$$
\Psi_t:=R_t \oplus R_t=\begin{pmatrix}  \cos t & -\sin t  & 0 & 0 \\ \sin t & \cos t & 0 & 0 \\ 0 & 0 & \cos t & -\sin t \\ 0 & 0 & \sin t & \cos t  \end{pmatrix}
$$
on $T^* \mathbb{R}^2=\mathbb{R}^2 \times \mathbb{R}^2$. We can find the following Theorem in many books, for example, see \cite{Kim}.

\begin{Thm}
Let $H^r$ be the Hamiltonian in a rotating system which rotate by $\Psi_t$. Then $H^r=H^i \circ \phi_K^t - K$ where $K=q_1p_2-q_2p_1$ and $\phi_K^t$ are Hamiltonian diffeomorphisms generated by $K$. In particular $H^r$ is autonomous.
\end{Thm}

We have a time-independent Hamiltonian
$$H_1^{\mu} : \mathbb{R}^2 \times (\mathbb{R}^2- \{A_1, A_2 \}) \rightarrow \mathbb{R},$$
$$H_1^{\mu}(p, q)={1\over 2}|p|^2-{\mu \over |q-A_2|}-{1-\mu \over |q-A_1|}+p_1q_2-p_2q_1$$
for the restricted three body problem of mass ratio $\mu$ in the rotating coordinates. Also, we can get this  equivalent Hamiltonian
$$H_2^{\mu} : \mathbb{R}^2 \times (\mathbb{R}^2- \{(0, 0), (1, 0) \}) \rightarrow \mathbb{R},$$
$$H_2^{\mu}(p, q)={1\over 2}|p|^2-{\mu \over |q-(1,0)|}-{1-\mu \over |q|}+p_1q_2-p_2q_1-\mu p_2$$
by translating in $q$-coordinates.

Many important studies of global properties of the restricted three body problem have been done in the study of this time-independent Hamiltonian using symplectic geometry. Recently there was a remarkable result \cite{AFFHvK} which tells us the existence of a disk-like global surfaces of section for Hamiltonian vector field in the restricted three body problem for $\mu \in (\mu_0(c), 1)$ where $-c$ is the energy below the first Lagrange value. We mention the definition of the disk-like global surface of section.

\begin{Def}
Let $\Sigma$ be a smooth 3-manifold with a nowhere vanishing vector field $X$. A global disk-like surface of section for $X$ consists of a embedded closed disk $\mathcal{D} \in \Sigma$ having the following properties:
\begin{enumerate}
\item The boundary $\partial \mathcal{D}$ is a periodic orbit, called the spanning orbit.
\item The interior $int(\mathcal{D})$ of the disk is a smooth submanifold of $\Sigma$ and is transversal to the flow.
\item Every orbit except the spanning orbit intersects $int(\mathcal{D})$ in forward and backward time.
\end{enumerate}
\end{Def}

One can easily recognize from this definition that a global disk-like surface of section reduces the study of the dynamics on a 3-manifold to the study of the return map on the disk. The result about the existence of a global disk-like surface of section for the restricted three body problem in \cite{AFFHvK} based on the result of Hofer, Wysocki and Zehnder \cite{HWZ} which uses a pseudoholomorphic curve theory for an energy hypersurface in $\mathbb{R}^4$. In \cite{HWZ}, they prove that strict convexity of an energy hypersurface implies dynamical convexity and dynamical convexity implies the existence of global disk-like surfaces of section. As an application of this theory, in \cite{AFFHvK}, they found pairs of $(\mu, c)$ where the energy hypersurfaces $K_{\mu, c}^{-1}(0)$ of the regularized Hamiltonian for such pairs bound a strictly convex region. We define the regularized Hamiltonian $K_{\mu, c}$ for the precise statement in \cite{AFFHvK}. We introduce the Levi-Civita coordinates $(u, v)$ using a 2:1 symplectic map, up to a constant factor, $q=2v^2, p={u \over \overline{v}}$ and apply to $H_2^{\mu}$.
$$K_{\mu, c}(u, v):=|v|^2(H_{2}^{\mu}(u, v)+c)={1\over 2}|u|^2+2|v|^2<u,iv>-\mu Im(uv)-{1-\mu \over 2}-{\mu |v|^2 \over |2v^2-1|}+c|v|^2$$
The energy hypersurface $K_{\mu, c}^{-1}(0)$ coincides with the energy hypersurface $(H_2^{\mu})^{-1}(-c)$. This implies the Hamiltonian flows $\phi_{K_{\mu, c}}^t$, $\phi_{H_2^{\mu}}^t$ are same up to time reparametrization. We state the result in \cite{AFFHvK}.

\begin{Thm}[Albers-Fish-Frauenfelder-Hofer-van Koert] Given $c>{3 \over 2}$, there exists $\mu_0=\mu_0 (c) \in [0,1)$ such that for all $\mu_0<\mu<1$ there exists a disk-like global surface of section for the hypersurface $K_{\mu, c}^{-1}(0)$ with its Reeb vector field.
\end{Thm}

One can ask the same question for the limit problems of the restricted three body problem. In \cite{AFFvK}, they give the answer for the rotating Kepler problem. The rotating Kepler problem is dynamically convex after Levi-Civita regularization for each energy below the critical value of the Jacobi energy. Thus, the energy hypersurfaces will have global surfaces of section for the Hamiltonian vector field. Since they also proved the failure of strict convexity in \cite{AFFvK}, the proof is entirely different from the proof in \cite{AFFHvK}. Instead of using theory in \cite{HWZ}, they observed every periodic orbit and computed the Conley-Zehnder indices. One main idea in the computation of the indices comes from the fiberwise convexity of the rotating Kepler problem. They regarded periodic orbits of the rotating Kepler problem as periodic Finsler geodesics and used local stability of Morse homology. On the other hand, we do not know the existence of a global surface of section for Hill's lunar problem.

Now let us explain briefly the derivation of Hill's lunar problem. We will borrow the simple derivation from \cite{MHO}. Apply a coordinate transformation on the Hamiltonian $H_1^{\mu}$ by translating $p, q$-coordinates in the following way.
$$ q_1 \rightarrow q_1+1-\mu, \quad q_2 \rightarrow q_2, \quad p_1 \rightarrow p_1, \quad p_2 \rightarrow p_2+1-\mu.
$$
We have the Hamiltonian
$$H_3^{\mu}(p, q)={1\over 2}|p|^2-{\mu \over |q|}-{1-\mu \over |q+(1, 0)|}+p_1q_2-p_2q_1-(1-\mu)q_1$$
up to constant. By Newton's binomial series $(1+x)^{-1 \over 2}=1-{1 \over 2}x+{3 \over 8}x^2+\cdots$, we get the expansion
$$-{{1-\mu} \over \sqrt{(q_1+1)^2+q_2^2}}=-(1-\mu)(1-q_1+q_1^2-{1\over2}q_2^2+ \cdots)$$
and we apply this on $H_3^{\mu}$
$$H_3^{\mu}(p, q)={1\over 2}|p|^2-{\mu \over |q|}+p_1q_2-p_2q_1-(1-\mu)(q_1^2-{1\over2}q_2^2+\cdots).$$
Consider the scaling $q \rightarrow \mu^{1\over3}q, p \rightarrow \mu^{1 \over 3}p$ which is symplectic with conformal coefficient $\mu^{-2\over3}$. We multiply this factor
$$\mu^{-2 \over 3} H_3^{\mu}(\mu^{1\over3}p, \mu^{1\over3}q)=H_{HLP}(p, q)+O(\mu^{1\over3}),$$
then we obtain the Hamiltonian 
$$H_{HLP}(p, q)={1\over2}|p|^2-{1 \over |q|}+p_1 q_2 -p_2 q_1-q_1^2+{1 \over 2}q_2^2$$
for Hill's lunar problem by taking $\mu \rightarrow 0$.

\section{Interpretation of Theorem \ref{main theorem}.}
\label{interpretation}

From now on, we will concentrate on Hill's lunar problem. Hence we will denote simply by $H$ the Hamiltonian $H_{HLP}$ of Hill's lunar problem. We showed that $H$ has unique critical value $-c_0:=-{3^{4 \over 3} \over 2}$ in section \ref{intro}. We want to show fiberwise convexity for all $-c<-c_0$. We define the Hamiltonian
$$K_c(q, p):=|q|(H(q, p)+c)$$
for the regularization of this problem. The Hamiltonian flow of $K_c$ on $K^{-1}_c (0)$ coincides with the Hamiltonian flow of $H$ on $H^{-1}(-c)$. Moreover, $K_c$ has no singularity. Observe the structure of  $K^{-1}_c (0)$.
\begin{eqnarray*}
(q, p) \in K^{-1}_c (0) &\iff& {1 \over 2}((p_1 +q_2)^2 +(p_2 -q_1)^2)={1\over \sqrt{q_1^2+q_2^2}}+{3 \over 2}q_1^2-c
\\ &\iff& \begin{cases} {1\over \sqrt{q_1^2+q_2^2}}+{3 \over 2}q_1^2=b \ge c \\  (p_1+q_2)^2 +(p_2-q_1)^2=2(b-c) \end{cases}
\end{eqnarray*}
We introduce polar coordinates $q_1=r \cos \theta, q_2=r \sin \theta$, then \begin{math} {1\over \sqrt{q_1^2+q_2^2}}+{3 \over 2}q_1^2=b \iff {3 \over 2}\cos^2 \theta r^3+1=br \end{math}. We can see the structure of the set \begin{math} \{(q_1, q_2) \in \mathbb{R}^2 | {1\over \sqrt{q_1^2+q_2^2}}+{3 \over 2}q_1^2=b\} \end{math} from the following Lemma.

\begin{Lem}
\label{Lemma 3.1}
For $b>c_0={3^{4\over3} \over 2}$, the polar equation ${3\over2}(\cos^2 \theta) r^3 +1=b r$ consists of one bounded closed curve and two unbounded curves. Moreover, if we denote the bounded component of ${3\over2}(\cos^2 \theta) r^3 +1=b r$ by $\sigma_b$, then $\sigma_b$ is contained in the inside of $\sigma_c$ for any $b>c>c_0$.
\end{Lem}

\begin{proof}
Let $f_{b, \theta} (r)={3\over2}(\cos^2 \theta) r^3 -b r+1$ be polynomial for fixed $b$ and $\theta$.  For $\theta \ne {\pi \over 2}, {3\pi \over 2}$, we observe the values $f_{b, \theta}(-\infty)=-\infty$, $f_{b, \theta}(0)=1>0$, $f_{b, \theta}(\sqrt{{2b \over 9 \cos^2 \theta}})=1-{2b \over 3}\sqrt{{2b \over 9 \cos^2 \theta}}<1-{2c_0 \over 3}\sqrt{{2c_0 \over 9}}=0$ and $f_{b, \theta}(+\infty)=+\infty$. This implies $f_{b, \theta}$ has one negative zero and two different positive zeros for any fixed $b$ and for any $\theta \ne {\pi \over 2}, {3\pi \over 2}$. Moreover, one can see that the larger positive zero goes to infinity as the term $\cos^2 \theta$ goes to $0$. Since the zeros vary continuously with $\theta$, the smaller positive zero goes to ${1\over b}$ as the term $\cos^2 \theta$ goes to $0$. This proves that ${3\over2}\cos^2 \theta r^3 +1=br$ consists of one closed curve and two unbounded curves. 

We define the positive smaller zero $r_{b, \theta}$ of the degree 3 polynomial $f_{b, \theta}$ for $\theta \ne {\pi \over 2}, {3\pi \over 2}$ and $r_{b, {\pi \over 2}}=r_{b, {3\pi \over 2}}={1 \over b}$. Then we have $f_{b,\theta} (r_{b,\theta})=0$ and $r_{b,\theta}<\sqrt{{2b \over 9 \cos^2 \theta}}$ by above computation. We differentiate $f_{b,\theta}(r_{b,\theta})=0$ with respect to $b$
$${9 \over 2}(\cos^2 \theta) r_{b,\theta}^2 {dr_{b, \theta} \over db}=r_b+b{dr_{b, \theta} \over db} \implies {dr_{b,\theta} \over db}={r_{b,\theta} \over {{9 \over 2}(\cos^2 \theta) r_{b,\theta}^2-b}}.$$
Since $r_{b,\theta}<\sqrt{{2b \over 9 \cos^2 \theta}}$, ${dr_{b, \theta} \over db}<0$. This implies the bounded component is getting smaller as $b$ increases. This proves the Lemma.
\end{proof}

We can see that the projection to $q$-coordinate $\pi(K^{-1}_c(0))$ of $K^{-1}_c(0)$ consists of one bounded component and two unbounded components for $c>c_0$ and the bounded component of $\pi(K^{-1}_c(0))$ is enclosed by the closed curve $\sigma_c$. We will focus on the case where $q$ is in this bounded component and so denote the bounded component of $\pi(K^{-1}_c(0))$ by $\mathfrak{R}_c$. We define the subset 
$$\Sigma_c=\{(q, p) \in K^{-1}_c (0) | q \in \mathfrak{R}_c \}$$
of $K^{-1}_c (0)$. As in Moser regularization, we regard $p$ as a position variable and $q$ as a momentum variable by using the symplectomorphism $(q, p) \mapsto (p, -q)$. Then we can regard $\Sigma_c$ as a subset of $T^* \mathbb{C}$ where $p \in \mathbb{C}$ is a position variable and $q$ is a momentum variable. We will prove that there exist $(p, q) \in \Sigma_c$ for any $p \in \mathbb{C}$ and such $q$'s form a closed curve in $T^*_p \mathbb{C}$ in the following Lemma. 

\begin{Lem}
\label{Lemma 3.2}
For any fixed $c>c_0$, the projection $pr: \Sigma_c \rightarrow \mathbb{C}, \quad pr(p, q)=p$ is surjective. Moreover, the fiber $pr^{-1}(p)$ at $p$ is a closed curve that encloses the origin for any $p \in \mathbb{C}$.
\end{Lem}
\begin{proof}
We can give an easy geometric interpretation for the subset
\begin{eqnarray*}
\Sigma_c&=&\{ (q, p) \in K^{-1}_c (0) | q \in \mathfrak{R}_c \}
\\ &=&\{ (q, p) \in T^* \mathbb{C} | q_1^2-{1 \over 2}q_2^2+{1 \over |q|}=p_1q_2-p_2q_1+{1 \over 2}|p|^2+c, q \in \mathfrak{R}_c \}
\end{eqnarray*}
of $T^* \mathbb{C}$. If we fix the variable $p$, then the set $\{ q \in T_{p}^*\mathbb{C} | q_1^2-{1 \over 2}q_2^2+{1 \over |q|}=p_1q_2-p_2q_1+{1 \over 2}|p|^2+c\}$ can be seen as the intersection of the graphs of functions $f(q_1, q_2)=q_1^2-{1 \over 2}q_2^2+{1 \over |q|}$ and $g_{p, c}(q_1, q_2)=p_1q_2-p_2q_1+{1 \over 2}|p|^2+c$. Note that the function $g_{p, c}$ of $q$ is a linear function for any fixed $p, c$ and so its graph is a plane. We have $f(q)>g_{p, c}(q)$ for $q$ with sufficiently small $|q|$. For fixed $q_2$, we also get $f(q)>g_{p, c}(q)$ when $q_1 \rightarrow \pm \infty$. On the other hand, we have the inequality
\begin{eqnarray*}
& & g_{p, c}(\pm 3^{-1 \over 3}, q_2)-f(\pm 3^{-1 \over 3}, q_2)
\\ &=& {1 \over 2}(p_1+q_2)^2+{1 \over 2}p_2^2 \mp 3^{-1 \over 3}  p_2-3^{-2 \over 3}-{1 \over (q_2^2+3^{-2 \over 3})^{1\over 2}}+c
\\ &>&{1 \over 2}(p_1+q_2)^2+{1 \over 2}(p_2 \mp 3^{-1 \over 3})^2+{3^{4 \over 3} \over 2}-3^{-2 \over 3}-{3^{-2 \over 3} \over 2}-3^{1 \over 3}
\\ &=& {1 \over 2}(p_1+q_2)^2+{1 \over 2}(p_2 \mp 3^{-1 \over 3})^2 \ge 0
\end{eqnarray*}
Thus $g_{p, c}>f$ along the lines $q_1=3^{-1 \over 3}$ for any $p, c$. Thus the intersection consists of two unbounded components lying in the regions of $q_1>3^{-1 \over 3}$ and $q_1<-3^{-1 \over 3}$, respectively, and one bounded component lying in $-3^{-1 \over 3}<q_1<3^{-1 \over 3}$. Since the plane does not pass the critical points, that component is a one dimensional submanifold and the topology is same for any $p, c$. Thus we know this $q_1$-bounded component is a closed curve by observing the case where $c$ is sufficiently large. Also, we know this closed curve encloses the origin because $f>g_{p, c}$ near the origin for any $p, c$. This proves Lemma 3.2.
\end{proof}

Using Lemma \ref{Lemma 3.2}, we can interpret  $pr : \Sigma_c \rightarrow \mathbb{C}$ as a fiber subbundle of $T^* \mathbb{C}$ with fiber a circle. By one point compactification, we can think of $\mathbb{C} \subset S^2$ and also $\Sigma_c \subset T^* \mathbb{C} \subset T^* S^2$ using stereographic projection as in Moser regularization. If every fiber in the cotangent plane bounds a convex region, which contains the origin, then we can think of $\Sigma_c$ as a unit cotangent bundle of some Finsler metric. As a result, its Hamiltonian flow can be interpreted as the geodesic flow on $S^2$ for a Finsler metric. We formulated two statements \hyperlink{F1}{$\bold{(F1)}$}, \hyperlink{F2}{$\bold{(F2)}$} which are equivalent to Theorem \ref{main theorem}. For \hyperlink{F1}{$\bold{(F1)}$}, we have to show that the closure $\overline{\Sigma_c}$ is a submanifold of $T^* S^2$. The problem of being a submanifold can occur only at the north pole. That is, we have to check if it has a unique limit in $T^* S^2$ when $|p|$ goes to infinity. This can be verified by looking at the fiber when $|p| \rightarrow \infty$. Let us use the notations in Lemma \ref{Lemma 3.2}. Since $q$ lies on the bounded set, $g_{p, c}(q)$ goes to infinity when $|p| \rightarrow \infty$ for any $c$. Thus, if $q_{\nu}$ is a sequence in the bounded region satisfying $f(q_{\nu})=g_{p_{\nu}, c}(q_{\nu})$ for $p_{\nu} \rightarrow \infty$, then $q_{\nu} \rightarrow 0$. Therefore the equation $f(q)=g_{p, c}(q)$ converges to the equation ${1 \over |q|}={1 \over 2}|p|^2+c$ which is the equation of the Kepler problem and so the limit at the north pole in any direction will correspond to the circle with radius ${1 \over \sqrt{-2c}}$ of the round metric. Therefore the closure $\overline{\Sigma _c}$ in $T^* S^2$ is a subbundle over $S^2$ and this proves \hyperlink{F1}{$\bold{(F1)}$}. Moreover, we have convex fiber at the north pole. Thus, from now on we can regard $p \in S^2$ as an element of $\mathbb{C} \cong \mathbb{R}^2$ when we discuss \hyperlink{F2}{$\bold{(F2)}$}, because we have already proved \hyperlink{F2}{$\bold{(F2)}$} at the north pole.

We investigate the region that $q$ can lie on. We will call this region Hill's region and will denote it by $\mathfrak{R}$. By Lemma \ref{Lemma 3.1} and \ref{Lemma 3.2}, we get the region
\begin{eqnarray*}
\mathfrak{R}&:=&\bigcup _{c>c_0} {\mathfrak{R}_c}=\bigcup _{c>c_0}{\pi(H^{-1}(-c))^b}
\\ &=&\{(q_1, q_2) \in \mathbb{R}^2 | {1\over \sqrt{q_1^2+q_2^2}}+{3 \over 2}q_1^2>c_0, |q_1|<3^{-1 \over 3}, |q_2|<2 \cdot 3^{-4 \over 3}\}
\end{eqnarray*}
where $X^b$ means the bounded component of $X$. It is illustrated as the bounded region enclosed by two curves in Figure \ref{fig 1}.

\begin{figure}
\centering
\includegraphics[]{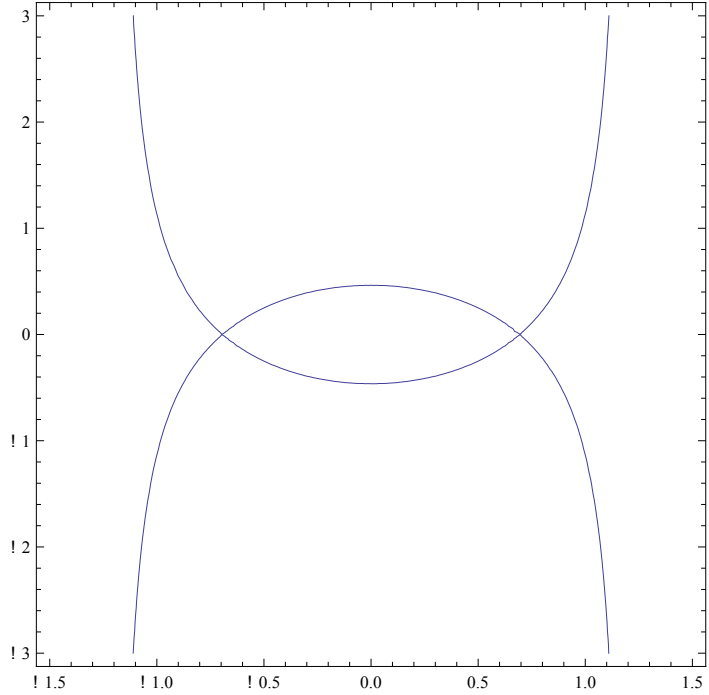}
\caption{Hill's region: The bounded part}
\label{fig 1}
\end{figure}

Since the coordinate change of a cotangent bundle induced from a coordinate change on its base manifold is linear on each cotangent space and a linear map preserves the convexity, $\overline{\Sigma_c} \cap T^*_{p} (S^2)$ bounds a strictly convex region if and only if $\Sigma_c \cap T^*_{p} \mathbb{C}$ bounds a strictly convex region under the stereographic projection for every $p \in S^2 \backslash \{ \textrm{north pole} \}$. Because we have already proved the strict convexity of the fiber at the north pole, we can reduce the problem to $\Sigma_c$ and we can regard $\Sigma_c$ as a fiber bundle over $\mathbb{C}$ for a fixed energy level $c>c_0$. For any $p \in \mathbb{C}$, the fiber $F_{c, p} =\{q \in \mathbb{R}^2 | (p, q) \in \Sigma_c \}$ of this bundle is a closed curve. Then we want to show that this fiber bounds a strictly convex region which contains the origin. The fact that this encloses the origin is already proved in Lemma \ref{Lemma 3.2}.

If we define $K_{c, p}: \mathbb{R}^2 \rightarrow \mathbb{R}$ by $K_{c, p}(q):=K_c(q, p)$, then we need to prove that the bounded component of $K_{c, p}^{-1}(0)$ bounds a strictly convex region for every fixed $p \in \mathbb{R}^2$ and $c>c_0$. Since $K_c$ and $H_c$ have the same 0 energy hypersurface, this is equivalent to prove that the bounded component of $H_{c, p}^{-1}(0)$ bounds a strictly convex region for every $p \in \mathbb{R}^2$ and $c>c_0$ where $H_{c, p}(q)=H_c (q, p)$. If the Hessian $HessH_{c, p}(q)$ of $H_{c, p}$ is positive definite, then its level curve bounds strictly convex region. However, this is not true and in fact we will see $HessH_{c, p}(q)$ has one positive eigenvalue and one negative eigenvalue. Thus we have to consider the tangential Hessian in order to check the strict convexity of the level curves. The tangential Hessian means simply the restriction of Hessian to the tangent space of the energy level set. Because level sets of $H_{c, p}$ are curves, it suffices to show the positivity only for one nonzero tangent vector of $H_{c, p}^{-1}$ and this tangent vector can be obtained simply by rotating the gradient vector $90$ degrees. We can state \hyperlink{F2}{$\bold{(F2)}$} numerically by the following Theorem.
\begin{Thm}
\label{Thm 3.3}
Suppose $q$ lies on a fiber at $p$ in $\Sigma_c$, equivalently $q \in \mathfrak{R} \cap H_{c, p}^{-1}(0)$ for $c>c_0$ and $p \in \mathbb{R}^2$. Then the inequality 
$$(J \triangledown H_{c, p}(q))^t HessH_{c, p}(q)(J \triangledown H_{c, p}(q))>0$$
holds where $J=\begin{pmatrix} 0 & 1 \\ -1 & 0 \end{pmatrix}$ is ${\pi \over 2}$ rotation.
\end{Thm}
Therefore we can reduce our problem into an inequality problem with some constraints. Moreover, we do not need to care about the fiber bundle structure. Namely, it suffices to show that the inequality $((J \triangledown H_{c, p}(q)))^t HessH_{c, p}(q)(J \triangledown H_{c, p}(q)) >0$ for all possible $(q, p)$ instead of seeing the bounded component of $H_{c, p}^{-1}(0)$ for fixed $p$.
We devote the remaining part of this paper to the proof of Theorem \ref{Thm 3.3}.

\begin{Rem}
\label{Rem 3.1}
There are symmetries of Hill's lunar problem. The reflections
$$R_1 : (q_1, q_2, p_1, p_2) \longmapsto (-q_1, q_2, p_1, -p_2), \quad R_2 : (q_1, q_2, p_1, p_2) \longmapsto (q_1, -q_2, -p_1, p_2)$$
are anti-symplectic. $H_{HLP}$ is invariant under these maps, namely $H_{HLP}(R_{i}(q, p))=H_{HLP}(q, p)$ for all $(q, p) \in \mathbb{R}^4$ and $i \in \{1, 2\}$. These will allow us to concentrate only on the first quadrant of $q$-coordinate.
\end{Rem}

\section{Preparation and Strategy}
\label{preparation and strategy}

The proof of Theorem \ref{Thm 3.3} consists of many complicated computations and notations. In section \ref{preparation}, we will introduce the necessary notations and derive Theorem \ref{Thm 4.3} which is stronger than Theorem \ref{Thm 3.3}. We apply many elementary methods to prove Theorem \ref{Thm 4.3}. The list of Propositions and Lemmas will be given and we will explain their relations and meanings in section \ref{strategy}. As one can immediately see in Theorem \ref{Thm 3.3}, the dimension of parameters for this problem is 4. Through computations, we try to reduce this dimension by building lower bounds or finding the subset where the minimum is attained.  When we achieve the reduction to 1 dimensional problem, we will provide computer plots of the graph to determine if the final term is positive or negative. The computer plots will be rigorously verified with a computer program in Appendix \ref{pf of fig}. The hardest part of this proof occurs near the critical point because the tangential Hessian goes to $0$ near the critical point. We will use the blow up coordinates to overcome this problem.

\subsection{Preparation for the proof of Theorem \ref{Thm 3.3}.}
\label{preparation}

We do not need to prove $p=0$ case separately, because this case will be covered by the general case, we will prove this case in order to introduce notations and to help understanding.

We compute the gradient and Hessian
$$\triangledown H_{c, 0}(q)=
\begin{pmatrix} -2q_1+{q_1 \over |q|^3} 
\\ q_2+{q_2 \over |q|^3} 
\end{pmatrix}, \quad
HessH_{c, 0}(q)={1 \over |q|^5} 
\begin{pmatrix} -2|q|^5+|q|^2-3q_1^2 & -3q_1 q_2 
\\ -3q_1 q_2 & |q|^5+|q|^2-3q_2^2 
\end{pmatrix}$$
of $H_{c,0}(q)=-q_1^2+{1\over2}q_2^2-{1\over|q|}+c$ when $p=0$. As we discussed before, we will see the tangential Hessian and so we need a tangent vector of the level curve. For the notational convenience, we define $v(q) \in T_{q}H_{c, 0}^{-1}(0)$ and $\mathcal{H}(q)$ for $q \in H_{c, 0}^{-1}(0)$ as follows.
$$v(q):=J \triangledown H_{c, 0}(q)=
\begin{pmatrix} q_2+{q_2 \over |q|^3} 
\\ 2q_1-{q_1 \over |q|^3} 
\end{pmatrix} 
\textrm{  where  } J=\begin{pmatrix} 0 & 1 \\ -1 & 0 \end{pmatrix}$$
$$\mathcal{H}(q):=Hess H_{c,0}(q)=
{1 \over |q|^5} \begin{pmatrix} -2|q|^5+|q|^2-3q_1^2 & -3q_1 q_2 
\\ -3q_1 q_2 & |q|^5+|q|^2-3q_2^2 
\end{pmatrix}$$
Then we can express the tangential Hessian $(J \triangledown H_{c, p}(q))^t HessH_{c, p}(q)(J \triangledown H_{c, p}(q))$ as a function
\begin{eqnarray*}
v(q)^t \mathcal{H}(q) v(q)&=&{1\over |q|^{11}}\Big[q_2^2(-2|q|^5+|q|^2-3q_1^2)(1+|q|^3)^2-6q_1^2 q_2^2 (1+|q|^3)(2|q|^3-1)
\\ & &+q_1^2 (|q|^5 +|q|^2 -3q_2^2)(2|q|^3 -1)^2 \Big]
\end{eqnarray*}
of variable $q$. As we discussed in section \ref{interpretation}, the curves $K_{c, 0}^{-1}(0)$ bound strictly convex domains if and only if the inequalities $v(q)^t \mathcal{H}(q) v(q)>0$ hold for all $q \in H_{c,p}^{-1}(0)$. Therefore, we have to show the following 'Warm-up Lemma'  in order to prove the case $p=0$.

\begin{Lem}[Warm-up Lemma]
\label{Lemma 4.1}
Suppose $q \in \mathfrak{R} \cap H_{c,0}^{-1}(0)$ for $c>c_0$. Then the tangential Hessian of $H_{c, 0}$ at $q$ is positive definite, namely the inequality $v(q)^t \mathcal{H}(q) v(q)>0$ holds for every $q \in \mathfrak{R} \cap H_{c,0}^{-1}(0)$ and for every $c>c_0$.
\end{Lem}

\begin{proof}
If we take $q \in H_{c,0}^{-1}(0)$, then $q$ satisfies the equation
$$q_1^2-{1\over2}q_2^2+{1\over|q|}=c>c_0.$$
This implies the inequality
$$|q|^2+{1\over|q|}>c_0 \iff |q|^3-c_0 |q|+1>0$$
for $|q|$. We have that $|q|$ is less than the smallest positive zero, say $\alpha$, of the polynomial $x^3-{3^{4 \over 3} \over 2}x+1$ and so $|q|<\alpha<0.54$. Thus it suffices to prove the function $v(q)^t \mathcal{H}(q) v(q)$ is positive for every $|q|<0.54$. We have the expression of  $v(q)^t \mathcal{H}(q) v(q)$ 
$$v(q)^t \mathcal{H}(q) v(q)={1\over |q|^7}-{3q_1^2 \over |q|^6}-{27q_1^2 q_2^2 \over |q|^5}-{3q_2^2 \over |q|^3}+4q_1^2 -2q_2^2$$
in terms of $q$. Since the following inequalities 
$${3q_1^2 \over |q|^6}+{3q_2^2 \over |q|^3} \le {3q_1^2+3q_2^2 \over |q|^6}={3 \over |q|^4}, \quad
{27q_1^2 q_2^2 \over |q|^5} \le {27 \over 4}{1\over |q|}$$
hold for all $|q|<0.54$. We get the following estimate.
$$v(q)^t \mathcal{H}(q) v(q) \ge {1\over |q|^7}-{3\over |q|^4}-{27\over4}{1\over |q|}-2|q|^2$$
As we can see in the graph of $y={1 \over x^7}-{3 \over x^4}-{27 \over 4}{1 \over x}-2x^2$ in Figure \ref{fig 2}, we have ${1 \over x^7}-{3 \over x^4}-{27 \over 4}{1 \over x}-2x^2>0$ for all $x \in (0, 0.54)$. Therefore we have $v(q)^t \mathcal{H}(q) v(q) > 0$ for all $|q|<0.54$ and this proves Lemma \ref{Lemma 4.1}.
\end{proof}

\begin{figure}
\centering
\includegraphics[]{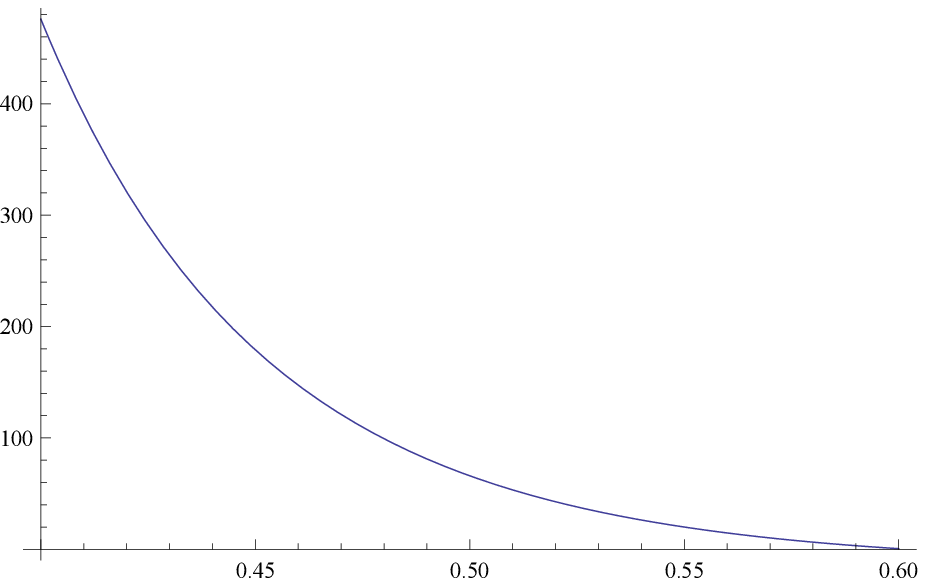}
\caption{Graph of $f_2(x)={1 \over x^7}-{3 \over x^4}-{27 \over 4}{1 \over x}-2x^2$ shows that it is positive on [0, 0.54].}
\label{fig 2}
\end{figure}

Now we consider general $p \in \mathbb{R}^2$. We recall the function $H_{c, p}: \mathbb{R}^2 \rightarrow \mathbb{R}$
$$H_{c,p}(q)={1\over2}|p|^2+p^t Jq-q_1^2+{1\over2}q_2^2-{1\over|q|}+c$$
defined for each $p \in \mathbb{R}^2$ and $c>c_0$. We calculate the gradient, tangent vector and Hessian of $H_{c, p}$.
$$\triangledown H_{c,p}(q)=\begin{pmatrix}-2q_1+{q_1 \over |q|^3}-p_2 \\ q_2+{q_2 \over |q|^3}+p_1 \end{pmatrix}, \quad J\triangledown H_{c,p}(q)=\begin{pmatrix} q_2+{q_2 \over |q|^3}+p_1 \\ 2q_1-{q_1 \over |q|^3}+p_2, \end{pmatrix}=v(q)+p$$
$$Hess H_{c, p}(q)=Hess H_{c, 0}(q)={1 \over |q|^5} \begin{pmatrix} -2|q|^5+|q|^2-3q_1^2 & -3q_1 q_2 \\ -3q_1 q_2 & |q|^5+|q|^2-3q_2^2 \end{pmatrix}=\mathcal{H}(q)$$
for every $q \in H_{c, p}^{-1}(0)$. We can express the tangential Hessian 
$$(J\triangledown H_{c,p}(q))^t Hess H_{c, p}(q) (J\triangledown H_{c,p}(q))=(v(q)+p)^t \mathcal{H}(q) (v(q)+p)$$
using $v(q)$ and $\mathcal{H}(q)$. We can rewrite Theorem \ref{Thm 3.3} with this notations.

\begin{Thm}
\label{Thm 4.2}
Suppose that $q \in \mathfrak{R} \cap H_{c, p}^{-1}(0)$ for $p \in \mathbb{R}^2$ and $c>c_0$.
Then the inequality $(v(q)+p)^t \mathcal{H}(q) (v(q)+p)>0$ holds.
\end{Thm}

It is hard to see that the numerical relation of $p, q$ and $c$ in Theorem \ref{Thm 4.2}. In particular, it is difficult to describe the trajectory of $q$ for a fixed $p$ and for some $c>c_0$. However, the corresponding $p$ to a fixed $q \in \mathfrak{R}$ form a disk with center $(-q_2, q_1)$. From the following equivalences
$$q \in \mathfrak{R} \cap H_{c, p}^{-1}(0) \textrm{ for some } c>c_0 \iff (q, p) \in H_c^{-1}(0) \textrm{ for some } c>c_0$$
$$\iff \begin{cases} {1\over \sqrt{q_1^2+q_2^2}}+{3 \over 2}q_1^2=b>c_0 \textrm{ and } \\  (p_1+q_2)^2 +(p_2-q_1)^2<2(b-c_0) \end{cases},$$
we have the set
$$\{ p \in \mathbb{R}^2 | q \in \mathfrak{R} \cap H_{c, p}^{-1}(0) \textrm{ for some } c>c_0 \} 
=\{p \in \mathbb{R}^2 | (p_1+q_2)^2 +(p_2-q_1)^2<2({1\over \sqrt{q_1^2+q_2^2}}+{3 \over 2}q_1^2-c_0) \}$$
of $p$ for a fixed $q$ with a simple inequality. We introduce new variables $w(q), s$ obtained by translations. If we set $s:=p+Jq$, then we can simplify the equation as follows.
$${1\over2}|p|^2+p^t Jq-q_1^2+{1\over2}q_2^2 -{1\over |q|}+c=0 \iff |s|^2=3q_1^2+{2\over|q|}-2c.$$
This induces the following equivalent condition
$$|s|^2 < 3q_1^2+{2\over|q|}-2c_0 \iff q \in  H_{c, -Jq+s}^{-1}(0) \textrm{ for some } c>c_0$$
for being a point of trajectory. With this substitution, we define the vector $w(q)=v(q)-Jq$ and we have that
$$v(q)+p=v(q)-Jq+s=\begin{pmatrix} {q_2 \over |q|^3} \\ 3q_1-{q_1 \over |q|^3} \end{pmatrix}+s=:w(q)+s,$$
$$(v(q)+p)^t \mathcal{H}(q) (v(q)+p)=(w(q)+s)^t \mathcal{H} (w(q)+s)$$
where $w(q)=\begin{pmatrix} {q_2 \over |q|^3} \\ 3q_1-{q_1 \over |q|^3} \end{pmatrix}$. Theorem \ref{Thm 4.2} has the following stronger statement. Here 'stronger' means that $|s|^2 < 3q_1^2+{2\over|q|}-2c_0$ is replaced by $|s|^2 \le 3q_1^2+{2\over|q|}-2c_0$ and it will be helpful for our argument. 

\begin{Thm} 
\label{Thm 4.3}
We define $w(q)=\begin{pmatrix} {q_2 \over |q|^3} \\ 3q_1-{q_1 \over |q|^3} \end{pmatrix}$, $\mathcal{H}(q)={1 \over |q|^5} \begin{pmatrix} -2|q|^5+|q|^2-3q_1^2 & -3q_1 q_2 \\ -3q_1 q_2 & |q|^5+|q|^2-3q_2^2 \end{pmatrix}$ for $q \in \mathfrak{R}=\{(q_1, q_2) \in \mathbb{R}^2 | {1\over \sqrt{q_1^2+q_2^2}}+{3 \over 2}q_1^2>c_0, |q_1|<3^{-1 \over 3}, |q_2|<2 \cdot 3^{-4 \over 3}\}$. Then the inequality $(w(q)+s)^t \mathcal{H}(q)(w(q)+s)>0$ holds for all $q \in \mathfrak{R}$ and $|s|^2 \le 3q_1^2+{2\over|q|}-2c_0$. 
\end{Thm}

It is suffices to prove Theorem \ref{Thm 4.3} for the proof of our main Theorem. This is nothing but an inequality problem with constraints and restrictions of parameters. At this moment we have 4 dimensional parameters. We will explain the strategy to reduce the dimension in the next section.

\subsection{Strategy for the proof of Theorem \ref{Thm 4.3}}
\label{strategy}

In this section, we discuss the strategy for the proof of Theorem \ref{Thm 4.3}. First, we divide Theorem \ref{Thm 4.3} into the following three steps.
\begin{itemize}
\item \hypertarget{step 1}{$\bold{Step \ \ 1}$} : $(w(q)+s)^t \mathcal{H}(q)(w(q)+s)>0$ for all $q \in \mathfrak{R} \cap B_{0.54}(0)$ and $|s|^2 \le 3q_1^2+{2\over|q|}-2c_0$.

\item \hypertarget{step 2}{$\bold{Step \ \ 2}$} : $(w(q)+s)^t \mathcal{H}(q)(w(q)+s)>0$ for all $q \in \mathfrak{R} \cap (B_{0.63}(0) \backslash B_{0.54}(0))$ and $|s|^2 \le 3q_1^2+{2\over|q|}-2c_0$.

\item \hypertarget{step 3}{$\bold{Step \ \ 3}$} : $(w(q)+s)^t \mathcal{H}(q)(w(q)+s)>0$ for all $q \in \mathfrak{R} \backslash B_{0.63}(0)$ and $|s|^2 \le 3q_1^2+{2\over|q|}-2c_0$.
\end{itemize}
Here, $B_{r}(0)$ is the open disk with center the origin and radius $r$. This division of steps is visualized in Figure \ref{fig 3}. Note that the radii $0.54, 0.63$ are taken only for computational convenience. Obviously, these three steps imply Theorem \ref{Thm 4.3}. \hyperlink{step 1}{$\bold{Step \ \ 1}$} can be proven directly by using simple estimates. In fact, this can be proven by the following simple argument.

\begin{figure}
\centering
\includegraphics[]{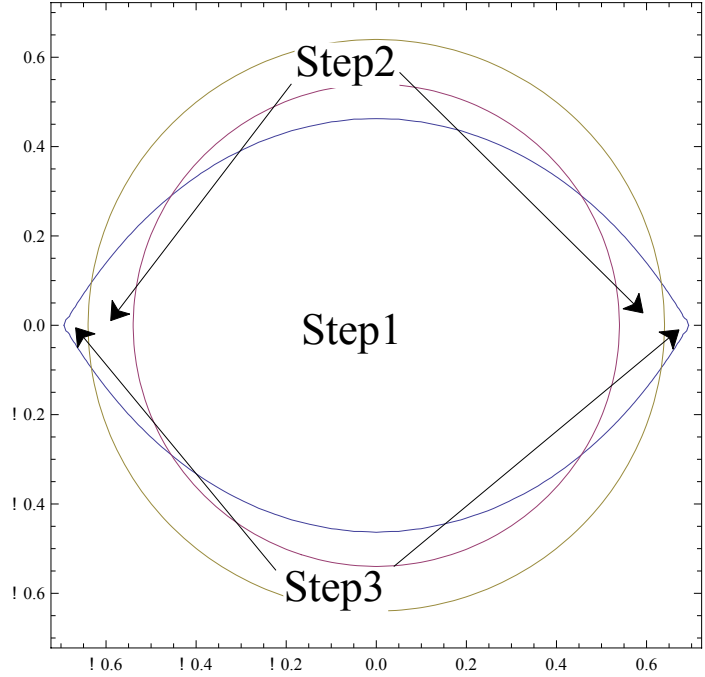}
\caption{Partition of $\mathfrak{R}$ by radius}
\label{fig 3}
\end{figure}

\begin{prop}
\label{prop 4.4}
We have the following estimate.
\begin{eqnarray*}
& &(w(q)+s)^t \mathcal{H}(q) (w(q)+s)
\\ &\ge& {11 \over 12}{1\over r^7}-{10\over3}{1\over r^4}-{29 \over 6}{1\over r}+4r^2-{3\over4}r^5-2({1 \over r^5}-{2 \over r^2})\sqrt{3r^2+{2 \over r}-3^{4\over3}}
\\ & & -(2+{2 \over r^3})(3r^2+{2 \over r}-3^{4\over3})
\end{eqnarray*}
for every  $q \in \mathfrak{R} \cap B_{0.54}(0)$ where $r=|q|$. Moreover, the inequality
$${11 \over 12}{1\over r^7}-{10\over3}{1\over r^4}-{29 \over 6}{1\over r}+4r^2-{3\over4}r^5-2({1 \over r^5}-{2 \over r^2})\sqrt{3r^2+{2 \over r}-3^{4\over3}}-(2+{2 \over r^3})(3r^2+{2 \over r}-3^{4\over3})>0$$
holds for every $r \in (0, 0.54)$.
\end{prop}

However, it is hard to obtain a strict inequality like in Proposition \ref{prop 4.4} for \hyperlink{step 2}{$\bold{Step \ \ 2}$} and \hyperlink{step 3}{$\bold{Step \ \ 3}$} because of the behaviors of $w(q)$, $\sqrt{3q_1^2+{2 \over |q|}-2c_0}$ and $\mathcal{H}(q)$ near the critical point. One can see that $w(q)$ and $\sqrt{3q_1^2+{2 \over |q|}-2c_0}$ go to zero and $\mathcal{H}(q)$ goes to $\begin{pmatrix}-8 & 0\\ 0 & 4 \end{pmatrix}$ as $q$ goes to the critical point $(3^{-1 \over 3}, 0)$. Thus we have to consider not only the convergent speeds of $w(q)$ and $\sqrt{3q_1^2+{2 \over |q|}-2c_0}$ but also the direction of $w(q)$ to see that the tangential Hessians $(w(q)+s)^t \mathcal{H}(q) (w(q)+s)$ are positive for all $|s|<\sqrt{3q_1^2+{2 \over |q|}-2c_0}$. We interpret Theorem \ref{Thm 4.3}. as a minimum value problem for the parameter $s$. Namely, we have the following equivalent form.
$$(w(q)+s)^t \mathcal{H}(q)(w(q)+s)>0 \textrm{ for all } q \in \mathfrak{R} \textrm{ and } |s|^2 \le 3q_1^2+{2\over|q|}-2c_0$$
$$\iff \min_{|s|^2 \le 3q_1^2+{2 \over |q|}-3^{4/3}} (w(q)+s)^t \mathcal{H}(q) (w(q)+s)>0 \textrm{ for all } q \in \mathfrak{R}.$$
We can concentrate only on the first quadrant of $\mathfrak{R}$ including axes by the symmetry argument in Remark \ref{Rem 3.1}. We define  $\mathfrak{R}^{+}:=\{(q_1, q_2) \in \mathbb{R}^2 | {1\over \sqrt{q_1^2+q_2^2}}+{3 \over 2}q_1^2>c_0, |q_1|<3^{-1 \over 3}, |q_2|<2 \cdot 3^{-4 \over 3}, q_1 \ge 0, q_2  \ge 0\}$ the first quadrant, including $q_1, q_2$-axis, of $\mathfrak{R}$ . Moreover, we can reduce the domain of $s$ which needs to be examined by one dimension by proving the following Proposition. We will use Proposition \ref{prop 4.5} for the proofs of \hyperlink{step 2}{$\bold{Step \ \ 2}$} and \hyperlink{step 3}{$\bold{Step \ \ 3}$}. 
\begin{prop}
\label{prop 4.5}
The following equality holds for every $q \in \mathfrak{R}^{+} \backslash B_{0.54}(0)$.
$$\min_{|s|^2 \le 3q_1^2+{2 \over |q|}-3^{4/3}} (w(q)+s)^t \mathcal{H}(q) (w(q)+s)=\min_{\alpha \in [\theta, \theta+{\pi \over 2}]} (w(q)+s_{q, \alpha})^t \mathcal{H}(q) (w(q)+s_{q, \alpha})$$
where $s_{q, \alpha}=\sqrt{3q_1^2+{2 \over |q|}-3^{4/3}}\begin{pmatrix} \cos \alpha \\ \sin \alpha \end{pmatrix}$ is a point of $\partial B_{\sqrt{3q_1^2+{2 \over |q|}-3^{4 \over 3}}}(0)$ and $\theta$ is the angle of $q$ in polar coordinates.
\end{prop}

Proposition \ref{prop 4.5} means that the minimum occurs at a specific part of the boundary of the disk. The proof of Proposition \ref{prop 4.5} consists of Lemma \ref{Lem 4.6}, \ref{Lem 4.7} and \ref{Lem 4.8}. To establish these Lemmas, we define 
$$F_q: D_q \rightarrow \mathbb{R} , \quad F_q(s)=(w(q)+s)^t \mathcal{H}(q) (w(q)+s)$$
for a fixed $q$. Here, $D_q:=\overline{B_{\sqrt{3q_1^2+{2 \over |q|}-3^{4 \over 3}}}(0)}$ denotes the closed disk of radius $\sqrt{3q_1^2+{2 \over |q|}-3^{4 \over 3}}$ with center at the origin.

\begin{Lem}
\label{Lem 4.6}
The function $F_q : D_q \rightarrow \mathbb{R}$ has no local minimum in $int(D_q)$ for all $q \in \mathfrak{R}^{+} \backslash B_{0.54}(0)$.
\end{Lem}

Lemma \ref{Lem 4.6} can be easily shown by examining the Hessian of $F_q$ in the variable $s$. If we prove Lemma \ref{Lem 4.6}, then we only need to see $F_q$ on the boundary of $D_q$. We define $F_q|_{\partial D_q} : S^1 \rightarrow \mathbb{R}$ by restricting $F_q$ to $\partial D_q$, that is, $F_q|_{\partial D_q}(\alpha)=F_q(s_{q, \alpha})=F_q(\sqrt{3q_1^2+{2 \over |q|}-3^{4 \over 3}} \begin{pmatrix} \cos \alpha \\ \sin \alpha \end{pmatrix}$. Here we use an abuse of notation that ignores the reparametrization of angle. With this notation, we introduce the following Lemmas.

\begin{Lem}
\label{Lem 4.7}
There exists a unique local minimum and a unique local maximum of the restricted function $F_q|_{\partial D_q}: S^1 \rightarrow \mathbb{R}$ for each $q \in \mathfrak{R}^{+} \backslash B_{0.54}(0)$.
\end{Lem}

\begin{Lem}
\label{Lem 4.8}
The unique minimum of $F_q|_{\partial D_q}: S^1 \rightarrow \mathbb{R}$ of Lemma \ref{Lem 4.7} is attained in $[\theta, \theta+{\pi \over 2}]$ where $q_1=r \cos \theta, q_2=r \sin \theta$.
\end{Lem}

Above three Lemmas will prove Proposition \ref{prop 4.5}. Once we prove Proposition \ref{prop 4.5}, it is enough to show the following inequality 
$$\min_{\alpha \in [0, {\pi \over 2}]} (w(q)+s_{q, \theta+\alpha})^t \mathcal{H}(q) (w(q)+s_{q, \theta+\alpha})>0 \textrm{ for all } q \in \mathfrak{R}^{+} \backslash B_{0.54}(0)$$
in order to prove \hyperlink{step 2}{$\bold{Step \ \ 2}$} and \hyperlink{step 3}{$\bold{Step \ \ 3}$} where $\theta$ is the angle of $q$ in polar coordinate. Thus we define the translated function 
$$f_q(\alpha):=F_q|_{D_q}(\theta+\alpha)=(w(q)+s_{q, \theta+\alpha})^t \mathcal{H}(q) (w(q)+s_{q, \theta+\alpha})$$
of $\alpha$ for each $q$. It suffices to prove that 
$$\min_{0 \le \alpha \le {\pi \over 2}}f_q(\alpha)>0$$
for all $q \in \mathfrak{R}^{+} \backslash B_{0.54}(0)$. Proposition \ref{prop 4.5} will be applied to both of \hyperlink{step 2}{$\bold{Step \ \ 2}$} and \hyperlink{step 3}{$\bold{Step \ \ 3}$}. But we will use different techniques for the proof of \hyperlink{step 2}{$\bold{Step \ \ 2}$} and \hyperlink{step 3}{$\bold{Step \ \ 3}$} from this point. We will use 'supporting tangent line inequality' for the proof of \hyperlink{step 2}{$\bold{Step \ \ 2}$} and 'blowing up the corner' for the proof of \hyperlink{step 3}{$\bold{Step \ \ 3}$}.

We will use a sharp estimate in the proof of \hyperlink{step 2}{$\bold{Step \ \ 2}$}. In general, it is hard to know where the minimum is attained for the problem $\min_{0 \le \alpha \le {\pi \over 2}}f_q(\alpha)$. Thus we need the following geometric observations to give another sufficient condition which allows us to forget $\alpha$.

\begin{Lem}
\label{Lem 4.9}
The function $f_q$ is convex on $[0, {\pi \over 2}]$ for each $q \in \mathfrak{R}^{+} \backslash B_{0.54}(0)$.
\end{Lem}

Using Lemma \ref{Lem 4.9}, we know that the tangent line at any point in this interval will lie below the graph of $f_q$. Let $l_q$ be the function for the tangent line at ${\pi \over 4}$, that is, $l_q$ is linear, $l_q ({\pi \over 4})=f_q({\pi \over 4})$ and ${d l_q \over d \alpha}({\pi \over 4})={d f_q \over d \alpha}({\pi \over 4})$. Then the inequality $f_q(\alpha) \ge l_q (\alpha)$ holds for every $\alpha \in [0, {\pi \over 2}]$. Moreover, we have the following obvious inequality 
$$\min_{\alpha \in [0, {\pi \over 2}]} l_q (\alpha) \ge \min_{\alpha \in [{\pi \over 4}-1, {\pi \over 4}+1]} l_q (\alpha)=\min \{l_q ({\pi \over 4}-1), l_q ({\pi \over 4}+1)\}$$
using ${\pi \over 4}-1<0$ and ${\pi \over 2}<{\pi \over 4}+1$ for the line $l_q$. We summarize the above arguments to get the following lower bound
$$\min_{\alpha \in [0, {\pi \over 2}]} f_q(\alpha) \ge \min_{\alpha \in [0, {\pi \over 2}]} l_q (\alpha) \ge \min \{l_q({\pi \over 4}-1), l_q({\pi \over 4}+1) \}.$$
Thus it is enough to prove $l_q ({\pi \over 4} \pm 1)>0$ in order to prove $\min_{\alpha \in [0, {\pi \over 2}]} f_q(\alpha)>0$ for a fixed $q$. We will prove $l_q ({\pi \over 4} \pm 1)>0$ one by one for some domains of $q$.
\begin{prop}
\label{prop 4.10}
The inequality $l_q ({\pi \over 4}+1)>0$ holds for every $q \in \mathfrak{R}^{+} \backslash B_{0.54}(0)$.
\end{prop}
\begin{prop}
\label{prop 4.11}
The inequality $l_q ({\pi \over 4}-1)>0$ holds for every $q \in \mathfrak{R}^{+} \cap (B_{0.63}(0) \backslash B_{0.54}(0))$.
\end{prop}
Because the inequality in Proposition \ref{prop 4.11} holds only for $q \in \mathfrak{R}^{+} \cap (B_{0.63}(0) \backslash B_{0.54}(0))$, Proposition \ref{prop 4.10} and \ref{prop 4.11} cannot cover the whole Hill's region and can imply only \hyperlink{step 2}{$\bold{Step \ \ 2}$}.

We will use blow-up coordinates for the proof of \hyperlink{step 3}{$\bold{Step \ \ 3}$} to resolve the convergence of $f_q$ to $0$ at the critical point. We will see $f_q(\alpha)$ as a function of $q$ and $\alpha$ again in order to factor out the zeros at the critical point $(3^{-1 \over 3}, 0)$. Namely, we define the function $G(q, \alpha):=f_q(\alpha)$. We introduce a lower bound function $D(q, \alpha)$ for $G(q, \alpha)$, that is $G(q, \alpha) \ge D(q, \alpha)$ by removing small terms. Roughly speaking, we remove the high degree terms of $(3^{-1 \over 3}-|q|)$ from $G(q, \alpha)$. Since we have $\lim_{q \rightarrow (3^{-1 \over 3}, 0)} {D(q, \alpha)}=0$, we want to factor out the factor $(3^{-1 \over 3}-|q|)$ as many times as possible in order to obtain the positivity of the function $D$ easier. We will see that the quotient ${D(q, \alpha) \over (3^{-1 \over 3}-|q|)^2}$ is well-defined on $\mathfrak{R}^{+} \backslash B_{0.54}(0)$. However it does not have a continuous extension to the boundary of $\mathfrak{R}^{+} \backslash B_{0.54}(0)$ because $\lim_{q \rightarrow (3^{-1 \over 3}, 0)}{D(q, \alpha) \over (3^{-1 \over 3}-|q|)^2}$ does not exist. Thus we want to enlarge near this critical point. We introduce a coordinate change which blows up the critical point by taking into account the direction to the critical point. We will see this coordinate change as a composition of two coordinate changes and we will prove its well-definedness in section \ref{proof}. We summarize the result here. We define the coordinate changes $\Phi$ and its inverse $\Psi$
$$\Phi : (0.54, 3^{-1 \over 3}) \times [0,1) \rightarrow \mathfrak{R}^{+} \backslash B_{0.54},$$
$$ \Phi(r, k)=(r \cos \theta(r, k), r \sin \theta(r, k)) \textrm{ where } \cos^2 \theta(r, k)={{1+3k(3^{1 \over 3}r-1)} \over {1+k(3r^3-1)}}$$
$$ \Psi :\mathfrak{R}^{+} \backslash B_{0.54} \rightarrow (0.54, 3^{-1 \over 3}) \times [0,1),$$
$$ \Psi(q)=(r, {{\sin^2 \theta} \over {3r^3 \cos^2 \theta -3^{4\over3}r+3-\cos^2 \theta}}) \textrm{ where } (q_1, q_2)=(r \cos \theta, r \sin \theta)$$
between a rectangle $(0.54, 3^{-1 \over 3}) \times [0, 1)$ and $\mathfrak{R}^{+} \backslash B_{0.54}$.

\begin{figure}
\centering
\includegraphics[]{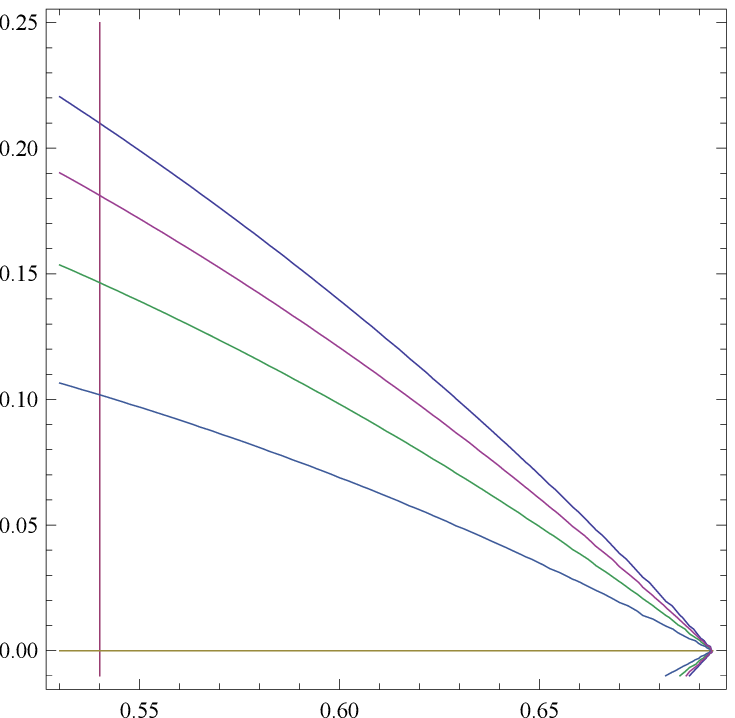}
\includegraphics[]{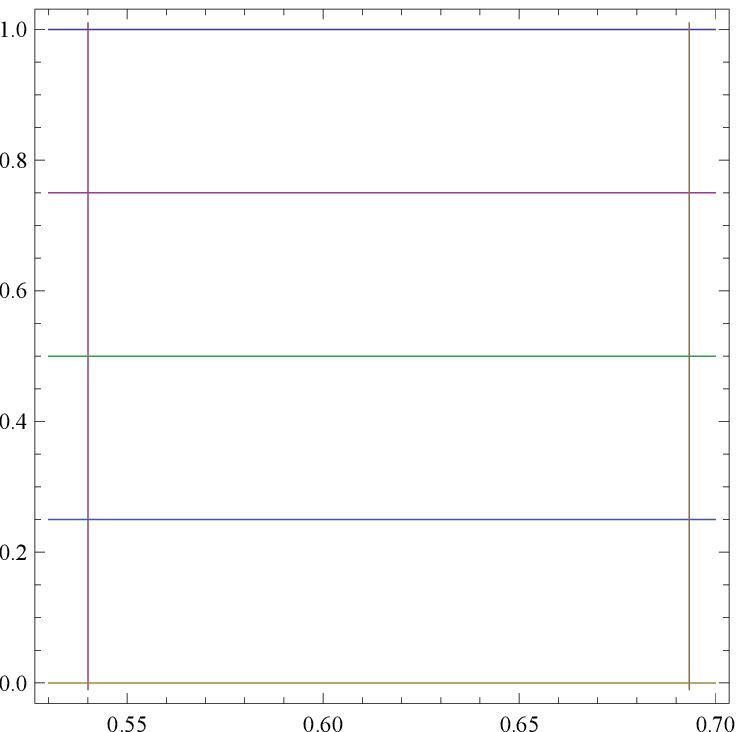}
\caption{Coordinate change which blows up the corner $(3^{-1 \over 3}, 0)$}
\label{fig 4}
\end{figure}

As we can see in Figure \ref{fig 4}, the critical point $(3^{-1 \over 3}, 0)$ corresponds to one side of the rectangle and also this side keeps the information of direction to the critical point like a "blow-up" procedure. We define the function $d(r, k, \alpha):={D(\Phi(r,k), \alpha) \over (3^{-1 \over 3}-r)^2}$ on the new coordinate $(0.54, 3^{-1 \over 3}) \times [0,1)$. Then it is sufficient to prove that $d(r, k, \alpha)>0$ in $(r, k, \alpha) \in (0.63, 3^{-1 \over 3}) \times [0, 1) \times [0, {\pi \over 2}]$. We will prove this in the following procedure.
\\ \\
Claim 1. The function $d : (0.54, 3^{-1 \over 3}) \times [0,1) \times [0, {\pi \over 2}] \rightarrow \mathbb{R}$ can be extended continuously to the boundary of its domain.
\\ \\
We will denote this extension also by $d$. Thus we have the function $d: [0.54, 3^{-1 \over 3}] \times [0, 1] \times [0, {\pi \over 2}] \rightarrow \mathbb{R}$ up to the boundary.
\\ \\
Claim 2. The function $d$ is monotone decreasing on $[0.63, 3^{-1 \over 3}] \times [0, 1] \times [0, {\pi \over 2}]$ with respect to $r$. 
\\ \\
We will show that ${{\partial d} \over {\partial r}}<0$ on $[0.63, 3^{-1 \over 3}] \times [0, 1] \times [0, {\pi \over 2}]$ to prove Claim 2.
\\ \\
Claim 3. The inequality $d(3^{-1 \over 3}, k, \alpha) \ge 0$ holds for all $(k, \alpha) \in [0, 1] \times [0, {\pi \over 2}]$. 
\\ \\
By showing the above three Claims, we will get $d(r, k, \alpha)>d(3^{-1 \over 3}, k, \alpha) \ge 0$ for all $(0.63, 3^{-1 \over 3}) \times [0, 1] \times [0, {\pi \over 2}]$. This will imply the desired inequality $G(q, \alpha)>0$ on $\mathfrak{R}^{+} \backslash B_{0.63}(0)$ for \hyperlink{step 3}{$\bold{Step \ \ 3}$}. This will complete the proof of \hyperlink{step 3}{$\bold{Step \ \ 3}$} by Proposition \ref{prop 4.5}.

We have introduced a numerical form of fiberwise convexity and the strategy of its proof. We will give the details of computations in section \ref{proof}, .

\section{Proof of Theorem \ref{Thm 4.3}}
\label{proof}

We introduced the notations to state and modified the main Theorem in section \ref{preparation}. We will use the notations again. We determined the region where the bounded component of the curve $H_{c, p}^{-1}(0)$ can lie on. We denoted the Hill's region by $\mathfrak{R}:=\{(q_1, q_2) \in \mathbb{R}^2 | {1\over \sqrt{q_1^2+q_2^2}}+{3 \over 2}q_1^2>c_0, |q_1|<3^{-1 \over 3}, |q_2|<2 \cdot 3^{-4 \over 3}\}$.
We defined a tangent vector $v(q)$ of $H_{c, 0}^{-1}(0)$ and the Hessian $\mathcal{H}(q)$ of $H_{c, 0}$ at each $q \in H_{c, 0}^{-1}(0)$. With these notations, we could get the following expressions
$$J\triangledown H_{c,p}(q)=
\begin{pmatrix} q_2+{q_2 \over |q|^3}+p_1 \\ 2q_1-{q_1 \over |q|^3}+p_2  
\end{pmatrix}
=v(q)+p \in T_q H_{c, p}^{-1}(0),$$
$$Hess H_{c, p}(q)=Hess H_{c, 0}(q)
={1 \over |q|^5} \begin{pmatrix} -2|q|^5+|q|^2-3q_1^2 & -3q_1 q_2 \\ -3q_1 q_2 & |q|^5+|q|^2-3q_2^2 
\end{pmatrix}=\mathcal{H}(q)$$
for a tangent vector and Hessian of $H_{c, p}$ at $q \in H_{c, p}^{-1}(0)$ for every $p \in \mathbb{R}^2$. Then the convexity of the closed curve $H_{c, p}^{-1}(0) \cap \mathfrak{R}$ is equivalent to the positivity of the tangential Hessian. Hence we have to prove the inequality
$$(J\triangledown H_{c,p}(q))^t Hess H_{c, p}(q) (J\triangledown H_{c,p}(q))=(v(q)+p)^t \mathcal{H}(v(q)+p)>0$$
for all $q \in H_{c, p}^{-1}(0) \cap \mathfrak{R}$. It is hard to describe the trajectory of $q \in H_{c, p}^{-1}(0) \cap \mathfrak{R}$. We will see this inequality with fixing $q$ rather than $p$ and $c$. We have seen that the equivalent condition.
$$q \in H_{c, p}^{-1}(0) \cap \mathfrak{R} \iff q \in \mathfrak{R} \textrm{ and } |p+Jq|^2<3q_1^2+{2 \over |q|}-2c_0$$
We also defined  that $w(q):=v(q)-Jq$, $s:=p+Jq$ and so $v(q)+p=w(q)+s$ for the computational convenience. We have established Theorem \ref{Thm 4.3}.
\begin{Thm 4.3.}
With above notations, the inequality $(w(q)+s)^t \mathcal{H}(q)(w(q)+s)>0$ holds for every $q \in \mathfrak{R}$ and $|s|^2 \le 3q_1^2+{2\over|q|}-2c_0$. 
\end{Thm 4.3.}
We divided Theorem \ref{Thm 4.3} into three steps by the position of $q$.
\begin{itemize}
\item $\bold{Step \ \ 1}$: $q \in  \mathfrak{R} \cap B_{0.54}(0)$ 
\item $\bold{Step \ \ 2}$: $q \in  \mathfrak{R} \cap (B_{0.63}(0) \backslash B_{0.54}(0))$ 
\item $\bold{Step \ \ 3}$: $q \in  \mathfrak{R} \backslash B_{0.63}(0)$ 
\end{itemize}
As we mentioned in section \ref{strategy}, the proof of \hyperlink{step 1}{$\bold{Step \ \ 1}$} can be done by proving Proposition \ref{prop 4.4}.

\begin{prop 4.4.}
We have the following estimate.
\begin{eqnarray*}
& &(w(q)+s)^t \mathcal{H}(q) (w(q)+s)
\\ &\ge& {11 \over 12}{1\over r^7}-{10\over3}{1\over r^4}-{29 \over 6}{1\over r}+4r^2-{3\over4}r^5-2({1 \over x^5}-{2 \over x^2})\sqrt{3r^2+{2 \over r}-3^{4\over3}}
\\ & & -(2+{2 \over r^3})(3r^2+{2 \over r}-3^{4\over3})
\end{eqnarray*}
for all  $q \in \mathfrak{R} \cap B_{0.54}(0)$ where $r=|q|$. Moreover, the following inequality 
$${11 \over 12}{1\over r^7}-{10\over3}{1\over r^4}-{29 \over 6}{1\over r}+4r^2-{3\over4}r^5-2({1 \over x^5}-{2 \over x^2})\sqrt{3r^2+{2 \over r}-3^{4\over3}}-(2+{2 \over r^3})(3r^2+{2 \over r}-3^{4\over3})>0$$
holds for every $r \in (0, 0.54)$.
\end{prop 4.4.}

\begin{proof}[Proof of Proposition \ref{prop 4.4}.] \hypertarget{pf of prop 4.4}
We will achieve the estimate. After that, the second inequality can be seen simply by its graph. We will omit $q$ of $w(q)$ and $\mathcal{H}(q)$ for notational convenience.
$$(w+s)^t \mathcal{H}(w+s)=w^t \mathcal{H}w+2w^t \mathcal{H}s+s^t \mathcal{H} s$$
We will make several estimate for each of the terms. We express the first term
$$w^t \mathcal{H} w={1 \over |q|^7}-{{5q_1^2+2q_2^2} \over |q|^6}+{3q_1^2 \over |q|^3}-{27q_1^2 q_2^2 \over |q|^5}+9q_1^2$$
in the polar coordinates $q_1=r \cos \theta, q_2=r \sin \theta$. Then we have the function $f_1(r, \theta)$
$$w^t \mathcal{H}w={1 \over r^7}-{{5 \cos^2 \theta} \over r^4}-{{2 \sin^2 \theta} \over r^4}+{{3 \cos^2 \theta} \over r}-{{27 \cos^2 \theta \sin^2 \theta} \over r}+9r^2 \cos^2 \theta=:f_1(r, \theta)$$
in terms of $r, \theta$. Differentiating $f_1$ with respect to $\theta$
\begin{eqnarray*}
{\partial f_1 \over \partial \theta}&=&{6 \over r^4}\cos \theta \sin \theta -{6 \over r}\cos \theta \sin \theta -{54 \over r}\cos \theta \sin \theta(\cos^2 \theta -\sin^2 \theta)-18r^2 \cos \theta \sin \theta
\\ &=&\cos \theta \sin \theta \big( {6 \over r^4}-{6 \over r}-{54 \over r}(2\cos^2 \theta -1)-18r^2 \big)
\end{eqnarray*}
gives us the candidates for the minimum points. Namely, $w^t \mathcal{H} w$ attains its minimum at one of these cases: $\cos^2 \theta=1, 0$ or ${1 \over 18r^3}+{4 \over 9}-{r^3 \over 6}$ for a fixed $r$.
\begin{eqnarray*}
w^t \mathcal{H}w \ge \min 
\begin{cases} 
1) \quad {1\over r^7}-{5 \over r^4}+{3 \over r}+9r^2 & \textrm{ when } \cos^2 \theta=1,
\\ 2) \quad {1 \over r^7}-{2 \over r^4} & \textrm{ when } \cos^2 \theta=0,
\\ 3) \quad {11\over12}{1 \over r^7}-{10 \over 3}{1 \over r^4}-{29 \over 6}{1 \over r}+ 4r^2-{3\over4}r^5  & \textrm{ when } \cos^2 \theta={1 \over 18r^3}+{4 \over 9}-{r^3 \over 6}.
\end{cases} 
\end{eqnarray*}
Claim: $2) \ge 1) \ge 3)$
\begin{proof}
\begin{eqnarray*}
1) \ge 3) : & & 12r^7 (({1 \over r^7}-{5 \over r^4}+{3 \over r}+9r^2)-({11 \over 12}{1 \over r^7}-{10 \over 3}{1 \over r^4}-{29 \over 6}{1 \over r}+4r^2-{3 \over 4}r^5))
\\ &=&9r^{12}+60r^9+94r^6-20r^3+1
\\ &=&(3r^6+10r^3-1)^2 \ge 0
\end{eqnarray*}
\begin{eqnarray*}
2) \ge 1) : & & r^4 (({1 \over r^7}-{2 \over r^4})-({1 \over r^7}-{5 \over r^4}+{3 \over r}+9r^2))
\\ &=&3-3r^3-9r^6
\\ &>&3-3 \times {1\over 3}- 9 \times {1\over 9}=1
\end{eqnarray*}
Here we use $r^3<0.54^3<{1 \over 3}$.
\end{proof}
By above Claim, we know that $3)$ is a lower bound for $w^t \mathcal{H}w$. Although $3)$ makes sense only when $0 \le {1 \over 18r^3}+{4 \over 9}-{r^3 \over 6} \le 1$, this is not required to get a lower bound. We have an estimate
$$w^t \mathcal{H}w \ge {11\over12}{1 \over r^7}-{10 \over 3}{1 \over r^4}-{29 \over 6}{1 \over r}+ 4r^2-{3\over4}r^5$$
for $w^t \mathcal{H}w$.

We make an estimate for the second term. We have $|\mathcal{H}w|^2$
$$\mathcal{H}w={1\over |q|^5}\begin{pmatrix} {q_2 \over |q|}-2q_2|q|^2-9q_1^2 q_2 \\ -{q_1 \over |q|}+2q_1 |q|^2-9q_1 q_2^2+3q_1|q|^5 \end{pmatrix}$$
$$\implies |\mathcal{H}w|^2={1\over |q|^{10}}-{4\over |q|^7}+{4 \over|q|^4}+{81q_1^2 q_2^2 \over |q|^8}-{6q_1^2 \over |q|^6}+{12q_1^2 \over |q|^3}-{54q_1^2 q_2^2 \over |q|^5}+9q_1^2$$
$$={1\over r^{10}}-{4\over r^7}+{4 \over r^4}+{81\over r^4}\cos^2 \theta \sin^2 \theta-{6\over r^4}\cos^2 \theta +{12\over r}\cos^2 \theta -{54\over r}\cos^2 \theta \sin^2 \theta +9r^2 \cos^2 \theta$$
in terms of $r, \theta$. This has its maximum at one of these cases: $\cos^2 \theta =0, 1$ or ${-3r^6-4r^3+2 \over 18r^3-27}$ by the similar computation as before. Therefore we have
\begin{eqnarray*}
|\mathcal{H}w|^2 \le \max 
\begin{cases} 
1') \quad{1\over r^{10}}-{4 \over r^7}+{4 \over r^4} & \textrm{ when } \cos^2 \theta=0,
\\ 2') \quad {1 \over r^{10}}-{4 \over r^7}-{2 \over r^4}+{12 \over r}+9r^2 & \textrm{ when } \cos^2 \theta=1,
\\ 3') \quad {1\over r^{10}}-{4 \over r^7}+{4 \over r^4}+{3(3r^6+4r^3-2) \over r^4} & \textrm{ when } \cos^2 \theta={-3r^6-4r^3+2 \over 18r^3-27}.
\end{cases}
\end{eqnarray*}
an upper bound for $|\mathcal{H}w|^2$. As before, it is easy to see that $1')>2'), 3')$. We get an estimate
$$|\mathcal{H}w|^2 \le {1\over r^{10}}-{4 \over r^7}+{4 \over r^4}=({1 \over r^5}-{2 \over r^2})^2$$
for $|\mathcal{H}w|$.

Finally, we will investigate the third term which is related with the eigenvalue of $\mathcal{H}$. The characteristic polynomial $p_{\mathcal{H}}(\lambda)$ of $\mathcal{H}=\begin{pmatrix} -2+{1 \over |q|^3}-{3q_1^2 \over |q|^5} & {-3q_1 q_2 \over |q|^5} \\ {-3q_1 q_2 \over |q|^5} & 1+{1\over |q|^3}-{3q_2^2 \over |q|^5} \end{pmatrix}$ has the following form.
$$p_{\mathcal{H}}(\lambda)=\lambda^2+(1+{1\over|q|^3})\lambda+(-2-{1\over |q|^3}-{2\over |q|^6}+{{6q_2^2-3q_1^2}\over |q|^5})$$
We claim that $\det (\mathcal{H})=-2-{1\over |q|^3}-{2\over |q|^6}+{{6q_2^2-3q_1^2}\over |q|^5}<0$ for all $q$ in the Hill's region by the following computation.
$$-2-{1\over |q|^3}-{2\over |q|^6}+{{6q_2^2-3q_1^2}\over |q|^5}=-2-{2\over |q|^6}+{{5q_2^2 - 4q_1^2} \over |q|^5}=-2+{{4q_2^2-2q_1^2-2c} \over |q|^5}< -2+{{4-2c_0} \over |q|^5}<0$$
Here we use $q_1^2-{1 \over 2}q_2^2+{1 \over |q|}=c \iff {2 \over |q|^6}={{2c-2q_1^2+q_2^2}\over |q|^5}$, $c>c_0>2$ and $q_2<1$ on the Hill's region.
Then we have one positive and one negative eigenvalue, say $\lambda_{+}, \lambda_{-}$ respectively. We have a lower bound
$$\lambda_{-}={1\over2} \Big( -(1+{1 \over |q|^3})-\sqrt{9+{6\over|q|^3}+{9 \over |q|^6}-4({{6q_2^2-3q_1^2} \over |q|^5})} \Big) \ge -(2+{2 \over |q|^3})$$
for $\lambda_{-}$. If we summarize all these results, then we can get an estimate for $(w+s)^t \mathcal{H}(w+s)$.
$$(w+s)^t \mathcal{H}(w+s)=w^t \mathcal{H}w+2w^t \mathcal{H}s+s^t \mathcal{H}s \ge w^t \mathcal{H}w-2|\mathcal{H}w||s|+\lambda_{-}|s|^2$$
$$ \ge {11 \over 12}{1\over r^7}-{10\over3}{1\over r^4}-{29 \over 6}{1\over r}+4r^2-{3\over4}r^5-2({1 \over x^5}-{2 \over x^2})\sqrt{3r^2+{2 \over r}-3^{4\over3}}-(2+{2 \over r^3})(3r^2+{2 \over r}-3^{4\over3})$$
This proves the first statement. We can see that ${11 \over 12}{1\over r^7}-{10\over3}{1\over r^4}-{29 \over 6}{1\over r}+4r^2-{3\over4}r^5-2({1 \over x^5}-{2 \over x^2})\sqrt{3r^2+{2 \over r}-3^{4\over3}}-(2+{2 \over r^3})(3r^2+{2 \over r}-3^{4\over3})>0$ for $r \in (0, 0.54)$ from its graph in Figure \ref{fig 5}. Therefore we have proven Proposition \ref{prop 4.4}.
\end{proof}

\begin{figure}
\centering
\includegraphics[]{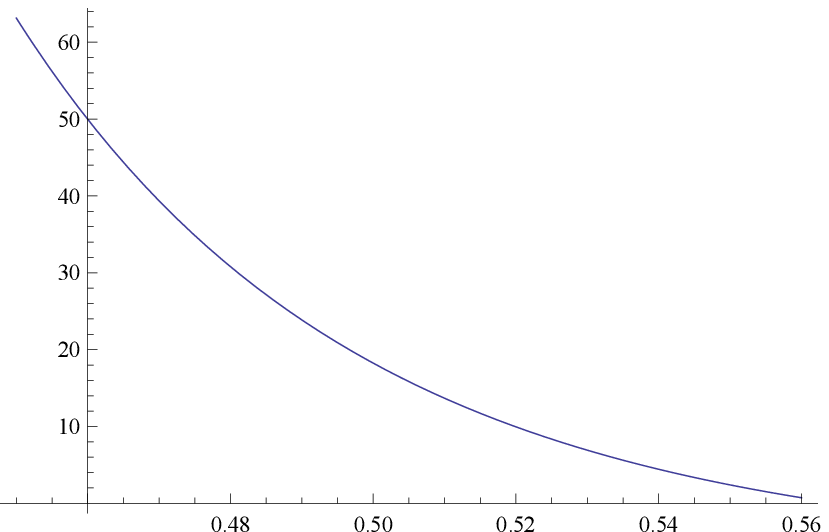}
\caption{Graph of $f_5(x)={11 \over 12}{1\over x^7}-{10\over3}{1\over x^4}-{29 \over 6}{1\over x}+4x^2-{3\over4}x^5
-2({1 \over x^5}-{2 \over x^2})\sqrt{3x^2+{2 \over x}-3^{4\over3}}-(2+{2 \over x^3})(3x^2+{2 \over x}-3^{4\over3})$ shows that it is positive on [0, 0.54].}
\label{fig 5}
\end{figure}

We have finished the proof of \hyperlink{step 1}{$\bold{Step \ \ 1}$}. Now we have to prove \hyperlink{step 2}{$\bold{Step \ \ 2, 3}$}. By the symmetry argument in Remark \ref{Rem 3.1}, we will see only the first quadrant of $\mathfrak{R} \backslash B_{0.54}(0)$. The first quadrant of the region for \hyperlink{step 2}{$\bold{Step \ \ 2, 3}$} is shown in Figure \ref{fig 6}. Therefore, we will assume that $q_1, q_2 \ge 0$, equivalently $0 \le \theta \le {\pi \over 2}$ in the polar coordinate, in the rest of this paper.

\begin{figure}
\centering
\includegraphics[]{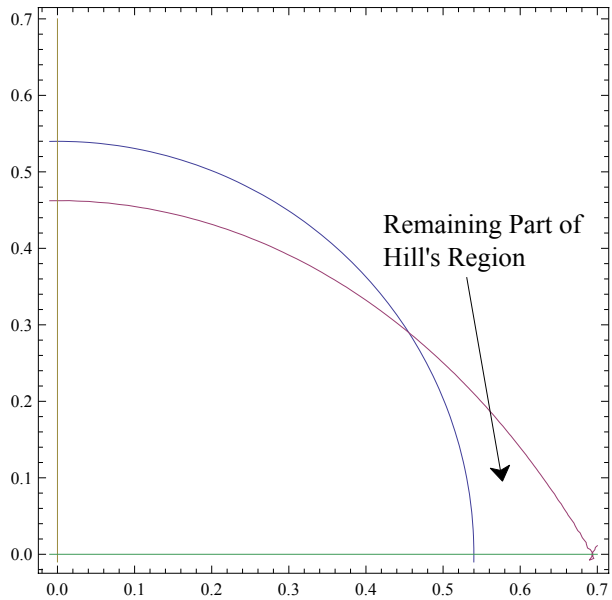}
\caption{The remaining Hill's region that we have to show on}
\label{fig 6}
\end{figure}

We consider the inequality 
$$0 \le |s|^2=3q_1^2+{2 \over |q|}-2c=3r^2 \cos^2 \theta +{2 \over r}-2c < 3r^2 \cos^2 \theta +{2 \over r}-2c_0$$
for the Hill's region. This implies that the boundary of the Hill's region satisfies the equation
$$3r^2 \cos^2 \theta +{2 \over r}=3^{4 \over 3} \iff \cos^2 \theta ={{3^{4 \over 3}-{2 \over r}}\over 3r^2}$$
in the polar coordinates. This gives the parametrization of the boundary. Since polar equations $r=0.54$ and $\cos^2 \theta ={{3^{4 \over 3}-{2 \over r}}\over 3r^2}$ intersect at $\cos^2 \theta ={{3^{4 \over 3}-{2 \over 0.54}}\over 3(0.54)^2}>0.7$, we may assume $\cos^2 \theta >0.7 $ in the remaining region. We can interpret this problem as an inequality problem of two dimensional variable $s$, if we fix the variable $q$. All possible $s$ forms a disk for a fixed $q$. The following Proposition allows us to reduce the domain of $s$ that we have to consider for the minimum value.

\begin{prop 4.5.}
The following equality holds for every $q \in \mathfrak{R}^{+} \backslash B_{0.54}(0)$.
$$\min_{|s|^2 \le 3q_1^2+{2 \over |q|}-3^{4/3}} (w(q)+s)^t \mathcal{H}(q) (w(q)+s)=\min_{\alpha \in [\theta, \theta+{\pi \over 2}]} (w(q)+s_{q, \alpha})^t \mathcal{H}(q) (w(q)+s_{q, \alpha})$$
where $s_{q, \alpha}=\sqrt{3q_1^2+{2 \over |q|}-3^{4/3}}\begin{pmatrix} \cos \alpha \\ \sin \alpha \end{pmatrix}$ is a point of $\partial B_{\sqrt{3q_1^2+{2 \over |q|}-3^{4 \over 3}}}(0)$ and $\theta$ is the angle of $q$ in polar coordinates.
\end{prop 4.5.}

\begin{proof} [Proof of Proposition \ref{prop 4.5}.]
As we mentioned before, we will show Lemma \ref{Lem 4.6}, \ref{Lem 4.7} and \ref{Lem 4.8}. Recall the function $F_q:D_q \rightarrow \mathbb{R}$ defined by $F_q(s)=(w(q)+s)^t \mathcal{H}(q) (w(q)+s)$ for each fixed $q$ where $D_q=B_{\sqrt{3q_1^2+{2 \over |q|}-3^{4 \over 3}}}(0)$.

\begin{Lem 4.6.}
The function $F_q : D_q \rightarrow \mathbb{R}$ has no local minimum in $int(D_q)$ for all $q \in \mathfrak{R}^{+} \backslash B_{0.54}(0)$.
\end{Lem 4.6.}

\begin{proof} [Proof of Lemma \ref{Lem 4.6}.]
For a fixed $q \in \mathfrak{R}^{+} \backslash B_{0.54}(0)$, $F_q$ is a quadratic function in variable $s$. Thus we get the Hessian $Hess F_q(s)=\mathcal{H}(q)$ of $F_q$ and we proved that $\mathcal{H}(q)$ has one positive eigenvalue and one negative eigenvalue in the \hyperlink{pf of prop 4.4}{proof of Proposition 4.4}. This implies that there is no local minimum and no local maximum in the interior of the range. This proves Lemma \ref{Lem 4.6}.
\end{proof}

As a result of Lemma \ref{Lem 4.6}, $F_q$ attains its minimum at the boundary of $D_q$. We define $s_{q, \alpha}=\sqrt{3r^2 \cos^2 \theta+{2 \over r}-3^{4 \over 3}}u_{\alpha}$ where $u_{\alpha}=\begin{pmatrix}\cos \alpha \\ \sin \alpha \end{pmatrix}$ for a fixed $\alpha$. Then Lemma \ref{Lem 4.6} implies that
$$\min_{|s|^2 \le 3q_1^2+{2 \over |q|}-3^{4/3}} (w+s)^t \mathcal{H} (w+s)=\min_{\alpha \in [0, 2\pi)} (w+s_{q, \alpha})^t \mathcal{H} (w+s_{q, \alpha}).$$
For convenience of computation, we will consider the translation of $\alpha$ by $\theta$ where $(q_1, q_2)=(r \cos \theta, r \sin \theta)$. Recall the function 
$$f_q : S^1 \rightarrow \mathbb{R}, \quad f_q(\alpha):=F_q|_{D_q}(\theta+\alpha)=(w+s_{q, \theta+\alpha})^t \mathcal{H} (w+s_{q, \theta+\alpha})$$
defined by restriction and translation. We have to prove the following statement 
$$\min_{\alpha \in [0, 2\pi)} f_q(\alpha)=\min_{\alpha \in [0, {\pi \over 2}]} f_q(\alpha)$$
in order to prove Proposition \ref{prop 4.5}.
We need the following Lemma.

\begin{Lem 4.7.}
There exists a unique local minimum and a unique local maximum of the restricted function $F_q|_{\partial D_q}: S^1 \rightarrow \mathbb{R}$ for each $q \in \mathfrak{R}^{+} \backslash B_{0.54}(0)$.
\end{Lem 4.7.}

\begin{proof} [Proof of Lemma \ref{Lem 4.7}] \hypertarget{pf of Lem 4.7}
We have the following expression for the function $f_q: S^1 \rightarrow \mathbb{R}$.
\begin{eqnarray*}
f_q(\alpha)&=&(w+s_{\theta+\alpha})^t \mathcal{H}(w+s_{\theta+\alpha})
\\ &=&w^t \mathcal{H}w+2\sqrt{2c-2c_0}w^t \mathcal{H}u_{\theta+\alpha}+(2c-2c_0)u_{\theta+\alpha}^t  \mathcal{H}u_{\theta+\alpha}
\\ &=&w^t \mathcal{H}w+2\sqrt{2c-2c_0}(\cos \alpha (3r-{9 \over r^2})\cos \theta \sin \theta+ \sin \alpha (-{1 \over r^5}+{2c \over r}))
\\ & &+(2c-2c_0)(\cos^2 \alpha (1-{2c \over r^2})+\sin^2 \alpha (-{1 \over r^3}+{2c \over r^2}-2)+2 \cos \alpha \sin \alpha(3 \cos \theta \sin \theta))
\end{eqnarray*}
We differentiate $f_q$ with respect to $\alpha$. Then we have
\begin{eqnarray*}
{df_q \over d\alpha}(\alpha)&=&{\partial \over \partial \alpha}((w+s_{\theta+\alpha})^t \mathcal{H}(w+s_{\theta+\alpha}))
\\ &=&2\sqrt{2c-2c_0}(-\sin \alpha (3r-{9 \over r^2})\cos \theta \sin \theta+ \cos \alpha (-{1 \over r^5}+{2c \over r}))
\\ & &+(2c-2c_0)(-2\cos \alpha \sin \alpha (1-{2c \over r^2})+2 \cos \alpha \sin \alpha (-{1 \over r^3}+{2c \over r^2}-2)
\\ & &+2 (\cos^2 \alpha- \sin^2 \alpha)(3 \cos \theta \sin \theta)
\\ &=&2\sqrt{2c-2c_0}(({9 \over r^2}-3r)\cos \theta \sin \theta)\sin \alpha + (-{1 \over r^5}+{2c \over r})\cos \alpha )
\\ & &+(2c-2c_0)((-{1 \over r^3}+{4c \over r^2}-3) \sin 2\alpha+(3 \cos \theta \sin \theta) \cos 2\alpha)
\\ &=:& A_1 \sin 2\alpha+A_2 \cos 2\alpha+B_1 \sin \alpha+B_2 \cos \alpha
\end{eqnarray*}

Claim 1 : $|B_1| \ge 2|A_2|$
\begin{proof} [Proof of Claim 1]
\begin{eqnarray*}
& &|B_1| \ge 2|A_2|
\\ &\iff& 2\sqrt{2c-2c_0}({9 \over r^2}-3r)\cos \theta \sin \theta \ge (2c-2c_0)(6 \cos \theta \sin \theta)
\\ &\iff& ({9 \over r^2}-3r-3\sqrt{2c-2c_0}) \ge 0
\end{eqnarray*}
This follows from the fact
\\ $r \in (0.54, 3^{-1\over 3})$ and $2c-2c_0 \le 3(0.54)^2+{2 \over 0.54}-3^{4 \over 3}<0.3.$ This proves Claim 1.
\end{proof}

Claim 2 : $|B_2| > 2 |A_1|$
\begin{proof} [Proof of Claim 2]
\begin{eqnarray*}
& &|B_2| \ge 2|A_1|
\\ &\iff& (-{1 \over r^5}+{2c \over r})^2-(2c-2c_0)(-{1 \over r^3}+{4c \over r^2}-3)^2>0
\\ &\iff& (-{1 \over r^5}-{2 \over r^2}-3r \cos^2 \theta)^2-(3r^2 \cos^2 \theta+{2 \over r}-3^{4 \over 3})(6 \cos^2 \theta+{3 \over r^3}-3)^2>0
\end{eqnarray*}
We define $\cos^2 \theta=:y$ and $g(r,y):=({1 \over r^5}-{2 \over r^2}-3ry)^2-(3r^2 y+{2 \over r}-3^{4 \over 3})(6y+{3 \over r^3}-3)^2$, then ${\partial g \over \partial y}=2({1\over r^5}-{2 \over r^2}-3ry)(-3r)-3r^2(6y+{3 \over r^3}-3)^2-6(3r^2 y+{2 \over r}-3^{4 \over 3})(6y+{3 \over r^3}-3)<0$. We can easily check three terms are all negative, and therefore it is enough to show that $g(r, 1)>0$,
that is, $({1 \over r^5}-{2 \over r^2 }-3r)^2-(3r^2+{2 \over r}-3^{4 \over 3})(3+{3 \over r^3})^2>0$. This is clear from a simple calculation. This proves Claim 2.
\end{proof}

Now we know $2\sqrt{A_1^2 + A_2^2}<\sqrt{B_1^2 + B_2^2}$ from Claim 1, 2. We need the following Lemma to get the number of local extrema. I borrow the following geometric proof of Lemma \ref{Lem 5.1} from Urs Frauenfelder. This Lemma can be proven with analytic way as well.

\begin{Lem}
\label{Lem 5.1}
If $A, B \in \mathbb{R}$ satisfy $2|A|<|B|$, then the equation for the unknown $\alpha$
$$A \sin(2 \alpha+\phi)+B\sin (\alpha+\mu)=0$$
has exactly 2 solutions on $[0, 2\pi)$ for any constant $\phi, \mu \in \mathbb{R}$.
\end{Lem}

\begin{proof}
Without loss of generality, we may assume that $B=1, A=t \in [0, {1 \over 2})$ and $\mu=0$.
\\ In the case of $t=0$, the above equation becomes $\sin \alpha=0$ and this has 2 solutions.
\\ Suppose that  there exist $t_0 \in [0, {1 \over 2})$ such that $t_0 \sin(2\alpha + \phi)+\sin \alpha=0$ does not have 2 solutions. We define a function
$$T: S^1 \times [0, t_0] \rightarrow \mathbb{R}, \quad T(\alpha, t)=t\sin(2\alpha+\phi)+\sin \alpha.$$
Then a critical point $(\alpha, t)$ of $T$ satisfies
$$\partial_t T=\sin(2\alpha+\phi)=0, \quad \partial_\alpha T=2t \cos(2\alpha+\phi)+\cos \alpha=0.$$
and this implies the equations
$$\sin(2\alpha+\phi)=0, \cos \alpha=\pm 2t.$$
Since $0 \le 2t<1$ these two equations are not compatible with the equation $t\sin(2\alpha+\phi)+\sin \alpha=0$. Thus $0$ is the regular value for $T$. Then we get $T^{-1}(0)$ is a smooth manifold with boundary. Because it has a different number of points in $S^1 \times \{0\}$ and $S^1 \times \{t_0\}$ by the assumption of $t_0$. There must be an appearance or disappearance of curve, so-called, 'birth and death' of curve. Let $(\alpha_1, t_1)$ be one of these points. Then $T(\alpha_1, t_1)=0$ and $\partial_{\alpha} T(\alpha_1, t_1)=0$, that is, we have
$$\begin{cases} t_1\sin(2\alpha_1+\phi)+\sin \alpha_1=0
\\ 2t_1\cos(2\alpha_1+\phi)+\cos \alpha_1=0
\end{cases}
\\ \implies
\begin{cases} 
t_1^2 \sin^2(2\alpha_1+\phi)=\sin^2 \alpha_1
\\ 4t_1^2 \cos^2(2\alpha_1+\phi)=\cos^2 \alpha_1
\end{cases}
$$
By adding these two equations, we get $1=t_1^2+3t_1^2 \cos^2(2\alpha_1+\phi) \le 4t_1^2<1$ and this gives a contradiction. Thus we have proven Lemma \ref{Lem 5.1}. 
\end{proof}

We continue the \hyperlink{pf of Lem 4.7}{proof of Lemma 4.7}. We have defined ${\partial \over \partial\alpha}[(w+s_{\theta+\alpha})^t \mathcal{H}(w+s_{\theta+\alpha})]=A_1\sin 2 \alpha+A_2 \cos 2 \alpha + B_1 \sin \alpha +B_2 \cos \alpha=A \sin(2\alpha+ \phi)+B \sin(\alpha+\mu)$ where $A=\sqrt{A_1^2+A_2^2}, B=\sqrt{B_1^2+B_2^2}$. We proved $2|A|<|B|$ by Claim 1, 2. Thus we get that ${\partial \over \partial\alpha}[(w+s_{\theta+\alpha})^t \mathcal{H}(w+s_{\theta+\alpha})]=0$ has exactly 2 solutions on $\alpha \in [0, 2\pi)$ by applying Lemma \ref{Lem 5.1}. This implies $f_q(\alpha)=(w+s_{\theta+\alpha})^t \mathcal{H}(w+s_{\theta+\alpha})$ has exactly two critical points. Since the domain of $f_q$ is $S^1$, there exist the unique local maximum and minimum respectively on $S^1$. This proves Lemma \ref{Lem 4.7}. 
\end{proof}

Now we need the following Lemma to reduce the region where the minimum is attained. The following Lemma will finish the proof of Proposition \ref{prop 4.5}.

\begin{Lem 4.8.}
The unique minimum of the function $f_q: S^1=\mathbb{R}/2\pi \mathbb{Z} \rightarrow \mathbb{R}$ is attained in $[0, {\pi \over 2}]$.
\end{Lem 4.8.}

\begin{proof}
Now we know $f_q$ has only one local minimum for fixed $q$ and so this will be the global minimum. We calculate the first derivative of $f_q$ 
\begin{eqnarray*}
{d \over df_q}(0)&=&{\partial \over \partial \alpha}\big|_{\alpha=0}(w^t \mathcal{H}w+2\sqrt{2c-2c_0}w^t \mathcal{H}u_{\theta+\alpha}+(2c-2c_0)u_{\theta+\alpha}^t  \mathcal{H}u_{\theta+\alpha})
\\ &=& 2\sqrt{2c-2c_0}(-{1 \over r^5}+{2c \over r})+(2c-2c_0)(6 \sin \theta \cos \theta))
\end{eqnarray*}
at $\alpha=0, {\pi \over 2}$. We can obtain ${d \over df_q}(0)<0$ because of the inequality ${1 \over r^5}-{2c \over r}>\sqrt{2c-2c_0}(6 \cos^2 \theta+{3 \over r^3}-3)$ from Claim 2 in the \hyperlink{Lem 4.7}{proof of Lemma 4.7} and the inequality $6\cos^2 \theta+{3 \over r^3}-3>3\sin \theta \cos \theta$.

Next, we compute the derivative of $f_q$ at ${\pi \over 2}$
\begin{eqnarray*}
{d \over df_q}({\pi \over 2})&=&{\partial \over \partial \alpha}\big|_{\alpha={\pi \over 2}}(w^t \mathcal{H}w+2\sqrt{2c-2c_0}w^t \mathcal{H}u_{\theta+\alpha}+(2c-2c_0)u_{\theta+\alpha}^t  \mathcal{H}u_{\theta+\alpha})
\\ &=& 2\sqrt{2c-2c_0}({9 \over r^2}-3r)\cos \theta \sin \theta+(2c-2c_0)(-6 \sin \theta \cos \theta))
\\ &=& 2\sqrt{2c-2c_0}\cos \theta \sin \theta({9 \over r^2}-3r-3\sqrt{2c-2c_0})
\end{eqnarray*}
and similarly we can obtain ${d \over df_q}({\pi \over 2})>0$. Therefore, there exists a unique local minimum on $\alpha \in (0, {\pi \over 2}]$ and this is the global minimum because the function $f_q:S^1 \rightarrow \mathbb{R}$ has only one local minimum. This proves Lemma \ref{Lem 4.8}.
\end{proof}

Now we can prove Proposition \ref{prop 4.5} by combining Lemma \ref{Lem 4.6}, \ref{Lem 4.7} and \ref{Lem 4.8}. We know that $F_q$ attains its minimum on the boundary of $D_q$ for any fixed $q \in \mathfrak{R}^{+} \backslash B_{0.54}(0)$ by Lemma \ref{Lem 4.6}. Moreover, we know $f_q$, the restriction of $F_q$ to $\partial D_q$ with the translation of angle by $\theta$, has only one local minimum and so it is global minimum and this minimum is attained in $[0, {\pi \over 2}]$ by Lemma \ref{Lem 4.7} and \ref{Lem 4.8}. Therefore, we get $\min_{|s|^2 \le 3q_1^2+{2 \over |q|}-3^{4/3}} (w(q)+s)^t \mathcal{H}(q) (w(q)+s)=\min_{\alpha \in [0, {\pi \over 2}]} (w(q)+s_{q, \theta+\alpha})^t \mathcal{H}(q) (w(q)+s_{q, \theta+\alpha})$ for all $q \in \mathfrak{R}^{+} \backslash B_{0.54}(0)$. This completes the proof of Proposition \ref{prop 4.5}.
\end{proof}

We will use the previous notations again. We recall that $f_q(\alpha):=(w+s_{\theta+\alpha})^t \mathcal{H}(w+s_{\theta+\alpha})$ for each fixed $q \in \mathfrak{R}^{+} \backslash B_{0.54}(0)$ where $s_{q, \alpha}=\sqrt{3r^2 \cos^2 \theta+{2 \over r}-3^{4 \over 3}}\begin{pmatrix}\cos \alpha \\ \sin \alpha \end{pmatrix} \in \partial D_q$.

\begin{Lem 4.9.}
The function $f_q$ is convex on $[0, {\pi \over 2}]$ for each $q \in \mathfrak{R}^{+} \backslash B_{0.54}(0)$.
\end{Lem 4.9.}

\begin{proof}
We calculate the second derivative
\begin{eqnarray*}
{\partial^2 \over \partial \alpha^2}f_q(\alpha)&=&{\partial^2 \over \partial \alpha^2}(w+s_{\theta+\alpha})^t \mathcal{H}(w+s_{\theta+\alpha})
\\ &=& 2\sqrt{2c-2c_0}[(({9 \over r^2}-3r) \cos \theta \sin \theta)\cos \alpha+({1 \over r^5}-{2c \over r})\sin \alpha]
\\ & &+(2c-2c_0)[2(-{1 \over r^3}+{4c \over r^2}-3) \cos 2 \alpha + 2(-6 \sin \theta \cos \theta) \sin 2\alpha]
\end{eqnarray*}
and we want to prove that it is positive  on $\alpha \in [0, {\pi \over 2}]$.

Claim 1: $({1 \over r^5}-{2c \over r})\sin \alpha+\sqrt{2c-2c_0}(-{1 \over r^3}+{4c \over r^2}-3) \cos 2 \alpha>0$ for $\alpha \in [0, {\pi \over 2}]$.
\begin{proof} [Proof of Claim 1]
We already know that $({1 \over r^5}-{2c \over r})>\sqrt{2c-2c_0}(-{1 \over r^3}+{4c \over r^2}-3)>0$. On the other hand, we have the following inequality. If $b>a>0$, then
$$a \cos 2 \alpha+b \sin \alpha>0$$
for all $\alpha \in [0, {\pi \over 2}]$. In fact, we have that $a \cos2\alpha+b \sin\alpha=-2a\sin^2\alpha+b\sin \alpha+a$ and $-2at^2+b+a>0$ for all $t \in [0, 1]$, if $b>a>0$.
Because $0 \le \sin \alpha \le 1$ on $\alpha \in [0, {\pi \over 2}]$. This proves Claim 1.
\end{proof}

Claim 2: $({9 \over r^2}-3r)\cos \theta \sin \theta \cos \alpha \ge \sqrt{2c-2c_0}(6\sin \theta \cos \theta) \sin2\alpha$ for $\alpha \in [0, {\pi \over 2}]$.
\begin{proof} [Proof of Claim 2]
It suffices to prove that $({9 \over r^2}-3r-12\sqrt{2c-2c_0} \sin \alpha)\cos \alpha \ge 0$ on $\alpha \in [0, {\pi \over 2}]$. This is clear, because we have that ${9 \over r^2}-3r^2>8 \cdot 3^{2 \over 3}$ and $2c-2c_0<0.3$ from the Claim 1 in the proof of Lemma \ref{Lem 4.7}. Thus, we have ${9 \over r^2}-3r^2>12\sqrt{2c-2c_0}$. This proves Claim 2
\end{proof}

We have shown that ${\partial^2 \over \partial \alpha^2}f_q(\alpha)>0$ for all $\alpha \in [0, {\pi \over 2}]$ by Claim 1, 2. This completes the proof of Lemma \ref{Lem 4.9}.
\end{proof}

We have proven that $\min_{|s| \le 3q_1^2+{2 \over |q|}-3^{4/3}} (w(q)+s)^t \mathcal{H}(q) (w(q)+s)=\min_{0 \le \alpha \le {\pi \over 2}}f_q(\alpha)$ and $f_q$ is convex on $[0, {\pi \over 2}]$ for all $q \in \mathfrak{R}^{+} \backslash B_{0.54}(0)$. Let $l_q(\alpha)=f_q^{`}({\pi \over 4})(\alpha-{\pi \over 4})+f_q({\pi \over 4})$ be the tangent line of $f_q$ at $\alpha={\pi \over 4}$ then this tangent line will be below the function. In particular, one of the end points of this line will be less than or equal to the minimum value of the function, see Figure \ref{fig 7}.  Thus we have that
$$\min_{|s| \le 3q_1^2+{2 \over |q|}-3^{4/3}} (w(q)+s)^t \mathcal{H}(q) (w(q)+s)=\min_{0 \le \alpha \le {\pi \over 2}} f_q(\alpha) \ge \min \{ l_q({\pi \over 4}+1), l_q({\pi \over 4}-1) \}.$$
Therefore we shall show that $\min_{|s| \le 3q_1^2+{2 \over |q|}-3^{4/3}} (w(q)+s)^t \mathcal{H}(q) (w(q)+s)>0$ by proving $l_q({\pi \over 4}+1)>0$ and $l_q({\pi \over 4}-1)>0$ in Proposition \ref{prop 4.10} and \ref{prop 4.11}, respectively.

\begin{figure}
\centering
\includegraphics[]{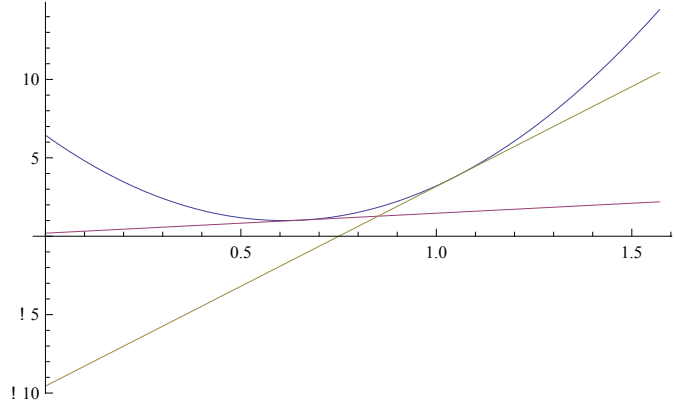}
\caption{Tangent lines of a convex function - The strategy is "One of the end points of a tangent line is below the minimum point of convex function."}
\label{fig 7}
\end{figure}

We need some new notations to prove Proposition \ref{prop 4.10} and \ref{prop 4.11}. From now on, we use the following coordinates and variables. We introduce new coordinates
$$x:=r, \quad y:=\cos^2 \theta$$
which are well-defined coordinates on the first quadrant of $(q_1, q_2)$-coordinate. The domain $\mathfrak{R}'$ of $(x, y)$ corresponding to the domain $\mathfrak{R}^{+} \backslash B_{0.54}(0)$ of $(q_1, q_2)$ is given by 
$$\mathfrak{R}':=\{(x, y) \in \mathbb{R}^2|0.54<x<3^{-1\over 3}, {{3^{4 \over 3}-{2 \over x}} \over 3x^2}<y \le 1 \}.$$
We will define a change of variables in terms of $x, y$ in the following Lemma.

\begin{Lem} [Blow-up coordinates change]
\label{Lem 5.2}
If we define the map $\phi : \mathfrak{R}'':=(0.54, 3^{-1 \over 3})\times[0, 1) \rightarrow \mathfrak{R}'$ by $(x, k) \rightarrow (x, y)$ where
$$y={1+3k(3^{1\over3}x-1) \over 1+k(3x^3-1)},$$
then $\phi$ is a diffeomorphism.
\end{Lem}

\begin{proof}
We compute the Jacobian of $\phi$. First, we note that ${{\partial y} \over {\partial k}}$ is not zero. In fact,
\begin{eqnarray*}
{{\partial y} \over {\partial k}}&=& {\partial \over \partial k}({1+3k(3^{1\over3}x-1) \over 1+k(3x^3-1)})
\\ &=& {-3x^3+3^{4 \over 3}x-2 \over (1+k(3x^3-1))^2}<0 \textrm{ for any fixed } x \in (0.54, 3^{-1\over3})
\end{eqnarray*}
Then the Jacobian of this map is given by
$${{\partial (x, y)}\over {\partial (x, k)}}
=\begin{pmatrix} 1 & 0 
\\ * & {-3x^3+3^{4 \over 3}x-2 \over (1+k(3x^3-1))^2} 
\end{pmatrix}$$
Thus we know that the Jacobian is nonsingular for every $(x, k) \in (0.54, 3^{-1 \over 3}) \times [0,1)$ by the above computation. This proves that $\phi$ is a local diffeomorphism.
We need to show that $\phi$ is a bijective map. If we assume $\phi(x_1, k_1)=\phi(x_2, k_2)$, then we have $x_1=x_2$ and ${1+3k_1(3^{1\over3}x_1-1) \over 1+k_1(3x_1^3-1)}={1+3k_2(3^{1\over3}x_2-1) \over 1+k_2(3x_2^3-1)}$. We get $k_1=k_2$ from the monotonicity of $y$ with respect to $k$ and so $(x_1, k_1)=(x_2, k_2)$. This proves the injectivity of $\phi$. For the surjectivity, we extend the map $\phi$ to the map on $(0.54, 3^{-1 \over 3})\times [0, 1]$ in obvious way. Then we have that
$$k=0 \implies y=1, \quad k=1 \implies y={3^{4\over3}x-2 \over 3x^3}.$$
This proves the surjectivity from the monotonicity. Therefore the map $\phi$ is a diffeomorphism. This proves Lemma \ref{Lem 5.2}.
\end{proof}

This diffeomorphism $\phi : \mathfrak{R}'' \rightarrow \mathfrak{R}'$ cannot be extended to the boundary as a diffeomorphism. As one can see in Figure \ref{fig 8}, the critical point $(3^{-1 \over 3}, 1)$ at the boundary of $\mathfrak{R}'$ corresponds to the one side $x=3^{-1 \over 3}, 0 \le k \le 1$ of the boundary of $\mathfrak{R}''$. If we use this map $\phi$ as a coordinate chart, then we can handle our problem on a rectangle domain. This coordinate chart will play an important role in the proof of \hyperlink{step 3}{$\bold{Step \ \ 3}$} as well as Proposition \ref{prop 4.10} and \ref{prop 4.11}.

\begin{figure}
\centering
\includegraphics[]{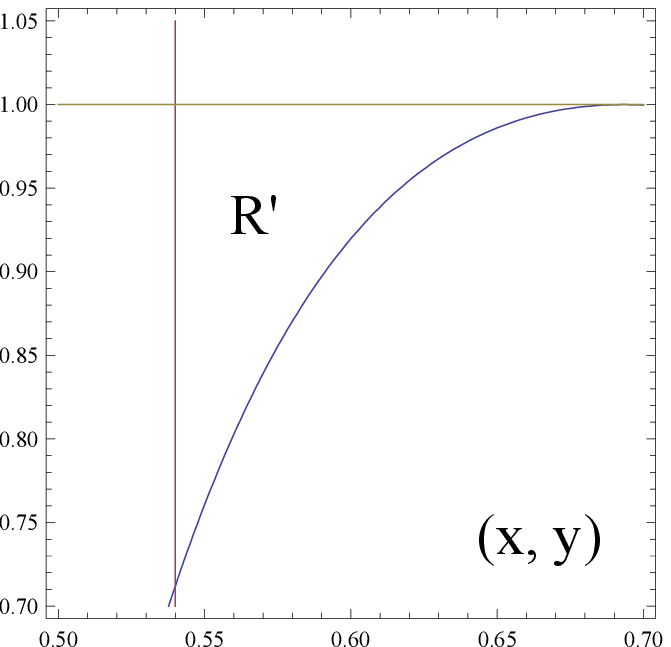}
\includegraphics[]{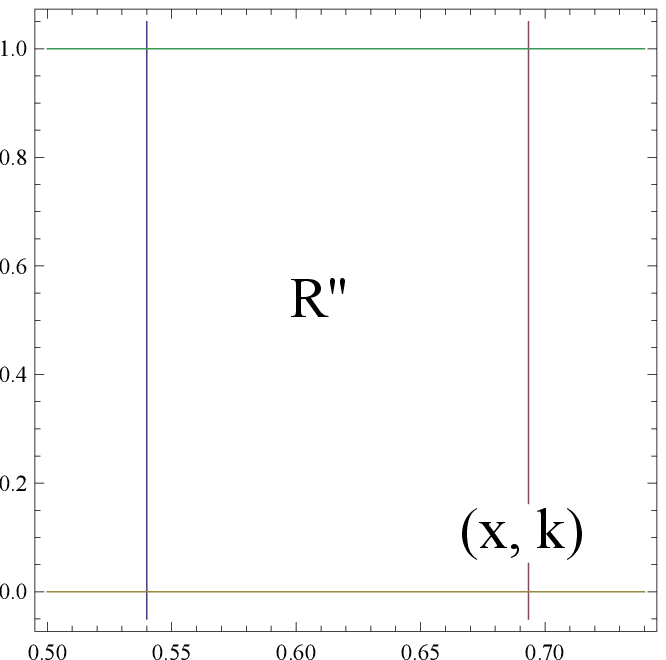}
\caption{The domains $\mathfrak{R}', \mathfrak{R}''$ of new variables}
\label{fig 8}
\end{figure}

We compute the evaluation at ${\pi \over 4}$
\begin{eqnarray*}
f_q({\pi \over 4})&=&w^t \mathcal{H}w+2\sqrt{2c-2c_0}({1 \over \sqrt{2}} (3r-{9 \over r^2})\cos \theta \sin \theta+ {1 \over \sqrt{2}} (-{1 \over r^5}+{2c \over r}))
\\ & &+(2c-2c_0)(-{1 \over 2}-{1 \over 2r^3}+3\sin \theta \cos \theta)
\\ f_q'({\pi \over 4})&=&2\sqrt{2c-2c_0}({1 \over \sqrt{2}}({9 \over r^2}-3r)\cos \theta \sin \theta)+{1 \over \sqrt{2}}(-{1 \over r^5}+{2c \over r}))+(2c-2c_0)(-{1 \over r^3}+{4c \over r^2}-3)
\end{eqnarray*}
of the functions $f_q$ and $f_q '$ to express the tangent line $l_q$ at ${\pi \over 4}$ in terms of $q$. For the tangent line
$$l_q(t)=f_q({\pi \over 4})+f_q'({\pi \over 4})(t-{\pi \over 4}),$$
of $f_q$ at ${\pi \over 4}$, we can express the values of $l_q$
$$l_q({\pi \over 4}+1)=w^t \mathcal{H}w+2\sqrt{2c-2c_0}(\sqrt{2}(-{1 \over r^5}+{2c \over r}))+(2c-2c_0)(-{7 \over 2}-{3 \over 2r^3}+{4c \over r^2}+3\sin\theta\cos\theta)$$
at ${\pi \over 4}+1$ in terms of $r, \theta, c$ explicitly.

\begin{prop 4.10.}
The inequality $l_q ({\pi \over 4}+1)>0$ holds for every $q \in \mathfrak{R}^{+} \backslash B_{0.54}(0)$.
\end{prop 4.10.}

\begin{proof} \hypertarget{pf of prop 4.10}
First, we note that we can express $w^t\mathcal{H}w$ in terms of $r, c$ using the equation $3r^2 \cos^2 \theta +{2 \over r}-2c=0 \iff \cos^2 \theta={{2cr-2} \over 3r^3}$ where $(r, \theta)$ is the polar coordinate system for $q$, namely $q_1=r \cos \theta, q_2=r \sin \theta$. 
\begin{eqnarray*}
& &w^t \mathcal{H}w
\\ &=&{1\over r^7}-{5 \over r^4} \cos^2 \theta-{2\over r^4} \sin^2 \theta +{3\over r} \cos^2 \theta-{27 \over r} \cos^2 \theta \sin^2 \theta +9r^2 \cos^2 \theta
\\ &=&{1 \over r^7}-{2 \over r^4}-{3 \over r^4}({2cr-2 \over 3r^3})+{3 \over r}({2cr-2 \over 3r^3})-{27 \over r}({2cr-2 \over 3r^3})(1-{2cr-2 \over 3r^3})+9r^2({2cr-2 \over 3r^3})
\\ &=&{1 \over r^7}-{2 \over r^4}-{3 \over r^4}({2c_0 r-2 \over 3r^3})+{3 \over r}({2c_0 r-2 \over 3r^3})-{27 \over r}({2c_0 r-2 \over 3r^3})(1-{2c_0 r-2 \over 3r^3})+9r^2({2c_0 r-2 \over 3r^3})
\\ &-&{3\over r^4}({2c-2c_0 \over 3r^2})+{3 \over r}({2c-2c_0 \over 3r^2})-27[({2c-2c_0 \over 3r^2})-(({2cr-2 \over 3r^3})^2-({2c_0 r-2 \over 3r^3})^2)]+9r^2({2c-2c_0 \over 3r^2})
\\ &=&{15 \over r^7}-{39\sqrt[3]{3} \over r^6}+{27\sqrt[3]{9} \over r^5}+{14 \over r^4}-{24\sqrt[3]{3} \over r^3}-{6 \over r}+9\sqrt[3]{3}
\\ & &+(2c-2c_0)(-{13 \over r^6}+{18\sqrt[3]{3} \over r^5}-{8 \over r^3}+3)+(2c-2c_0)^2({3 \over r^5})
\end{eqnarray*}

We have to see that  $w^t \mathcal{H}w+2\sqrt{2c-2c_0}(\sqrt{2}(-{1 \over r^5}+{2c \over r}))+(2c-2c_0)(-{7 \over 2}-{3 \over 2r^3}+{4c \over r^2})>0$

By inserting the last computation and using $c_0={3^{4 \over 3} \over 2}$, we get the following estimate.
\begin{eqnarray*}
& &w^t \mathcal{H}w+2\sqrt{2c-2c_0}(\sqrt{2}(-{1 \over r^5}+{2c \over r}))+(2c-2c_0)(-{7 \over 2}-{3 \over 2r^3}+{4c \over r^2})
\\ &=&{15 \over r^7}-{39\sqrt[3]{3} \over r^6}+{27\sqrt[3]{9} \over r^5}+{14 \over r^4}-{24\sqrt[3]{3} \over r^3}-{6 \over r}+9\sqrt[3]{3}
\\ & &+(2c-2c_0)(-{13 \over r^6}+{18\sqrt[3]{3} \over r^5}-{8 \over r^3}+3)+(2c-2c_0)^2({3 \over r^5})
\\ & &+2\sqrt{2c-2c_0}(\sqrt{2}(-{1\over r^5}+{3\sqrt[3]{3}\over r}))+2(2c-2c_0)^{3\over2}({\sqrt{2}\over r})
\\ & &+(2c-2c_0)(-{3 \over 2r^3}-{7 \over 2}+{6\sqrt[3]{3}\over r^2})+(2c-2c_0)^2 ({2 \over r^2})
\\ &=&{15 \over r^7}-{39\sqrt[3]{3} \over r^6}+{27\sqrt[3]{9} \over r^5}+{14 \over r^4}-{24\sqrt[3]{3} \over r^3}-{6 \over r}+9\sqrt[3]{3}+2\sqrt{2c-2c_0}(\sqrt{2}(-{1\over r^5}+{3\sqrt[3]{3}\over r}))
\\ & &+(2c-2c_0)(-{13 \over r^6}+{18\sqrt[3]{3} \over r^5}-{19 \over 2r^3}+{6\sqrt[3]{3} \over r^2}-{1 \over 2})
\\ & &+2(2c-2c_0)^{3\over2}({\sqrt{2}\over r})+(2c-2c_0)^2({3 \over r^5}+{2 \over r^2})
\\ &\ge&{15 \over r^7}-{39\sqrt[3]{3} \over r^6}+{27\sqrt[3]{9} \over r^5}+{14 \over r^4}-{24\sqrt[3]{3} \over r^3}-{6 \over r}+9\sqrt[3]{3}+2\sqrt{2c-2c_0}(\sqrt{2}(-{1\over r^5}+{3\sqrt[3]{3}\over r}))
\\ & &+(2c-2c_0)(-{13 \over r^6}+{18\sqrt[3]{3} \over r^5}-{19 \over 2r^3}+{6\sqrt[3]{3} \over r^2}-{1 \over 2})
\end{eqnarray*}
Therefore, it suffices to prove the following inequality
\begin{equation}
\label{ineq1 in prop 4.10}
\begin{split}
&  {15 \over r^7}-{39\sqrt[3]{3} \over r^6}+{27\sqrt[3]{9} \over r^5}+{14 \over r^4}-{24\sqrt[3]{3} \over r^3}-{6 \over r}+9\sqrt[3]{3}+2\sqrt{2c-2c_0}(\sqrt{2}(-{1\over r^5}+{3\sqrt[3]{3}\over r}))
\\  & +(2c-2c_0)(-{13 \over r^6}+{18\sqrt[3]{3} \over r^5}-{19 \over 2r^3}+{6\sqrt[3]{3} \over r^2}-{1 \over 2}) > 0
\end{split}
\end{equation}

We will use the variables $(x,k)$ in Lemma \ref{Lem 5.2} which have the relation of $x:=r, y:=\cos^2 \theta, y={1+3k(3^{1\over3}x-1) \over 1+k(3x^3-1)}$. Note that the following identities.
\begin{equation}
\label{id1 in prop 4.10}
\begin{split}
& 2c-2c_0=3x^2 y+{2 \over x}-3^{4 \over 3}
\\ & \quad = 3x^2 ({1+3k(3^{1\over3}x-1) \over 1+k(3x^3-1)})+{2 \over x}-3^{4 \over 3}={1-k \over 1+k(3x^3-1)}(3x^2+{2 \over x}-3^{4 \over 3})
\\ & \quad ={1-k \over 1+k(3x^3-1)}(3^{-1\over 3}-x)^2(3+{2 \cdot 3^{2\over3}\over x})
\end{split}
\end{equation}
Using the above identities, the inequality (\ref{ineq1 in prop 4.10}) can be written as follows.
\begin{eqnarray*}
& & {15 \over x^7}-{39\sqrt[3]{3} \over x^6}+{27\sqrt[3]{9} \over x^5}+{14 \over x^4}-{24\sqrt[3]{3} \over x^3}-{6 \over x}+9\sqrt[3]{3}
\\ & &+2\sqrt{{1-k \over 1+k(3x^3-1)}(3^{-1\over 3}-x)^2(3+{2 \cdot 3^{2\over3}\over x})}(\sqrt{2}(-{1\over x^5}+{3\sqrt[3]{3}\over x}))
\\ & &+({1-k \over 1+k(3x^3-1)}(3^{-1\over 3}-x)^2(3+{2 \cdot 3^{2\over3}\over x}))(-{13 \over x^6}+{18\sqrt[3]{3} \over x^5}-{19 \over 2x^3}+{6\sqrt[3]{3} \over x^2}-{1 \over 2}) > 0
\end{eqnarray*}
Since we have the following decompositions
\begin{equation}
\label{decom1 in prop 4.10}
{15 \over x^7}-{39\sqrt[3]{3} \over x^6}+{27\sqrt[3]{9} \over x^5}+{14 \over x^4}-{24\sqrt[3]{3} \over x^3}-{6 \over x}+9\sqrt[3]{3}=(3^{-1 \over 3}-x)^2 ({15 \cdot 3^{2\over3} \over x^7}-{27 \over x^6}-{18 \cdot 3^{1 \over 3} \over x^5}+{5 \cdot 3^{2 \over 3} \over x^4}+{12 \over x^3}+{9 \cdot 3^{1\over 3} \over x^2}),
\end{equation}
\begin{equation}
\label{decom2 in prop 4.10}
-{1\over x^5}+{3\sqrt[3]{3}\over x}=-(3^{-1 \over 3}-x)({3^{1\over3} \over x^5}+{3^{2 \over 3} \over x^4}+{3 \over x^3}+{3^{4\over3} \over x^2}),
\end{equation}
we can factor out the term $(3^{-1 \over 3}-x)^2$. Then inequality (\ref{ineq1 in prop 4.10}) is equivalent to the following inequality
\begin{equation}
\label{ineq2 in prop 4.10}
\begin{split}
& {15 \cdot 3^{2\over3} \over x^7}-{27 \over x^6}-{18 \cdot 3^{1 \over 3} \over x^5}+{5 \cdot 3^{2 \over 3} \over x^4}+{12 \over x^3}+{9 \cdot 3^{1\over 3} \over x^2}
\\ & -2\sqrt{2} \sqrt{{1-k \over 1+k(3x^3-1)}(3+{2 \cdot 3^{2\over3}\over x})}({3^{1\over3} \over x^5}+{3^{2 \over 3} \over x^4}+{3 \over x^3}+{3^{4\over3} \over x^2})
\\ & +({1-k \over 1+k(3x^3-1)}(3+{2 \cdot 3^{2\over3}\over x}))(-{13 \over x^6}+{18\sqrt[3]{3} \over x^5}-{19 \over 2x^3}+{6\sqrt[3]{3} \over x^2}-{1 \over 2}) > 0
\end{split}
\end{equation}
We will prove this inequality (\ref{ineq2 in prop 4.10}). We define a degree 2 polynomial
\begin{eqnarray*}
g_x(t)&:=& (-{13 \over x^6}+{18\cdot3^{1\over3} \over x^5}-{19 \over 2x^3}+{6\cdot3^{1\over3} \over x^2}-{1 \over 2})t^2-2\sqrt{2}({3^{1\over3} \over x^5}+{3^{2 \over 3} \over x^4}+{3 \over x^3}+{3^{4\over3} \over x^2})t
\\ & &+({15 \cdot 3^{2\over3} \over x^7}-{27 \over x^6}-{18 \cdot 3^{1 \over 3} \over x^5}+{5 \cdot 3^{2 \over 3} \over x^4}+{12 \over x^3}+{9 \cdot 3^{1\over 3} \over x^2}).
\end{eqnarray*}
in variable $t$. The coefficients of $g_x$ are functions of $x$. We note that $g_x(\sqrt{{1-k \over 1+k(3x^3-1)}(3+{2 \cdot 3^{2\over3}\over x})})$ is the left hand side of inequality (\ref{ineq2 in prop 4.10}) which we have to show. Thus we want to prove that the inequality
$$g_x(\sqrt{{1-k \over 1+k(3x^3-1)}(3+{2 \cdot 3^{2\over3}\over x})})>0$$
for all $(x,k) \in (0.54, 3^{-1 \over 3}) \times [0, 1)$. We calculate the discriminant $D_x$ of the polynomial $g_x$.
\begin{eqnarray*}
{D_x \over 4}&=& 2({3^{1\over3} \over x^5}+{3^{2 \over 3} \over x^4}+{3 \over x^3}+{3^{4\over3} \over x^2})^2
\\ & &-(-{13 \over x^6}+{18\cdot 3^{1\over3} \over x^5}-{19 \over 2x^3}+{6\cdot 3^{1\over3} \over x^2}-{1 \over 2})({15 \cdot 3^{2\over3} \over x^7}-{27 \over x^6}-{18 \cdot 3^{1 \over 3} \over x^5}+{5 \cdot 3^{2 \over 3} \over x^4}+{12 \over x^3}+{9 \cdot 3^{1\over 3} \over x^2})
\end{eqnarray*}

\begin{figure}
\centering
\includegraphics[]{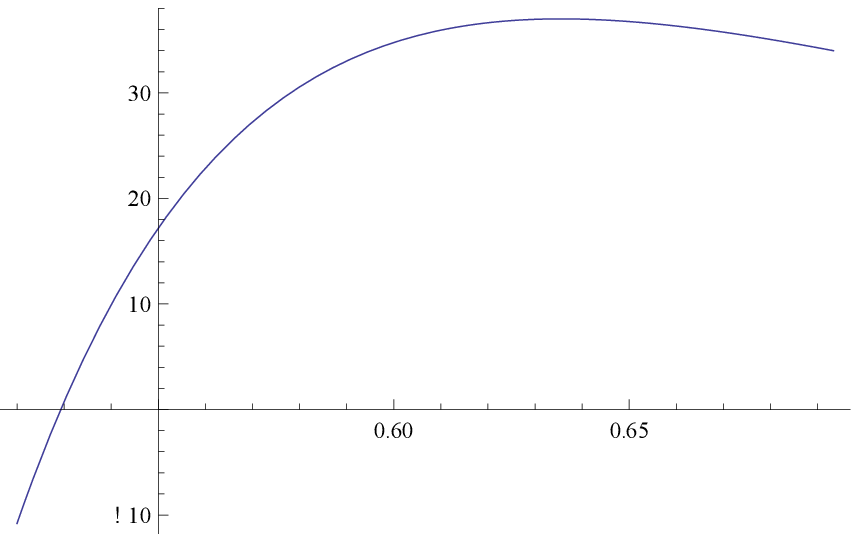}
\caption{Graph of $f_9(x)=-{13 \over x^6}+{18\cdot3^{1\over3} \over x^5}-{19 \over 2x^3}+{6\cdot3^{1\over3} \over x^2}-{1 \over 2}$ shows that it is positive on $[0.54, 3^{-1 \over 3})$.}
\label{fig 9}
\end{figure}

\begin{figure}
\centering
\includegraphics[]{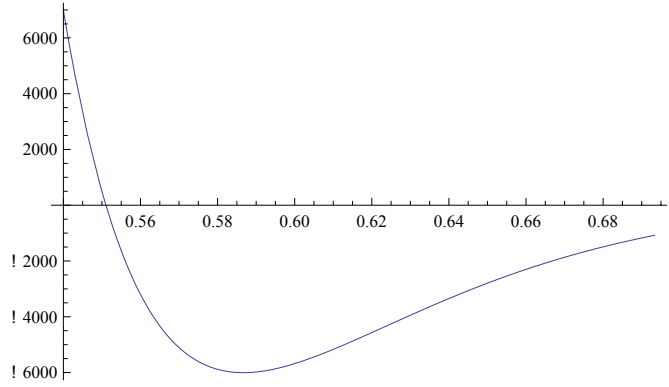}
\caption{Graph of $f_{10}(x)={D_x \over 4}$ shows that it is negative on $[0.56, 3^{-1 \over 3})$.}
\label{fig 10}
\end{figure}

We can see that the coefficient of $t^2$ for $g_x$ is positive for all $x \in (0.54, 3^{-1 \over 3})$, namely $-{13 \over x^6}+{18\cdot3^{1\over3} \over x^5}-{19 \over 2x^3}+{6\cdot3^{1\over3} \over x^2}-{1 \over 2}>0 \textrm{ on } x \in (0.54, 3^{-1 \over 3})$, from its graph in Figure \ref{fig 9} and this discriminant $D_x <0$ for all  $x \in [0.56, 3^{-1 \over 3})$ from the graph of ${D_x \over 4}$ in Figure \ref{fig 10}. This means that degree 2 polynomial $g$ has a positive coefficient for $t^2$ and has no real root for all $x \in [0.56, 3^{-1 \over 3})$. Therefore, we have proven that 
$$g_x(\sqrt{{1-k \over 1+k(3x^3-1)}(3+{2 \cdot 3^{2\over3}\over x})})>0 \textrm{ for all } (x, k) \in [0.56, 3^{-1\over3}) \times [0, 1).$$
The inequality for $(x, k) \in [0.54, 0.56) \times [0, 1)$ is still left. To complete the proof, we note that the possible values of $t$ for the proof satisfy the inequality
$$0 < \sqrt{{1-k \over 1+k(3x^3-1)}(3+{2 \cdot 3^{2\over3}\over x})} < \sqrt{3+{2\cdot 3^{2 \over 3}\over x}}$$
We compute the derivative
$${dg_x \over dt}(\sqrt{3+{2 \cdot 3^{2\over3} \over x}})
=2(-{13 \over x^6}+{18\cdot3^{1\over3} \over x^5}-{19 \over 2x^3}+{6\cdot3^{1\over3} \over x^2}-{1 \over 2})(\sqrt{3+{2 \cdot 3^{2\over3} \over x}})-2\sqrt{2}({3^{1\over3} \over x^5}+{3^{2 \over 3} \over x^4}+{3 \over x^3}+{3^{4\over3} \over x^2})$$
of $g_x$ at $\sqrt{3+{2 \cdot 3^{2\over3} \over x}}$. Then we have that
$${dg_x \over dt}(\sqrt{3+{2 \cdot 3^{2\over3} \over x}})<0 \textrm{ for all } x \in [0.54, 0.56)$$
from the graph of ${dg_x \over dt}(\sqrt{3+{2 \cdot 3^{2\over3} \over x}})$ in Figure \ref{fig 11}.

\begin{figure}
\centering
\includegraphics[]{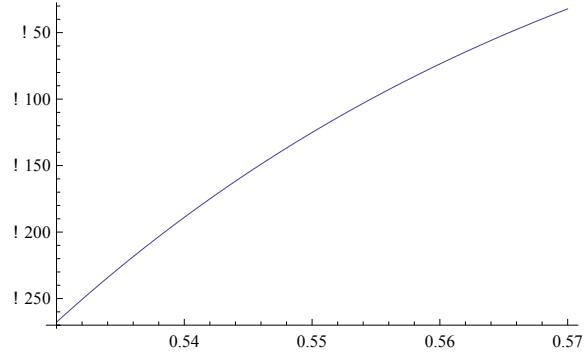}
\caption{Graph of $f_{11}(x)={dg_x \over dt}(\sqrt{3+{2 \cdot 3^{2\over3} \over x}})$ shows that it is negative on [0.54, 0.56].}
\label{fig 11}
\end{figure}

Thus we have the inequality $g_x(t)>g_x(\sqrt{3+{2 \cdot 3^{2\over3} \over x}})$ for all $t \in (0, \sqrt{3+{2 \cdot 3^{2\over3} \over x}})$, when $x \in (0.54, 0.56]$. In particular, we have
$$g_x(\sqrt{{1-k \over 1+k(3x^3-1)}(3+{2 \cdot 3^{2\over3}\over x})}) > g_x(\sqrt{3+{2 \cdot 3^{2\over3} \over x}})$$
for all $(x, k) \in [0.54, 0.56) \times [0, 1)$. Therefore, it is enough to see the following inequality
\begin{eqnarray*}
g_x(\sqrt{3+{2 \cdot 3^{2\over3} \over x}})&:=& (-{13 \over x^6}+{18\cdot3^{1\over3} \over x^5}-{19 \over 2x^3}+{6\cdot3^{1\over3} \over x^2}-{1 \over 2})(3+{2 \cdot 3^{2\over3} \over x})
\\ & &-2\sqrt{2}({3^{1\over3} \over x^5}+{3^{2 \over 3} \over x^4}+{3 \over x^3}+{3^{4\over3} \over x^2})(\sqrt{3+{2 \cdot 3^{2\over3} \over x}})
\\ & &+({15 \cdot 3^{2\over3} \over x^7}-{27 \over x^6}-{18 \cdot 3^{1 \over 3} \over x^5}+{5 \cdot 3^{2 \over 3} \over x^4}+{12 \over x^3}+{9 \cdot 3^{1\over 3} \over x^2}) > 0 \textrm{ on } x \in (0.54, 0.56]
\end{eqnarray*}
in order to prove $g_x(\sqrt{{1-k \over 1+k(3x^3-1)}(3+{2 \cdot 3^{2\over3}\over x})})>0$. This inequality $g_x(\sqrt{3+{2 \cdot 3^{2\over3} \over x}})>0$ for $x \in (0.54, 0.56]$ can be seen from its graph in Figure \ref{fig 12}.

\begin{figure}
\centering
\includegraphics[]{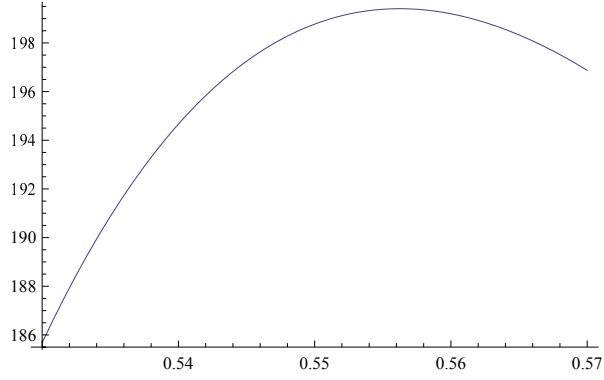}
\caption{Graph of $f_{12}(x)=g(\sqrt{3+{2\cdot3^{2 \over 3} \over x}})$ shows that it is positive on $[0.54, 3^{-1 \over 3})$.}
\label{fig 12}
\end{figure}

Therefore, we have proven
$$g_x(\sqrt{{1-k \over 1+k(3x^3-1)}(3+{2 \cdot 3^{2\over3}\over x})})>0$$
for all $(x, k) \in [0.54, 3^{-1 \over 3})$. This implies inequality (\ref{ineq2 in prop 4.10}) and so (\ref{ineq1 in prop 4.10}). This completes the proof of Proposition \ref{prop 4.10}.
\end{proof}

As in the computation for $l_q({\pi \over 4}+1)$, we can express
$$
l_q({\pi \over 4}-1)=w^t \mathcal{H}w+2\sqrt{2c-2c_0}(\sqrt{2}(3r-{9\over r^2})\cos \theta \sin \theta)+(2c-2c_0)({5 \over 2}+{1 \over {2r^3}}-{4c \over r^2}+3 \sin \theta \cos \theta)
$$
in terms of $r, \theta$.

\begin{prop 4.11.}
The inequality $l_q ({\pi \over 4}-1)>0$ holds for every $q \in \mathfrak{R}^{+} \cap (B_{0.63}(0) \backslash B_{0.54}(0))$.
\end{prop 4.11.}

\begin{proof} \hypertarget{pf of prop 4.11}
Following the computations in the proof of Proposition \ref{prop 4.10}, we can get a lower bound for $l_q({\pi \over 4}-1)$
\begin{eqnarray*}
l_q({\pi \over 4}-1) &\ge& w^t \mathcal{H}w+2\sqrt{2c-2c_0}(\sqrt{2}(3r-{9\over r^2})\cos \theta \sin \theta)+(2c-2c_0)({5 \over 2}+{1 \over 2r^3}-{4c \over r^2})
\\ &=& {15 \over r^7}-{39\cdot3^{1\over3}\over r^6}+{27\cdot3^{2\over3} \over r^5}+{14\over r^4}-{24\cdot3^{1\over3} \over r^3}-{6\over r}+9\cdot3^{1\over3}
\\ & &+2\sqrt{2c-2c_0}(\sqrt{2}(3r-{9 \over r^2})\cos \theta \sin \theta)
\\ & &+(2c-2c_0)(-{13\over r^6}+{6 \cdot3^{4\over3} \over r^5}-{15\over 2r^3}-{2\cdot3^{4\over3} \over r^2}+{11\over2}+3\cos \theta \sin \theta)
\\ & &+(2c-2c_0)^2({3\over r^5}-{2\over r^2})
\\ &\ge& {15 \over r^7}-{39\cdot3^{1\over3}\over r^6}+{27\cdot3^{2\over3} \over r^5}+{14\over r^4}-{24\cdot3^{1\over3} \over r^3}-{6\over r}+9\cdot3^{1\over3}
\\ & &+2\sqrt{2c-2c_0}(\sqrt{2}(3r-{9 \over r^2})\sin \theta)
\\ & &+(2c-2c_0)(-{13\over r^6}+{6 \cdot3^{4\over3} \over r^5}-{15\over 2r^3}-{2\cdot3^{4\over3} \over r^2}+{11\over2})
\end{eqnarray*}
for all $q \in \mathfrak{R}^{+}  \backslash B_{0.54}(0)$. Thus it is enough to prove the following inequality
\begin{equation}
\label{ineq1 in prop 4.11}
\begin{split}
 & {15 \over r^7}-{39\cdot3^{1\over3}\over r^6}+{27\cdot3^{2\over3} \over r^5}+{14\over r^4}-{24\cdot3^{1\over3} \over r^3}-{6\over r}+9\cdot3^{1\over3}
\\  &+2\sqrt{2c-2c_0}(\sqrt{2}(3r-{9 \over r^2})\sin \theta)+(2c-2c_0)(-{13\over r^6}+{6 \cdot3^{4\over3} \over r^5}-{15\over 2r^3}-{2\cdot3^{4\over3} \over r^2}+{11\over2})>0
\end{split}
\end{equation}
for all $q \in \mathfrak{R}^{+} \cap (B_{0.63}(0) \backslash B_{0.54}(0))$ in order to prove Proposition \ref{prop 4.11}. Using the variables in Lemma \ref{Lem 5.2}, we have that
\begin{equation}
\label{id1 in prop 4.11}  
y={1+3k(3^{1\over3}x-1) \over 1+k(3x^3-1)},\quad \sin^2 \theta=1-y={xk \over 1+k(3x^3-1)}(3x^2+{2\over x}-3^{4\over3}).
\end{equation}
For notational convenience, we define $f(x):={15 \over x^7}-{39\cdot3^{1\over3}\over x^6}+{27\cdot3^{2\over3} \over x^5}+{14\over x^4}-{24\cdot3^{1\over3} \over x^3}-{6\over x}+9\cdot3^{1\over3}$. Using the identity (\ref{id1 in prop 4.10}) with above computations, we can write the left hand side of inequality (\ref{ineq1 in prop 4.11})
\begin{eqnarray*}
& & {15 \over r^7}-{39\cdot3^{1\over3}\over r^6}+{27\cdot3^{2\over3} \over r^5}+{14\over r^4}-{24\cdot3^{1\over3} \over r^3}-{6\over r}+9\cdot3^{1\over3}
\\ & &+2\sqrt{2c-2c_0}(\sqrt{2}(3r-{9 \over r^2})\sin \theta)+(2c-2c_0)(-{13\over r^6}+{6 \cdot3^{4\over3} \over r^5}-{15\over 2r^3}-{2\cdot3^{4\over3} \over r^2}+{11\over2})
\\ &=& f(x)+2\sqrt{2}{\sqrt{x}\sqrt{k-k^2} \over 1+k(3x^3-1)}(3x^2+{2\over x}-3^{4\over3})(3x-{9 \over x^2})
\\ & &+{1-k \over 1+k(3x^3-1)}(3x^2+{2\over x}-3^{4\over3})(-{13\over x^6}+{6 \cdot3^{4\over3} \over x^5}-{15\over 2x^3}-{2\cdot3^{4\over3} \over x^2}+{11\over2})
\end{eqnarray*}
in terms of $x, k$. We note that ${1-k \over 1+k(3x^3-1)}$ decreases as $k$ increases for fixed $x$. We compute the partial derivative of ${\sqrt{k-k^2} \over 1+k(3x^3-1)}$ with respect to $k$
$${\partial \over \partial k}({\sqrt{k-k^2} \over 1+k(3x^3-1)})={1-k(3x^3+1) \over 2\sqrt{k-k^2}(1+k(3x^3-1))^2}.$$
One can easily see that ${\sqrt{k-k^2} \over 1+k(3x^3-1)}$ attain its maximum at $k={1\over 3x^3+1}>{1\over 2}$. Moreover, ${\sqrt{k-k^2} \over 1+k(3x^3-1)}$ increases for $k<{1\over 3x^3+1}$ and decreases for $k>{1\over 3x^3+1}$ with respect to $k$ when we fix the other variable $x$. We recall the decompositions (\ref{decom1 in prop 4.10}) in the \hyperlink{pf of prop 4.10}{proof of Proposition 4.10}. Then we can factor out the term $(3^{-1 \over 3}-x)^2$ from inequality (\ref{ineq1 in prop 4.11}). With these notations and discussions, we will prove the following inequality
\begin{equation}
\label{ineq2 in prop 4.11}
\begin{split}
& g(x)+2\sqrt{2}{\sqrt{x}\sqrt{k-k^2} \over 1+k(3x^3-1)}(3+{{2 \cdot 3^{2 \over 3}} \over x})(3x-{9 \over x^2})
\\ & +{1-k \over 1+k(3x^3-1)}(3+{{2 \cdot 3^{2 \over 3}} \over x})(-{13\over x^6}+{6 \cdot3^{4\over3} \over x^5}-{15\over 2x^3}-{2\cdot3^{4\over3} \over x^2}+{11\over2})>0
\end{split}
\end{equation}
for all $(x, k) \in (0.54, 0.63) \times [0, 1]$ where $g(x)=({15 \cdot 3^{2\over3} \over x^7}-{27 \over x^6}-{18 \cdot 3^{1 \over 3} \over x^5}+{5 \cdot 3^{2 \over 3} \over x^4}+{12 \over x^3}+{9 \cdot 3^{1\over 3} \over x^2})$. This is equivalent with inequality (\ref{ineq1 in prop 4.11}).

The strategy for the proof of (\ref{ineq2 in prop 4.11}) can be described as follows.
\\ \\ 
$\bold{1.}$ We divide the region into several cases in terms of $k$. 
\begin{eqnarray*}
& & \hyperlink{case 1}{\bold{Case \ \ 1)}} \ \ 0 \le k \le {1\over3}, \ \ \hyperlink{case 2}{\bold{Case \ \ 2)}} \ \  {1\over3} \le k \le {2\over3}, \ \ \hyperlink{case 3}{\bold{Case \ \ 3)}} \ \ {2\over3} \le k \le {3\over4}
\\ & & \hyperlink{case 4}{\bold{Case \ \ 4)}} \ \ {3\over4} \le k \le {4\over5}, \ \ \hyperlink{case 5}{\bold{Case \ \ 5)}} \ \ {4\over5} \le k \le 1
\end{eqnarray*}
\\ \\
$\bold{2.}$ We make the following estimates
$$ {\sqrt{x}\sqrt{k-k^2} \over 1+k(3x^3-1)}(3+{{2 \cdot 3^{2 \over 3}} \over x}) \le U_i (x), \quad 
m_i (x) \le {1-k \over 1+k(3x^3-1)}(3+{{2 \cdot 3^{2 \over 3}} \over x}) \le M_i (x)$$
for each $\bold{Case \ \ i}$ of $i=1, 2, 3, 4, 5$.
\\ \\
$\bold{3.}$ We construct lower bounds 
\begin{eqnarray*}
& & g(x)+2\sqrt{2}{\sqrt{x}\sqrt{k-k^2} \over 1+k(3x^3-1)}(3+{{2 \cdot 3^{2 \over 3}} \over x})(3x-{9 \over x^2})
\\ & & +{1-k \over 1+k(3x^3-1)}(3+{{2 \cdot 3^{2 \over 3}} \over x})(-{13\over x^6}+{6 \cdot3^{4\over3} \over x^5}-{15\over 2x^3}-{2\cdot3^{4\over3} \over x^2}+{11\over2}) \ge L_i (x) 
\end{eqnarray*}
for each $\bold{Case \ \ i}$ where the function $L_i$ of $x$ is defined by
\begin{eqnarray*}
L_i (x)= \min\{L_i^m (x):=g(x)+2\sqrt{2}U_i(x)(3x-{9\over x^2})+m_i(x)(-{13\over x^6}+{6 \cdot3^{4\over3} \over x^5}-{15\over 2x^3}-{2\cdot3^{4\over3} \over x^2}+{11\over2}),
\\ L_i^M (x):=g(x)+2\sqrt{2}U_i(x)(3x-{9\over x^2})+M_i(x)(-{13\over x^6}+{6 \cdot3^{4\over3} \over x^5}-{15\over 2x^3}-{2\cdot3^{4\over3} \over x^2}+{11\over2})\}.
\end{eqnarray*}
\\ \\
$\bold{4.}$ We prove the inequality $L_i(x)>0$ for all $x \in [0.54, 0.63]$ and $i=1, 2, 3, 4, 5$ by showing the following inequalities
\begin{eqnarray*}
L_i^m (x)=g(x)+2\sqrt{2}U_i(x)(3x-{9\over x^2})+m_i(x)(-{13\over x^6}+{6 \cdot3^{4\over3} \over x^5}-{15\over 2x^3}-{2\cdot3^{4\over3} \over x^2}+{11\over2})>0
\\ L_i^M (x)=g(x)+2\sqrt{2}U_i(x)(3x-{9\over x^2})+M_i(x)(-{13\over x^6}+{6 \cdot3^{4\over3} \over x^5}-{15\over 2x^3}-{2\cdot3^{4\over3} \over x^2}+{11\over2})>0
\end{eqnarray*}
for all $x \in [0.54, 0.63]$, respectively.
\\ We will use the graph of each of functions $L_i^m, L_i^M$ to show $L_i(x)>0$ for each $i=1, 2, 3, 4, 5$. 
\\ \\

\hypertarget{case 1}{$\bold{Case \ \ 1)}$} $0 \le k \le {1\over3}$

For each fixed $x \in [0.54, 0.63]$, the term ${\sqrt{x}\sqrt{k-k^2} \over 1+k(3x^3-1)}(3+{{2 \cdot 3^{2 \over 3}} \over x})$ attains its maximum at $k={1\over3}$ for $k \in [0, {1 \over 3}]$ and the maximum value is given by the function $U_1(x)={\sqrt{2x} \over 3x^3+2}(3+{{2 \cdot 3^{2 \over 3}} \over x})$ of $x$. The term ${1-k \over 1+k(3x^3-1)}(3+{{2 \cdot 3^{2 \over 3}} \over x})$ has the value between its value $m_1(x)={2 \over 3x^3+2}(3+{{2 \cdot 3^{2 \over 3}} \over x})$ at ${1 \over 3}$ and its value $M_1(x)=(3+{{2 \cdot 3^{2 \over 3}} \over x})$ at $k=0$ for $k \in [0, {1 \over 3}]$. It suffices to show that the functions
\begin{eqnarray*}
L_1^m (x) &=& g(x)+2\sqrt{2}{\sqrt{2x} \over 3x^3+2}(3+{{2 \cdot 3^{2 \over 3}} \over x})(3x-{9\over x^2})
\\ & &+{2 \over 3x^3+2}(3+{{2 \cdot 3^{2 \over 3}} \over x})(-{13\over x^6}+{6 \cdot3^{4\over3} \over x^5}-{15\over 2x^3}-{2\cdot3^{4\over3} \over x^2}+{11\over2}),
\\ L_1^M (x) &=&g(x)+2\sqrt{2}{\sqrt{2x} \over 3x^3+2}(3+{{2 \cdot 3^{2 \over 3}} \over x})(3x-{9\over x^2})
\\ & &+(3+{{2 \cdot 3^{2 \over 3}} \over x})(-{13\over x^6}+{6 \cdot3^{4\over3} \over x^5}-{15\over 2x^3}-{2\cdot3^{4\over3} \over x^2}+{11\over2})
\end{eqnarray*}
of $x$ are positive for all $x \in [0.54, 0.63]$. We can see that the inequalities $L_1^m (x)>0, L_1^M (x)>0$ hold for all $x \in [0.54, 0.63]$ from their graphs in Figure \ref{fig 13}.

\begin{figure}
\centering
\includegraphics[]{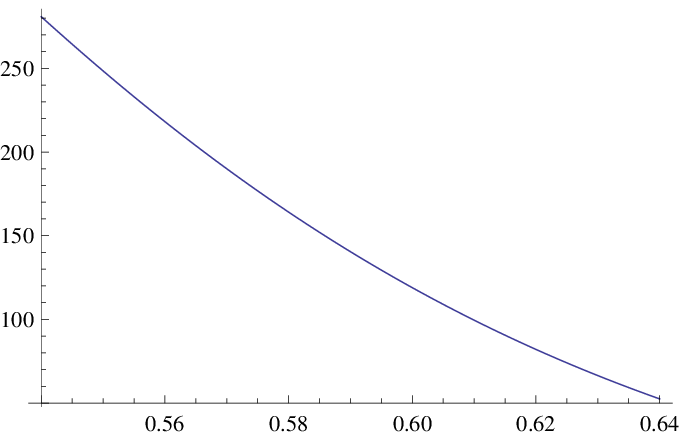}
\includegraphics[]{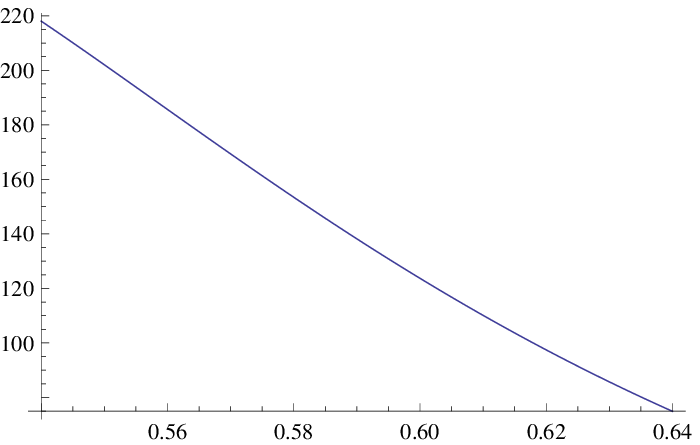}
\caption{Case 1) Graphs of $f^m_{13}(x)=L_1^m (x)$(left) and $f^M_{13}(x)=L_1^M (x)$(right) show that they are positive on $[0.54, 0.63]$.}
\label{fig 13}
\end{figure}

\hypertarget{case 2}{$\bold{Case \ \ 2)}$} ${1\over 3} \le k \le {2 \over 3}$

For each fixed $x \in [0.54, 0.63]$, the term ${\sqrt{x}\sqrt{k-k^2} \over 1+k(3x^3-1)}(3+{{2 \cdot 3^{2 \over 3}} \over x})$ attains its maximum at $k={1\over 3x^3+1}$ among $k \in [{1 \over 3}, {2 \over 3}]$ and the maximum value is the function $U_2 (x)={1 \over 2 \sqrt{3}x}(3+{{2 \cdot 3^{2 \over 3}} \over x})$ of $x$. The term ${1-k \over 1+k(3x^3-1)}(3+{{2 \cdot 3^{2 \over 3}} \over x})$ has the value between $m_2 (x)={1 \over 6x^3+1}(3+{{2 \cdot 3^{2 \over 3}} \over x})$ and $M_2 (x)={2 \over 3x^3+2}(3+{{2 \cdot 3^{2 \over 3}} \over x})$ for $k \in [{1 \over 3}, {2 \over 3}]$. It suffices to show that the functions
\begin{eqnarray*}
L_2^m (x) &=& g(x)+{\sqrt{2} \over \sqrt{3}x}(3+{{2 \cdot 3^{2 \over 3}} \over x})(3x-{9\over x^2})
\\ & &+{1 \over 6x^3+1}(3+{{2 \cdot 3^{2 \over 3}} \over x})(-{13\over x^6}+{6 \cdot3^{4\over3} \over x^5}-{15\over 2x^3}-{2\cdot3^{4\over3} \over x^2}+{11\over2}),
\\ L_2^M (x) &=& g(x)+{\sqrt{2} \over \sqrt{3}x}(3+{{2 \cdot 3^{2 \over 3}} \over x})(3x-{9\over x^2})
\\ & &+{2 \over 3x^3+2}(3+{{2 \cdot 3^{2 \over 3}} \over x})(-{13\over x^6}+{6 \cdot3^{4\over3} \over x^5}-{15\over 2x^3}-{2\cdot3^{4\over3} \over x^2}+{11\over2})
\end{eqnarray*}
of $x$ are positive for all $x \in [0.54, 0.63]$. We can see that the inequalities $L_2^m (x)>0, L_2^M (x)>0$ hold for all $x \in [0.54, 0.63]$ from their graphs in Figure \ref{fig 14}.

\begin{figure}
\centering
\includegraphics[]{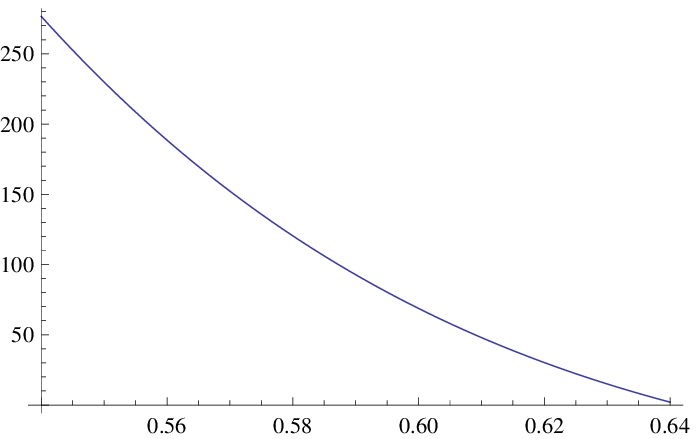}
\includegraphics[]{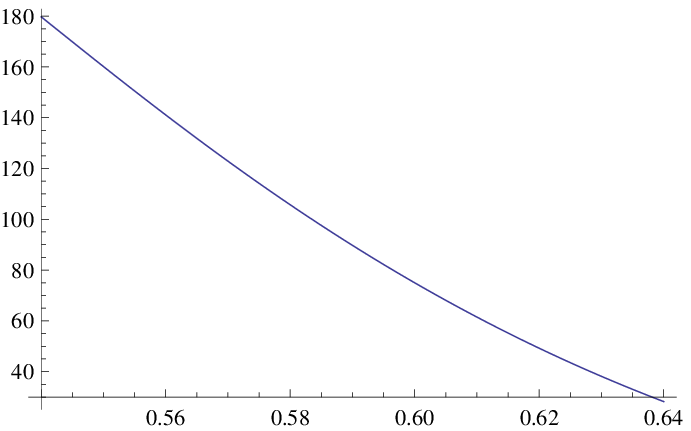}
\caption{Case 2) Graphs of $f^m_{14}(x)=L_2^m (x)$(left) and $f^M_{14}(x)=L_2^M (x)$(right) show that they are positive on $ [0.54, 0.63]$.}
\label{fig 14}
\end{figure}

\hypertarget{case 3}{$\bold{Case \ \ 3)}$} ${2\over 3} \le k \le {3 \over 4}$

For each fixed $x \in [0.54, 0.63]$, the term ${\sqrt{x}\sqrt{k-k^2} \over 1+k(3x^3-1)}(3+{{2 \cdot 3^{2 \over 3}} \over x})$ attains its maximum at $k={2\over 3}$ among $k \in [{2 \over 3}, {3 \over 4}]$ and the maximum value is the function $U_3 (x)={\sqrt{2x} \over 6x^3+1}(3+{{2 \cdot 3^{2 \over 3}} \over x})$ of $x$. The term ${1-k \over 1+k(3x^3-1)}(3+{{2 \cdot 3^{2 \over 3}} \over x})$ has the value between $m_3 (x)={1 \over 9x^3+1}(3+{{2 \cdot 3^{2 \over 3}} \over x})$ and $M_3 (x)={1 \over 6x^3+1}(3+{{2 \cdot 3^{2 \over 3}} \over x})$ for $k \in [{2 \over 3}, {3 \over 4}]$. It suffices to show that the functions
\begin{eqnarray*}
L_3^m (x)&= &g(x)+2\sqrt{2}{\sqrt{2x} \over 6x^3+1}(3+{{2 \cdot 3^{2 \over 3}} \over x})(3x-{9\over x^2})
\\ & &+{1 \over 9x^3+1}(3+{{2 \cdot 3^{2 \over 3}} \over x})(-{13\over x^6}+{6 \cdot3^{4\over3} \over x^5}-{15\over 2x^3}-{2\cdot3^{4\over3} \over x^2}+{11\over2}),
\\ L_3^M (x)&= &g(x)+2\sqrt{2}{\sqrt{2x} \over 6x^3+1}(3+{{2 \cdot 3^{2 \over 3}} \over x})(3x-{9\over x^2})
\\ & &+{1 \over 6x^3+1}(3+{{2 \cdot 3^{2 \over 3}} \over x})(-{13\over x^6}+{6 \cdot3^{4\over3} \over x^5}-{15\over 2x^3}-{2\cdot3^{4\over3} \over x^2}+{11\over2})
\end{eqnarray*}
of $x$ are positive for all $x \in [0.54, 0.63]$. We can see that the inequalities $L_3^m (x)>0, L_3^M (x)>0$ hold for all $x \in [0.54, 0.63]$ from their graphs in Figure \ref{fig 15}.

\begin{figure}
\centering
\includegraphics[]{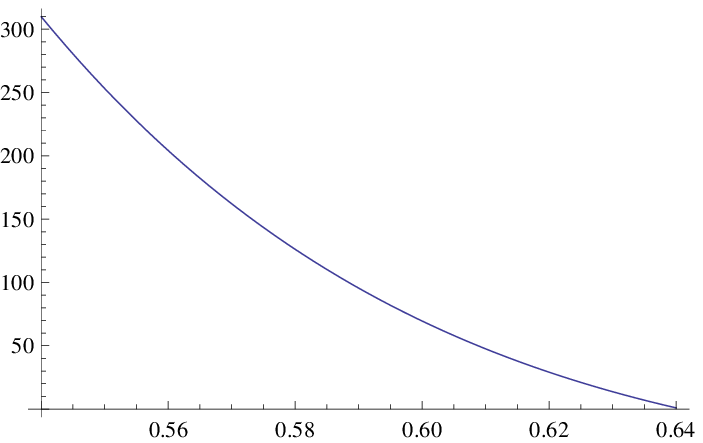}
\includegraphics[]{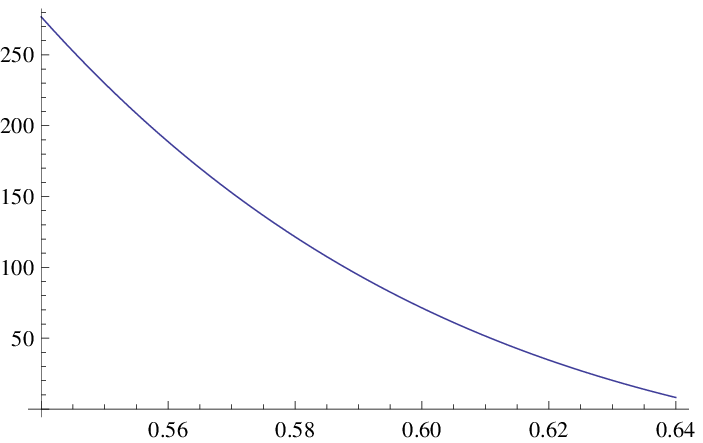}
\caption{Case 3) Graphs of $f^m_{15}(x)=L_3^m (x)$(left) and $f^M_{15}(x)=L_3^M (x)$(right) show that they are positive on $ [0.54, 0.63]$.}
\label{fig 15}
\end{figure}

\hypertarget{case 4}{$\bold{Case \ \ 4)}$} ${3\over 4} \le k \le {4 \over 5}$

For each fixed $x \in [0.54, 0.63]$, the term ${\sqrt{x}\sqrt{k-k^2} \over 1+k(3x^3-1)}(3+{{2 \cdot 3^{2 \over 3}} \over x})$ attains its maximum at $k={3\over 4}$ among $k \in [{3 \over 4}, {4 \over 5}]$ and the maximum value is the function $U_4 (x)={\sqrt{3x} \over 9x^3+1}(3+{{2 \cdot 3^{2 \over 3}} \over x})$ of $x$. The term ${1-k \over 1+k(3x^3-1)}(3+{{2 \cdot 3^{2 \over 3}} \over x})$ has the value between $m_4 (x)={1 \over 12x^3+1}(3+{{2 \cdot 3^{2 \over 3}} \over x})$ and $M_4 (x)={1 \over 9x^3+1}(3+{{2 \cdot 3^{2 \over 3}} \over x})$ for $k \in [{3 \over 4}, {4 \over 5}]$. It suffices to show that the functions
\begin{eqnarray*}
L_4^m (x)&=&g(x)+2\sqrt{2}{\sqrt{3x} \over 9x^3+1}(3+{{2 \cdot 3^{2 \over 3}} \over x})(3x-{9\over x^2})
\\ & &+{1 \over 12x^3+1}(3+{{2 \cdot 3^{2 \over 3}} \over x})(-{13\over x^6}+{6 \cdot3^{4\over3} \over x^5}-{15\over 2x^3}-{2\cdot3^{4\over3} \over x^2}+{11\over2}),
\\ L_4^M (x)&=&g(x)+2\sqrt{2}{\sqrt{3x} \over 9x^3+1}(3+{{2 \cdot 3^{2 \over 3}} \over x})(3x-{9\over x^2})
\\ & &+{1 \over 9x^3+1}(3+{{2 \cdot 3^{2 \over 3}} \over x})(-{13\over x^6}+{6 \cdot3^{4\over3} \over x^5}-{15\over 2x^3}-{2\cdot3^{4\over3} \over x^2}+{11\over2})
\end{eqnarray*}
of $x$ are positive for all $x \in [0.54, 0.63]$. We can see that the inequalities $L_4^m (x)>0, L_4^M (x)>0$ hold for all $x \in [0.54, 0.63]$ from their graphs in Figure \ref{fig 16}.

\begin{figure}
\centering
\includegraphics[]{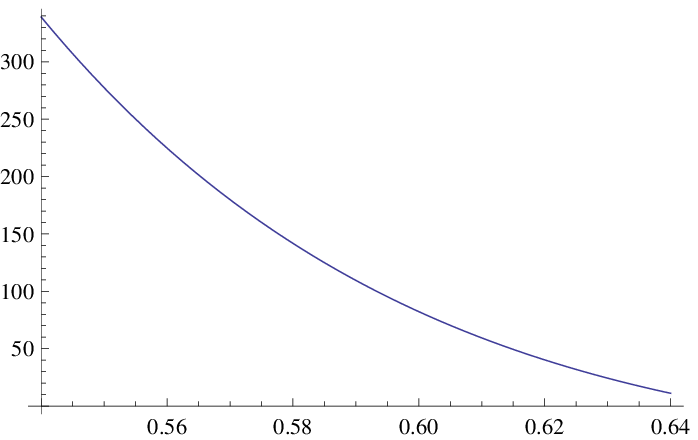}
\includegraphics[]{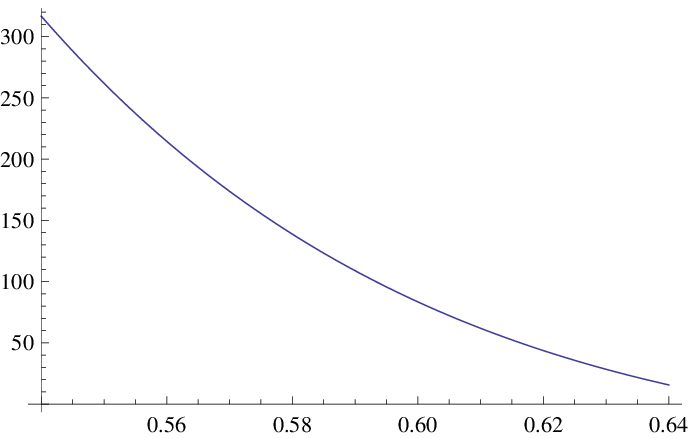}
\caption{Case 4) Graphs of $f^m_{16}(x)=L_4^m (x)$(left) and $f^M_{16}(x)=L_4^M (x)$(right) show that they are positive on $ [0.54, 0.63]$.}
\label{fig 16}
\end{figure}

\hypertarget{case 5}{$\bold{Case \ \ 5)}$} ${4\over 5} \le k \le 1$

For each fixed $x \in [0.54, 0.63]$, the term ${\sqrt{x}\sqrt{k-k^2} \over 1+k(3x^3-1)}(3+{{2 \cdot 3^{2 \over 3}} \over x})$ attains its maximum at $k={4\over 5}$ among $k \in [{4 \over 5}, 1]$ and the maximum value is the function  $U_5 (x)={2\sqrt{x} \over 12x^3+1}(3+{{2 \cdot 3^{2 \over 3}} \over x})$ of $x$. The term ${1-k \over 1+k(3x^3-1)}(3+{{2 \cdot 3^{2 \over 3}} \over x})$ has the value between $m_5 (x)=0$ and $M_5 (x)={1 \over 12x^3+1}(3+{{2 \cdot 3^{2 \over 3}} \over x})$ for $k \in [{4 \over 5}, 1]$. It suffices to show that the functions
\begin{eqnarray*}
L_5^m (x)&=&g(x)+2\sqrt{2}{2\sqrt{x} \over 12x^3+1}(3+{{2 \cdot 3^{2 \over 3}} \over x})(3x-{9\over x^2}),
\\ L_5^M (x)&=&g(x)+2\sqrt{2}{2\sqrt{x} \over 12x^3+1}(3+{{2 \cdot 3^{2 \over 3}} \over x})(3x-{9\over x^2})
\\ & &+{1 \over 12x^3+1}(3+{{2 \cdot 3^{2 \over 3}} \over x})(-{13\over x^6}+{6 \cdot3^{4\over3} \over x^5}-{15\over 2x^3}-{2\cdot3^{4\over3} \over x^2}+{11\over2})
\end{eqnarray*}
of $x$ are positive for all $x \in [0.54, 0.63]$. We can see that the inequalities $L_4^m (x)>0, L_4^M (x)>0$ hold for all $x \in [0.54, 0.63]$ from their graphs in Figure \ref{fig 17}.

\begin{figure}
\centering
\includegraphics[]{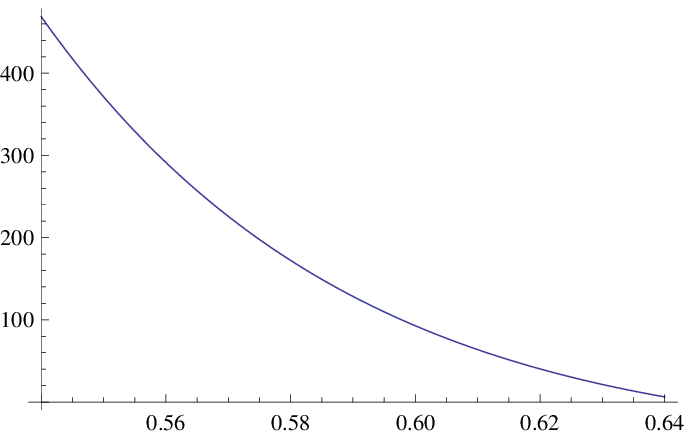}
\includegraphics[]{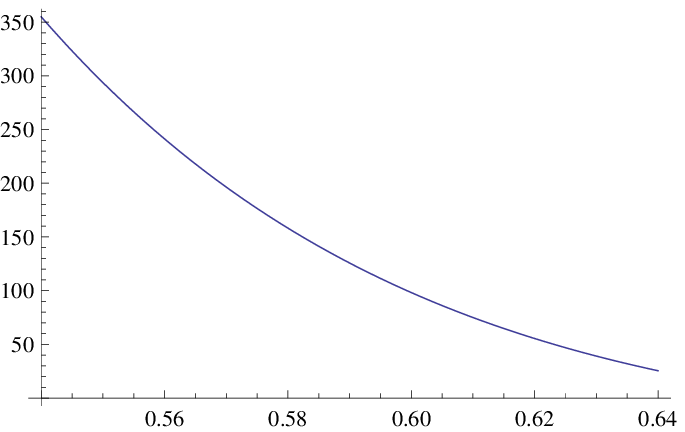}
\caption{Case 5) Graphs of $f^m_{17}(x)=L_5^m (x)$(left) and $f^M_{17}=L_5^M (x)$(right) show that they are positive on $[0.54, 0.63]$.}
\label{fig 17}
\end{figure}

These cases complete the \hyperlink{pf of prop 4.11}{proof of Proposition 4.11}.
\end{proof}

We can prove \hyperlink{step 2}{$\bold{Step \ \ 2}$} by combining the above results.
\\

\begin{proof} [Proof of \hyperlink{step 2}{$\bold{Step \ \ 2}$}]
By Proposition \ref{prop 4.5}, we only need to show that $f_q(\alpha)>0$ for all $q \in \mathfrak{R} \cap (B_{0.63} \backslash B_{0.54}(0))$ and $\alpha \in [0, {\pi \over 2}]$. The function $f_q$ is convex on $[0, {\pi \over 2}]$ by Lemma \ref{Lem 4.9}. This implies that the tangent line $l_q$ of $f_q$ at ${\pi \over 4}$ is below $f_q$. Thus we have
$$\min_{\alpha \in [0, {\pi \over 2}]} {f_q (\alpha)} \ge \min_{\alpha \in [0, {\pi \over 2}]} {l_q (\alpha)}
 \ge \min_{\alpha \in [{\pi \over 4}-1, {\pi \over 4}+1]} {l_q (\alpha)}=\min\{l_q({\pi \over 4}+1), l_q({\pi \over 4}-1)\} \textrm{ for any } q \in \mathfrak{R} \backslash B_{0.54}(0)$$
Proposition \ref{prop 4.10} and \ref{prop 4.11} prove that $\min\{l_q({\pi \over 4}+1), l_q({\pi \over 4}-1)\}>0$ for any $q \in \mathfrak{R} \cap (B_{0.63} \backslash B_{0.54}(0))$. Therefore, $\min_{\alpha \in [0, {\pi \over 2}]} {f_q (\alpha)}>0$ for any $q \in \mathfrak{R} \cap (B_{0.63} \backslash B_{0.54}(0))$. This proves \hyperlink{step 2}{$\bold{Step \ \ 2}$}.
\end{proof}

We will prove \hyperlink{step 3}{$\bold{Step \ \ 3}$} from now. We can consider only when $r>0.63$. We recall that
\begin{eqnarray*}
f_q(\alpha)&:=&(w+s_{\theta+\alpha})^t \mathcal{H} (w+s_{\theta+\alpha})
\\ &=& w^t \mathcal{H}w+2\sqrt{2c-2c_0}(\cos \alpha (3r-{9 \over r^2})\cos \theta \sin \theta+ \sin \alpha (-{1 \over r^5}+{2c \over r}))
\\ & &+(2c-2c_0)(\cos^2 \alpha (1-{2c \over r^2})+\sin^2 \alpha (-{1 \over r^3}+{2c \over r^2}-2)+2 \cos \alpha \sin \alpha(3 \cos \theta \sin \theta))
\end{eqnarray*}

\begin{proof} [Proof of \hyperlink{step 3}{$\bold{Step \ \ 3}$}]
It suffices to show that $f_q(\alpha)>0$ for all $q \in \mathfrak{R} \backslash B_{0.63}(0)$ and $\alpha \in [0, {\pi \over 2}]$ by Proposition \ref{prop 4.5}. We define the function $G(q, \alpha):=f_q(\alpha)$ of $q$ and $\alpha$. We have that
\begin{eqnarray*}
G(q,\alpha)&:=&f_q(\alpha)
\\ &=& w^t \mathcal{H}w+2\sqrt{2c-2c_0}(\cos \alpha (3r-{9 \over r^2})\cos \theta \sin \theta+ \sin \alpha (-{1 \over r^5}+{2c \over r}))
\\ & &+(2c-2c_0)(\cos^2 \alpha (1-{2c \over r^2})+\sin^2 \alpha (-{1 \over r^3}+{2c \over r^2}-2)+2 \cos \alpha \sin \alpha(3 \cos \theta \sin \theta))
\\ &\ge& w^t \mathcal{H}w+2\sqrt{2c-2c_0}(\cos \alpha (3r-{9 \over r^2})\sin \theta+ \sin \alpha (-{1 \over r^5}+{2c \over r}))
\\ & &+(2c-2c_0)(\cos^2 \alpha (1-{2c \over r^2})+\sin^2 \alpha (-{1 \over r^3}+{2c \over r^2}-2))
\\ &=:& E(q, \alpha)
\end{eqnarray*}
for $\alpha \in [0, {\pi \over 2}]$ and $q \in \mathfrak{R} \backslash B_{0.63}(0)$. It suffices to prove that $E(q, \alpha)>0$ for all $q \in \mathfrak{R} \backslash B_{0.63}(0)$ and $\alpha \in [0, {\pi \over 2}]$. We use the variables $x:=r, y:=\cos^2 \theta$ in Lemma \ref{Lem 5.2} and will denote again $E(x, y, \alpha)$ by ignoring the composition of this change of variables. Then we can express the function $E$ in terms of $x, y$ and $\alpha$ as follows.
 \begin{eqnarray*}
E(x, y, \alpha)&=&{15 \over x^7}-{39\sqrt[3]{3} \over x^6}+{27\sqrt[3]{9} \over x^5}+{14 \over x^4}-{24\sqrt[3]{3} \over x^3}-{6 \over x}+9\sqrt[3]{3}
\\ & &+(2c-2c_0)(-{13 \over x^6}+{18\sqrt[3]{3} \over x^5}-{8 \over x^3}+3)+(2c-2c_0)^2({3 \over x^5})
\\ & &+2\sqrt{2c-2c_0}(\cos \alpha (3x-{9 \over x^2})\sqrt{1-y}+ \sin \alpha (-{1 \over x^5}+{2c \over x}))
\\ & &+(2c-2c_0)(\cos^2 \alpha (1-{2c \over x^2})+\sin^2 \alpha (-{1 \over x^3}+{2c \over x^2}-2))
\\ &=& {15 \over x^7}-{39\sqrt[3]{3} \over x^6}+{27\sqrt[3]{9} \over x^5}+{14 \over x^4}-{24\sqrt[3]{3} \over x^3}-{6 \over x}+9\sqrt[3]{3}
\\ & &+(2c-2c_0)(-{13 \over x^6}+{18\sqrt[3]{3} \over x^5}-{8 \over x^3}+3)+(2c-2c_0)^2({3 \over x^5})
\\ & &+2\sqrt{2c-2c_0}(\cos \alpha (3x-{9 \over x^2})\sqrt{1-y}+ \sin \alpha (-{1 \over x^5}+{2c \over x}))
\\ & &+(2c-2c_0)(\cos^2 \alpha (1-{2c \over x^2})+\sin^2 \alpha (-{1 \over x^3}+{2c \over x^2}-2))
\\ &=& {15 \over x^7}-{39\sqrt[3]{3} \over x^6}+{27\sqrt[3]{9} \over x^5}+{14 \over x^4}-{24\sqrt[3]{3} \over x^3}-{6 \over x}+9\sqrt[3]{3}
\\ & &+(2c-2c_0)(-{13 \over x^6}+{18\sqrt[3]{3} \over x^5}-{8 \over x^3}+3)+(2c-2c_0)^2({3 \over x^5})
\\ & &+2\sqrt{2c-2c_0}(\cos \alpha (3x-{9 \over x^2})\sqrt{1-y}+ \sin \alpha (-{1 \over x^5}+{3^{4 \over 3} \over x}))+(2c-2c_0)^{3 \over 2} ({{2 \sin \alpha} \over x})
\\ & &+(2c-2c_0)(\cos^2 \alpha (1-{2c \over x^2})+\sin^2 \alpha (-{1 \over x^3}+{3^{4 \over 3} \over x^2}-2))+(2c-2c_0)^2 (-{{\cos^2 \alpha} \over x^2}+{{\sin^2 \alpha} \over x^2})
\end{eqnarray*}
where $2c=3x^2 y+{2 \over x}$. We can find a lower bound function $D(x, y, \alpha)$
\begin{eqnarray*}
E(x, y, \alpha)&=&{15 \over x^7}-{39\sqrt[3]{3} \over x^6}+{27\sqrt[3]{9} \over x^5}+{14 \over x^4}-{24\sqrt[3]{3} \over x^3}-{6 \over x}+9\sqrt[3]{3}
\\ & & +2\sqrt{2c-2c_0}(\cos \alpha (3x-{9 \over x^2})\sqrt{1-y}+ \sin \alpha (-{1 \over x^5}+{3^{4 \over 3} \over x}))
\\ & & +(2c-2c_0)(-{13 \over x^6}+{18\sqrt[3]{3} \over x^5}-{8 \over x^3}+3+\cos^2 \alpha (1-{3^{4 \over 3} \over x^2})+\sin^2 \alpha (-{1 \over x^3}+{3^{4 \over 3} \over x^2}-2))
\\ & & +(2c-2c_0)^{3 \over 2} ({{2 \sin \alpha} \over x})+(2c-2c_0)^2({3 \over x^5}-{{\cos^2 \alpha} \over x^2}+{{\sin^2 \alpha} \over x^2})
\\ &\ge& {15 \over x^7}-{39\sqrt[3]{3} \over x^6}+{27\sqrt[3]{9} \over x^5}+{14 \over x^4}-{24\sqrt[3]{3} \over x^3}-{6 \over x}+9\sqrt[3]{3}
\\ & & +2\sqrt{2c-2c_0}(\cos \alpha (3x-{9 \over x^2})\sqrt{1-y}+ \sin \alpha (-{1 \over x^5}+{3^{4 \over 3} \over x}))
\\ & & +(2c-2c_0)(-{13 \over x^6}+{18\sqrt[3]{3} \over x^5}-{8 \over x^3}+3+\cos^2 \alpha (1-{3^{4 \over 3} \over x^2})+\sin^2 \alpha (-{1 \over x^3}+{3^{4 \over 3} \over x^2}-2))
\\ &=:& D(x, y, \alpha)
\end{eqnarray*}
of $E(x, y, \alpha)$ by removing the degree 3, 4 terms of $\sqrt{2c-2c_0}$, because the degree 3, 4 terms are positive. Therefore it is enough to show that $D(x, y, \alpha)>0$ for all $q \in \mathfrak{R} \backslash B_{0.63}(0)$ and $\alpha \in [0, {\pi \over 2}]$. We use the variables $(x, k)$ which is given by the change of variables $y={{1+3k(3^{1 \over 3}x-1)} \over {1+k(3x^3-1)}}$ in Lemma \ref{Lem 5.2} and will denote by $D(x, k, \alpha)$ again. Recall that the decompositions (\ref{decom1 in prop 4.10}), (\ref{decom2 in prop 4.10}) in the proof of Proposition \ref{prop 4.10}, then we get the common factor $(3^{-1 \over 3}-x)^2$ in $D(x, k, \alpha)$. Precisely, we have that
\begin{eqnarray*}
D(x, k, \alpha)&=&(3^{-1 \over 3}-x)^2 ({15 \cdot 3^{2\over3} \over x^7}-{27 \over x^6}-{18 \cdot 3^{1 \over 3} \over x^5}+{5 \cdot 3^{2 \over 3} \over x^4}+{12 \over x^3}+{9 \cdot 3^{1\over 3} \over x^2})
\\ & & - \cos \alpha \Big( 2 {\sqrt{x}\sqrt{k-k^2} \over 1+k(3x^3-1)}(3^{-1 \over 3}-x)^2(3+{2 \cdot 3^{2 \over 3} \over x})({9 \over x^2}-3x) \Big)
\\ & & -\sin \alpha \Big( 2 \sqrt{{1-k \over 1+k(3x^3-1)}(3+{2 \cdot 3^{2 \over 3} \over x})} (3^{-1 \over 3}-x)^2 ({3^{1\over3} \over x^5}+{3^{2 \over 3} \over x^4}+{3 \over x^3}+{3^{4\over3} \over x^2}) \Big)
\\ & & +{1-k \over 1+k(3x^3-1)}(3^{-1 \over 3}-x)^2 (3+{2 \cdot 3^{2 \over 3} \over x}) \Big(-{13 \over x^6}+{18\sqrt[3]{3} \over x^5}-{8 \over x^3}+3
\\ & & \qquad \qquad \qquad \qquad \qquad \qquad +\cos^2 \alpha (1-{3^{4 \over 3} \over x^2})+\sin^2 \alpha (-{1 \over x^3}+{3^{4 \over 3} \over x^2}-2) \Big)
\end{eqnarray*}

We can factor out the common factor $(3^{-1 \over 3}-x)^2$ from $D(x,k,\alpha)$. Define $d(x, k, \alpha):={D(x, k, \alpha) \over (3^{-1 \over 3}-x)^2}$. In fact, functions $d, D$ are defined on $\mathfrak{R}'' \times [0, {\pi \over 2}]$ where the domain $\mathfrak{R}''=(0.54, 3^{-1 \over 3}) \times [0, 1)$ for $(x,k)$, we can extend $d$ continuously to the function on $\overline{\mathfrak{R}''} \times [0, {\pi \over 2}]$. If we prove the inequality $d(x, k, \alpha)>0$ on $(\overline{\mathfrak{R}''} \cap \{ x \ge 0.63 \}) \times [0, {\pi \over 2}]$, then we have the inequality $d(x, k, \alpha)>0$ on $(\mathfrak{R}'' \cap \{ x \ge 0.63 \}) \times [0, {\pi \over 2}]$. This implies $D(x, k, \alpha)>0$ for all $(\mathfrak{R}'' \cap \{ x \ge 0.63 \}) \times [0, {\pi \over 2}]$. Therefore, we will prove that $d(x, k, \alpha)>0$ for all $(x, k, \alpha) \in [0.63, 3^{-1 \over 3}] \times [0, 1] \times [0, {\pi \over 2}]$. We abbreviate the terms of $d$ by the functions $C_i, (i=1,2,3,4,5)$ as follows.

\begin{eqnarray*}
d(x, k, \alpha)&=&({15 \cdot 3^{2\over3} \over x^7}-{27 \over x^6}-{18 \cdot 3^{1 \over 3} \over x^5}+{5 \cdot 3^{2 \over 3} \over x^4}+{12 \over x^3}+{9 \cdot 3^{1\over 3} \over x^2})
\\ & & - \cos \alpha \Big( 2 {\sqrt{x}\sqrt{k-k^2} \over 1+k(3x^3-1)}(3+{2 \cdot 3^{2 \over 3} \over x})({9 \over x^2}-3x) \Big)
\\ & & -\sin \alpha \Big( 2 \sqrt{{1-k \over 1+k(3x^3-1)}(3+{2 \cdot 3^{2 \over 3} \over x})} ({3^{1\over3} \over x^5}+{3^{2 \over 3} \over x^4}+{3 \over x^3}+{3^{4\over3} \over x^2}) \Big)
\\ & & +{1-k \over 1+k(3x^3-1)} (3+{2 \cdot 3^{2 \over 3} \over x}) (-{13 \over x^6}+{18\sqrt[3]{3} \over x^5}-{8 \over x^3}+3)
\\ & &+\cos^2 \alpha {1-k \over 1+k(3x^3-1)} (3+{2 \cdot 3^{2 \over 3} \over x})(1-{3^{4 \over 3} \over x^2})
\\ & &+\sin^2 \alpha {1-k \over 1+k(3x^3-1)} (3+{2 \cdot 3^{2 \over 3} \over x})(-{1 \over x^3}+{3^{4 \over 3} \over x^2}-2) )
\\ &=:& C_1(x, k)-C_2(x, k)\cos \alpha -C_3(x, k)\sin \alpha +C_4(x, k)\cos^2 \alpha +C_5(x, k)\sin^2 \alpha.
\end{eqnarray*}
Namely, $C_i$ are the following functions
\begin{eqnarray*}
C_1(x, k)&=&({15 \cdot 3^{2\over3} \over x^7}-{27 \over x^6}-{18 \cdot 3^{1 \over 3} \over x^5}+{5 \cdot 3^{2 \over 3} \over x^4}+{12 \over x^3}+{9 \cdot 3^{1\over 3} \over x^2})
\\ & & +{1-k \over 1+k(3x^3-1)} (3+{2 \cdot 3^{2 \over 3} \over x}) (-{13 \over x^6}+{18\sqrt[3]{3} \over x^5}-{8 \over x^3}+3),
\\ C_2(x, k)&=&2 {\sqrt{x}\sqrt{k-k^2} \over 1+k(3x^3-1)}(3+{2 \cdot 3^{2 \over 3} \over x})({9 \over x^2}-3x),
\\ C_3(x, k)&=& 2 \sqrt{{1-k \over 1+k(3x^3-1)}(3+{2 \cdot 3^{2 \over 3} \over x})} ({3^{1\over3} \over x^5}+{3^{2 \over 3} \over x^4}+{3 \over x^3}+{3^{4\over3} \over x^2}),
\\ C_4(x, k)&=&{1-k \over 1+k(3x^3-1)} (3+{2 \cdot 3^{2 \over 3} \over x})(1-{3^{4 \over 3} \over x^2}),
\\ C_5(x, k)&=& {1-k \over 1+k(3x^3-1)} (3+{2 \cdot 3^{2 \over 3} \over x})(-{1 \over x^3}+{3^{4 \over 3} \over x^2}-2)
\end{eqnarray*}
 of $(x, k)$. We want to prove that the function $d(x, k, \alpha)$ is monotone with respect to $x$ on $[0.63, 3^{-1 \over 3}] \times [0, 1] \times [0, {\pi \over 2}]$. We will prove that ${{\partial d} \over {\partial x}}(x, k, \alpha)<0$ for all $(x, k ,\alpha) \in [0.63, 3^{-1 \over 3}] \times [0, 1] \times [0, {\pi \over 2}]$. Observe that the following Lemma.
\begin{Lem} \label{C2&C3}
The inequalities
$${{\partial C_2} \over {\partial x}}(x, k)<0, \quad {{\partial C_3} \over {\partial x}}(x, k) <0 $$
hold for all $(x, k) \in [0.63, 3^{-1 \over 3}] \times [0, 1]$.
\end{Lem}
\begin{proof}
We can easily see the first inequality
\begin{eqnarray*}
& &{{\partial C_2} \over {\partial x}}(x, k)={(3 \sqrt{k-k^2} \over x^{7 \over 2} (1+k(3 x^3-1))^2} \times
\\ & & \big[(27 x^7 k+30 \cdot 3^{2 \over 3} x^6 k-234 x^4 k-9 x^4-196 \cdot 3^{2 \over 3} x^3 k-2\cdot 3^{2 \over 3} x^3)
\\ & &+(27 x k-27 x)+(30\cdot 3^{2 \over 3} k-30\cdot 3^{2 \over 3})\big]<0
\end{eqnarray*}
from the $(\cdot)$-block and the fact $0 \le k \le 1$. The second inequality
\begin{eqnarray*}
{{\partial C_3} \over {\partial x}}(x, k)&=& 2 \left(-{4 \cdot 3^{2 \over 3} \over x^5}-{5 \cdot 3^{1 \over 3} \over x^6}-{9 \over x^4}-{6 \cdot 3^{1 \over 3} \over x^3} \right) \sqrt{{(1-k) ({2 \cdot 3^{2 \over 3} \over x}+3) \over (1+k(3 x^3-1))}}
\\ &+&{\left({3^{1 \over 3} \over x^5}+{3^{2 \over 3} \over x^4}+{3 \over x^3}+{3^{4 \over 3} \over x^2} \right)
\left( -{\left(9 x^2 (k-k^2)  \left({2 \cdot 3^{2 \over 3} \over x}+3\right)\right) \over (1+k(3 x^3-1))^2}-{2 \cdot 3^{2 \over 3} (1-k) \over x^2 (1+k(3 x^3-1))} \right) \over \sqrt{{(1-k) ({2 \cdot 3^{2 \over 3} \over x}+3) \over (1+k(3 x^3-1))}}}<0
\end{eqnarray*}
is also obvious from the above form. This completes the proof of Lemma \ref{C2&C3}.
\end{proof}
 We consider the estimate 
\begin{eqnarray*}
& & {{\partial d} \over {\partial x}}(x, k, \alpha)
\\ &=& {{\partial C_1} \over {\partial x}}(x, k)-{{\partial C_2} \over {\partial x}}(x, k) \cos \alpha -{{\partial C_3} \over {\partial x}}(x, k) \sin \alpha+{{\partial C_4} \over {\partial x}}(x, k) \cos^2 \alpha +{{\partial C_5} \over {\partial x}}(x, k) \sin^2 \alpha
\\ &\le& {{\partial C_1} \over {\partial x}}(x, k)-{{\partial C_2} \over {\partial x}}(x, k)-{{\partial C_3} \over {\partial x}}(x, k)+\max\{{{\partial C_4} \over {\partial x}}(x, k), {{\partial C_5} \over {\partial x}}(x, k)\}
\\ &=& {{\partial C_1} \over {\partial x}}(x, k)-{{\partial C_2} \over {\partial x}}(x, k)-{{\partial C_3} \over {\partial x}}(x, k)+{{\partial C_4} \over {\partial x}}(x, k) 
\end{eqnarray*}
for ${{\partial d} \over {\partial x}}(x, k, \alpha)$. The inequality is from Lemma \ref{C2&C3}. The last equality follows from the Claim.
\\
Claim: We have the inequality ${{\partial C_4} \over {\partial x}}(x, k) \ge {{\partial C_5} \over {\partial x}}(x, k)$ for all $(x, k) \in [0.63, 3^{-1 \over 3}] \times [0, 1]$.

\begin{proof}[Proof of Claim]
It is enough to show that ${1 \over 1-k} {\partial (C_4-C_5) \over \partial x}(x, k)>0$ for all $(x, k) \in [0.63, 3^{-1 \over 3}] \times [0, 1]$. 
$${C_4-C_5 \over 1-k}={1 \over {1+k(3x^3-1)}}(3+{2 \cdot 3^{2 \over 3} \over x})(3- {2 \cdot 3^{4 \over 3} \over x^2}+{1 \over x^3})$$
We differentiate it by $x$.
\begin{eqnarray*}
{\partial \over \partial x}({C_4-C_5 \over 1-k})&=&{-9kx^2 \over ({1+k(3x^3-1)})^2}(3+{2 \cdot 3^{2 \over 3} \over x})(3- {2 \cdot 3^{4 \over 3} \over x^2}+{1 \over x^3})
\\ & & +{1 \over {1+k(3x^3-1)}}(-{2 \cdot 3^{2 \over 3} \over x^2})(3- {2 \cdot 3^{4 \over 3} \over x^2}+{1 \over x^3})
\\ & & +{1 \over {1+k(3x^3-1)}}(3+{2 \cdot 3^{2 \over 3} \over x})({4 \cdot 3^{4 \over 3} \over x^3}-{3 \over x^4})
\end{eqnarray*}
One can easily see that ${\partial \over \partial x}({C_4-C_5 \over 1-k})>0$ for all $(x, k) \in [0.63, 3^{-1 \over 3}] \times [0, 1]$ from the following inequalities. 
$$3- {2 \cdot 3^{4 \over 3} \over x^2}+{1 \over x^3}<0, \quad {4 \cdot 3^{4 \over 3} \over x^3}-{3 \over x^4}>0 \quad \textrm{for all } x \in [0.63, 3^{-1 \over 3}].$$
\end{proof}

The proof of the inequality
\begin{equation} \label{final ineq}
{{\partial C_1} \over {\partial x}}(x, k)-{{\partial C_2} \over {\partial x}}(x, k)-{{\partial C_3} \over {\partial x}}(x, k)+{{\partial C_4} \over {\partial x}}(x, k) <0
\end{equation}
will be given in Appendix \ref{sec pf of ineq}. Once we prove inequality (\ref{final ineq}), then we obtain ${\partial \over \partial x}d(x, k, \alpha)<0$ for all $(x, k, \alpha) \in [0.63, 3^{-1 \over 3}] \times [0,1] \times [0, {\pi \over 2}]$. Therefore $d(x, k, \alpha) \ge d(3^{-1 \over 3}, k, \alpha)$ for all $(x, k, \alpha) \in [0.63, 3^{-1 \over 3}] \times [0,1] \times [0, {\pi \over 2}]$, in particular the inequality is strict when $x \ne 3^{-1 \over 3}$, and  so it is sufficient to prove that $d(3^{-1 \over 3}, k, \alpha) \ge 0$ for all $(k, \alpha) \in [0,1] \times [0, {\pi \over 2}]$. In fact, we have that
\begin{eqnarray*}
d(3^{-1 \over 3}, k, \alpha)&=&108+216(1-k)-144 \cdot 3^{1 \over 2} \sqrt{k-k^2} \cos \alpha-216\sqrt{1-k} \sin \alpha
\\ & & -72(1-k) \cos^2 \alpha+36(1-k) \sin^2 \alpha
\\ &=& 36[(3\sqrt{1-k}\sin \alpha-1)^2+(\sqrt{2k}-\sqrt{6(1-k)}\cos \alpha)^2].
\end{eqnarray*}
This implies the inequality $d(3^{-1 \over 3}, k, \alpha) \ge 0$ for all $(x, k, \alpha) \in [0.63, 3^{-1 \over 3}] \times [0,1] \times [0, {\pi \over 2}]$ where the equality holds if and only if $k={2 \over 3}$ and $\sin \alpha=\sqrt{{1 \over 3}}, \cos \alpha=\sqrt{{2 \over 3}}$.

We summarize the above results below
\begin{eqnarray*}
& & d(3^{-1 \over 3}, k, \alpha) \ge 0 \textrm{ for all } (k, \alpha) \in [0, 1] \times[0, {\pi \over 2}]
\\ &\Rightarrow& d(x, k, \alpha) > 0 \textrm{ for all } (x, k, \alpha) \in [0.63, 3^{-1 \over 3}) \times [0, 1] \times[0, {\pi \over 2}]
\\ &\Rightarrow& D(x, k, \alpha) > 0 \textrm{ for all } (x, k, \alpha) \in [0.63, 3^{-1 \over 3}) \times [0, 1] \times[0, {\pi \over 2}]
\\ &\Rightarrow& E(x, y, \alpha) > 0 \textrm{ for all } (x, y, \alpha) \in (\mathfrak{R}' \backslash B_{0.63}(0)) \times[0, {\pi \over 2}]
\\ &\Rightarrow& G(q, \alpha) > 0 \textrm{ for all } (q, \alpha) \in (\mathfrak{R} \backslash B_{0.63}(0)) \times[0, {\pi \over 2}]
\\ &\Rightarrow& \min_{\alpha \in [0, {\pi \over 2}]} f_q(\alpha) > 0 \textrm{ for all }  q \in \mathfrak{R} \backslash B_{0.63}(0)
\end{eqnarray*}
This implies that $\min_{|s|^2 \le 3q_1^2+{2 \over |q|}-3^{4/3}} (w(q)+s)^t \mathcal{H}(q) (w(q)+s)>0$ for all $q \in \mathfrak{R} \backslash B_{0.63}(0)$ by Proposition \ref{prop 4.5} and this proves \hyperlink{step 3}{$\bold{Step \ \ 3}$}.
\end{proof}

Therefore, we have proven \hyperlink{step 1}{$\bold{Step \ \ 1, 2, 3}$} and these cover all domain of $q$ for Theorem \ref{Thm 4.3}. As we mentioned before, Theorem \ref{Thm 4.3} implies Theorem \ref{main theorem}, which tells us the fiberwise convexity of Hill's lunar problem.

\appendix

\section{Appendix: Numerical proofs} \label{Numerical}

In the proof of Theorem \ref{Thm 4.3}, some proofs of inequalities are replaced by computer plots in order to simplify the argument. We will verify the Figures in Appendix \ref{pf of fig}. As we mentioned, we will prove inequality (\ref{final ineq}) in Appendix \ref{sec pf of ineq}. These proofs will be done by a computer program. The author want to emphasize that there has been lots of advice and help for this Appendix from Otto van Koert and referees.

\subsection{The proofs of figures} \label{pf of fig}

In this section, we will verify the graphs(Figure \ref{fig 2}, \ref{fig 5}, \ref{fig 9}, \ref{fig 10}, \ref{fig 11}, \ref{fig 12}, \ref{fig 13}, \ref{fig 14}, \ref{fig 15}, \ref{fig 16} and \ref{fig 17}) which we used to show inequalities. We rely on the computer program. Let $f: I=[a, b] \rightarrow \mathbb{R}$ be a function which we want to verify the inequality $f(x)>0$ or $f(x)<0$ for all $x \in I$.
\begin{itemize}
\item Replace the inequality by $g(x)=\pm x^n f(x)>0$.
\item Pick a small $\epsilon>0$.
\item Compute $m:=\min \{g(x)| x=a+k \epsilon, k \in \mathbb{N}, x<b\}$, in practice, its lower bound.
\item Derive an upper bound, say $B$, for $\max_{x \in I} \left| {dg \over dx}(x) \right|$.
\item Show that $m>\epsilon B$.
\end{itemize}
First three steps can be done with the following simple python program.
\begin{Verbatim}
init=a
final=b
epsilon=0.000001

def g(x):
    return "function we are interested in"
    
min=g(init)
temp=init

while temp<final:
    if min>g(temp):
        min=g(temp)
    temp+=epsilon

print(min)
\end{Verbatim}
Here, we take $\epsilon=10^{-6}$ for every $g$'s in Figure \ref{fig 2}, \ref{fig 5}, \ref{fig 9}, \ref{fig 10}, \ref{fig 11}, \ref{fig 12}, \ref{fig 13}, \ref{fig 14}, \ref{fig 15}, \ref{fig 16} and \ref{fig 17}.

\begin{center}
\begin{tabular}{ |c|c|c|c|c| } 
 \hline
Figure  & $g$ & $I$ & $m$ & $B$ \\ 
 \hline
 Fig. \ref{fig 2} & $x^7 f_2(x) $ & $[0, 0.54]$ & 0.3524 & $4 \times 10^4 $\\ 
 Fig. \ref{fig 5} & $x^7 f_5(x) $ & $[0, 0.54]$ & 0.0453 & $4 \times 10^4 $\\
 Fig. \ref{fig 9} & $x^6 f_9(x) $ & $[0.54, 3^{-1 \over 3})$ & 0.2461 & $4 \times 10^4 $\\
 Fig. \ref{fig 10} & $-x^{13} f_{10}(x) $ & $[0.56, 3^{-1 \over 3})$ & 1.8777 & $4 \times 10^4 $\\
 Fig. \ref{fig 11} & $-x^7 f_{11}(x) $ & $[0.54, 0.56]$ & 1.2197 & $4 \times 10^4 $\\
 Fig. \ref{fig 12} & $x^7 f_{12}(x) $ & $[0.54, 0.56]$ & 2.7452 & $4 \times 10^4 $\\
 Fig. \ref{fig 13} & $x^7 f^m_{13}(x), \quad x^7 f^M_{13}(x)$ & $[0.54, 0.63]$ & 2.6154, 2.9192 & $4 \times 10^4 $\\
 Fig. \ref{fig 14} & $x^7 f^m_{14}(x), \quad x^7 f^M_{14}(x)$ & $[0.54, 0.63]$ & 0.5905, 1.5023 & $4 \times 10^4 $\\
 Fig. \ref{fig 15} & $x^7 f^m_{15}(x), \quad x^7 f^M_{15}(x)$ & $[0.54, 0.63]$ & 0.5395, 0.7966 & $4 \times 10^4 $\\
 Fig. \ref{fig 16} & $x^7 f^m_{16}(x), \quad x^7 f^M_{16}(x)$ & $[0.54, 0.63]$ & 0.9569, 1.1176 & $4 \times 10^4 $\\
 Fig. \ref{fig 17} &$x^7 f^m_{17}(x), \quad x^7 f^M_{17}(x)$  & $[0.54, 0.63]$ & 0.8420, 1.5383 & $4 \times 10^4 $\\
 \hline
\end{tabular}
\end{center}

The derivative bound $B$ can be shown with the following argument. Note that every $g$ has no negative degree and consists of rationals and square roots which have no pole on each interval. We can easily see that the coefficient of $g$ is less than 40 and the highest degree is less than 10. Thus the absolute value of the coefficient of ${dg \over dx}$ is less than 400. Moreover, it is clear that each ${dg \over dx}$ has at most 100 terms. Therefore, we have that
$$\left| {dg\over dx}(x) \right|<400 \times 100=4\times 10^4.$$
Clearly, the property $\epsilon B <m$ holds for every case. This proves inequalities
$$g_i(x)>0 \textrm{ on } I_i \quad \textrm{ for all }  i=2,5,9,10,11, 12, 13, 14, 15, 16, 17.$$

\subsection{Proof of inequality (\ref{final ineq})} \label{sec pf of ineq}

We prove inequality (\ref{final ineq}) using a computer program. The strategy is basically same with Appendix \ref{pf of fig}. We will use a python program for finding the minimum on the lattice. We will use Maple program to obtain derivative bounds. This can be used to get a derivative bound in general case. Thus, we give the coding at the end of Appendix \ref{sec pf of ineq}. Let us explain the idea of the proof.
\begin{itemize}

\item $F(x, u):={{\partial C_1} \over {\partial x}}(x, k(u))-{{\partial C_2} \over {\partial x}}(x, k(u))-{{\partial C_3} \over {\partial x}}(x, k(u))+{{\partial C_4} \over {\partial x}}(x, k(u))$ where $k(u)=-2u^3+3u^2$. The new variable $u$ resolves the singularities of derivatives.

\item Pick a small $\epsilon_x>0$ and $\epsilon_u>0$.

\item Compute $M:=\max \{F(x, u)| x=0.63+m \epsilon_x<3^{-1 \over 3}, u=0+n \epsilon_u<1, (m, n) \in \mathbb{N}\times \mathbb{N}\}$, in practice, its upper bound. Check $M<0$.

\item Derive an upper bounds $B_x>\max_{(x, u) \in [0.63, 3^{-1 \over 3}] \times [0, 1]} \left| {\partial F \over \partial u}(x, u) \right|$ and
\\ $B_u>\max_{(x, u) \in [0.63, 3^{-1 \over 3}] \times [0, 1]} \left| {\partial F \over \partial x}(x, u) \right|$.

\item Show that $|M|>\epsilon_x B_x$ and $|M|>\epsilon_u B_u$.

\end{itemize}

Motivated by the derivative bounds we will derive in the following, we take the $\epsilon_x={1 \over 2.05 \times 10^5}$ and $\epsilon_u={1 \over 4.59 \times 10^4}$. The next step is can be done with the following simple code(pseudo code).

\begin{Verbatim}
init_x, final_x=0.63, 3^(-1/3)
init_u, final_u=0, 1
epsilon_x, epsilon_u= "As above"

F(x, u)="The function we want to get the maximum on the lattice."
    
Max=F(init_x, init_u)
temp_x, temp_u=init_x, init_u

while temp_x<final_x:
	while temp_u<final_u:
		if Max<F(temp_x, temp_u):
			Max=F(temp_x, temp_u)
		temp_u+=epsilon_u
	temp_x+=epsilon_x

print(Max)
\end{Verbatim}

As a result, we can get $M<-19$. We explain the procedure to get the derivative bounds $B_x$ and $B_u$. We divide the functions composing the functions $C_1, C_2, C_3$ and $C_4$ into two classes of functions, say "abstract functions", "rational functions". Define abstract functions
$$a_0(x, u)={1 \over \sqrt{1+k(u)(3x^3-1)}}, \quad a_1(u)=\sqrt{1-k(u)}, \quad a_2(u)=\sqrt{k(u)},$$
$$a_3(x)=\sqrt{3+{2 \cdot 3^{2 \over 3} \over x}}, \quad a_4(x)={9 \over x^{3 \over 2}}-3x^{3 \over 2},$$
and rational functions
$$r_0(x)={15 \cdot 3^{2\over3} \over x^7}-{27 \over x^6}-{18 \cdot 3^{1 \over 3} \over x^5}+{5 \cdot 3^{2 \over 3} \over x^4}+{12 \over x^3}+{9 \cdot 3^{1\over 3} \over x^2}, \quad r_1(x)=-{13 \over x^6}+{18\sqrt[3]{3} \over x^5}-{8 \over x^3}+3,$$
$$r_2(x)={3^{1\over3} \over x^5}+{3^{2 \over 3} \over x^4}+{3 \over x^3}+{3^{4\over3} \over x^2}, \quad r_3(x)=1-{3^{4 \over 3} \over x^2}.$$
Note that $r_i$'s do not have a positive degree term. Then we have that
$$C_1(x, u)=r_0(x)+a_0^2(x, u)a_1^2(u)a_3^2(x)r_1(x), \quad C_2(x, u)=2a_0^2(x, u)a_1(u)a_2(u)a_3^2(x)a_4(x),$$
$$C_3(x, u)=2a_0(x, u)a_1(u)a_3(x)r_2(x), \quad C_4(x, u)=a_0^2(x, u)a_1^2(u)a_3^2(x)r_3(x).$$
When we take derivatives, abstract functions are formally differentiated. Since $F(x, u)={{\partial C_1} \over {\partial x}}(x, u)-{{\partial C_2} \over {\partial x}}(x, u)-{{\partial C_3} \over {\partial x}}(x, u)+{{\partial C_4} \over {\partial x}}(x, u)$, the following functions
$$b_1=a_0, \quad b_2=\partial_x a_0, \quad b_3=\partial_u a_0, \quad b_4=\partial_{xu} a_0,\quad b_5=\partial_{xx} a_0,$$
$$b_6=a_1, \quad b_7=\partial_u a_1, \quad \quad b_8=a_2, \quad b_9=\partial_u a_2,$$
$$b_{10}=a_3,\quad  b_{11}=\partial_x a_3, \quad b_{12}=\partial_{xx} a_3, \quad \quad b_{13}=a_4, \quad b_{14}=\partial_x a_4,\quad b_{15}=\partial_{xx} a_4$$
will appear in the expansions of ${\partial F \over \partial x}$ and ${\partial F \over \partial u}$. Then we have the following expressions
\begin{equation} \label{lin comb}
{\partial F \over \partial x}=\sum_{I \subset J}{c_I^x \cdot b_I}, \quad {\partial F \over \partial u}=\sum_{I \subset J}{c_I^u \cdot b_I}
\end{equation}
where $J=\{1, 2, \cdots ,15\}$, $b_I=\prod_{i \in I} {b_i}$ and $c_I^x, c_I^u$ are monomials of non-positive degree. Note that $b_I$ can be repeated in order to have monomial coefficients. We can easily see that the inequalities
$$|a_0(x, u)|<1.2, \quad |\partial_x a_0(x, u)|<3.4, \quad |\partial_u a_0(x, u)|<0.3, \quad |\partial_{xx} a_0|<15, \quad |\partial_{xu} a_0(x, y)|<1.7,$$
$$|a_1(u)| \le 1, \quad |\partial_u a_1(u)| \le \sqrt{3}, \quad |a_2(u)| \le 1, \quad |\partial_u a_2(u)| \le \sqrt{3},$$
$$|a_3(x)| \le \sqrt{3+{2\cdot3^{2 \over 3}\over 0.63}},\quad |\partial_x a_3(x)|<1.7, \quad, |\partial_{xx} a_3(x)|<5,$$ 
$$|a_4(x)|<16.5,\quad |\partial_x a_4(x)|<46.5, \quad, |\partial_{xx} a_4(x)|<168.$$
Using triangle inequalities, we can obtain upper bounds 
$$\left|{\partial F \over \partial x}\right| \le \sum_{I \subset J}{|c_I^x \cdot b_I|}, \quad \left|{\partial F \over \partial u}\right| \le \sum_{I \subset J}{|c_I^u \cdot b_I|}.$$
We replace the abstract functions and their derivatives by their individual upper bounds obtained above. In addition we substitute $x=0.63$ into $|c_I^x|$ and $|c_I^u|$ to get upper bounds
$$\left|{\partial F \over \partial x}\right| \le B_x \cong 2.040754753 \times 10^6, \quad \left|{\partial F \over \partial u}\right| \le B_u\cong 4.580163896 \times 10^5.$$
This proves inequality (\ref{final ineq}). We leave the Maple code of this procedure for the bound of $\left|{\partial F \over \partial u}\right|$.

\begin{Verbatim}
Declare the "rational functions"  r_0(x), r_1(x), r_2(x) and r_3(x) as above:
# We do not specify the "abstract functions".
Declare C_1(x, u), C_2(x, u), C_3(x, u) and C_4(x, u) as above:

F := expand(diff(C_1, x)-(diff(C_2, x))-(diff(C_3, x))+diff(C_4, x)); 
DuF := expand(diff(F, u)); 
boundDuF := 0: 

for i to nops(DuF) do 
	tmp := subs(diff(diff(a_0(x, u), u), x) = BDxua_0, op(i, DuF)): 
	tmp := subs(diff(diff(a_0(x, u), x), x) = BDxxa_0, tmp): 
	tmp := subs(diff(a_0(x, u), u) = BDua_0, tmp): 
	tmp := subs(diff(a_0(x, u), x) = BDxa_0, tmp): 
	tmp := subs(diff(a_1(u), u) = BDua_1, tmp): 
	tmp := subs(diff(a_2(u), u) = BDua_2, tmp): 
	tmp := subs(diff(a_3(x), x) = BDxa_3, tmp):
	tmp := subs(diff(a_4(x), x) = BDxa_4, tmp):
	
	tmp := subs(a_0(x, u) = Ba_0, tmp):
	tmp := subs(a_1(u) = Ba_1, tmp):
	tmp := subs(a_2(u) = Ba_2, tmp):
	tmp := subs(a_3(x) = Ba_3, tmp):
	tmp := subs(a_4(x) = Ba_4, tmp):
	
	boundDuF := boundDuF+abs(tmp):
end do:

xmin := 0.63: 
Ba_0 := 1.2: BDxa_0 := 3.4: BDua_0 := .3: BDxxa_0 := 15: BDxua_0 := 1.7: 
Ba_1 := 1: BDua_1 := sqrt(3): Ba_2 := 1: BDua_2 := sqrt(3): 
Ba_3 := sqrt(3+2*3^(2/3)/xmin): BDxa_3 := 1.8: BDxxa_3 := 5:
Ba_4 := 16.5: BDxa_4 := 46.5: BDxxa_4 := 168: 

evalf(subs(x = xmin, boundDuF));

				4.580163896 10^5
\end{Verbatim}

Using the similar code, derivative bound 
$$2.040754753 \times 10^6$$
for $\left| {\partial F \over \partial x}\right|$ can be obtained analogously. This completes the proof of inequality (\ref{final ineq}).

\end{document}